\documentclass[12pt]{article}

\pdfoutput=1
\usepackage{amsfonts}
\usepackage{amsthm}
\usepackage[hcentering,vcentering]{geometry}  
\usepackage{graphics}
\usepackage{epsfig}
\usepackage{mathptmx}
\usepackage{times}
\usepackage{amsmath}
\usepackage{amssymb}
\usepackage{pst-all}
\usepackage{relsize}
\usepackage{cite}

\newtheorem{thm}{Theorem}[section]
\newtheorem{definition}[thm]{Definition}
\newtheorem{example}[thm]{Example}

\newtheorem{problem}[thm]{Problem}
\newtheorem{assumption}[thm]{Assumption}

\newtheorem{theorem}[thm]{Theorem}
\newtheorem{lemma}[thm]{Lemma}
\newtheorem{corollary}[thm]{Corollary}
\newtheorem{proposition}[thm]{Proposition}
\newtheorem{remark}[thm]{Remark}

\def\R{\mathbb{R}}
\def\N{\mathbb{N}}
\def\KK{{\cal K}}

\def\KL{{\cal KL}}
\def\UU{{\cal U}}
\def\wmu{{\tilde{\mu}}}
\def\wV{\widetilde{V}}

\begin{document}

\title{Analysis of unconstrained nonlinear MPC schemes with time varying control horizon\footnote{This work was supported by the DFG priority program 1305, Grant Gr1569/12-1.}}

\author{Lars Gr\"{u}ne\footnote{Lars Gr\"{u}ne, J\"{u}rgen Pannek, Martin Seehafer and Karl Worthmann are with the Mathematical Institute, University of Bayreuth, 95440 Bayreuth, Germany. 
	\newline {\tt lars.gruene@uni-bayreuth.de}, {\tt juergen.pannek@uni-bayreuth.de},
	\newline {\tt martin.seehafer@uni-bayreuth.de}, {\tt karl.worthmann@uni-bayreuth.de}}
	\and J\"{u}rgen Pannek \and Martin Seehafer \and Karl Worthmann}

\date{June 2010}

\maketitle

\begin{abstract}
For discrete time nonlinear systems
satisfying an exponential or finite time controllability assumption,
we present an analytical formula for a suboptimality estimate for
model predictive control schemes without stabilizing
terminal constraints. Based on our formula, we perform a detailed
analysis of the impact of the optimization horizon and the possibly
time varying control horizon on stability and performance of the
closed loop. 

\medskip

\end{abstract}

\section{Introduction}
The stability and performance analysis of model predictive control (MPC) schemes has attracted considerable attention during the last years. MPC relies on an iterative online solution of finite horizon optimal control problems in order to deal with an optimal control problem on an infinite horizon. To this end, a performance criterion --- often the distance to some desired reference --- is optimized over the predicted trajectories of the system. This method is particularly attractive due to its ability to explicitly incorporate constraints in the controller design. Due to the rapid development of efficient optimization algorithms MPC becomes increasingly applicable also to nonlinear and large scale systems. 

Two central questions in the analysis of MPC schemes are asymptotic stability, i.e., whether the closed loop system trajectories converge to the reference and stay close to it, and closed loop performance of the MPC controlled system. In particular -- since desired performance specifications (like, e.g., minimizing energy or maximizing the output in a chemical process) can be explicitly included in the optimization objective -- the latter provides information on how good this objective is eventually satisfied by the resulting closed loop system. For MPC schemes with stabilizing terminal constraints the available analysis methods have reached a certain degree of maturity, see, e.g., the survey \cite{MaRRS00} and the references therin. Despite their widespread use in applications, cf.\ \cite{QB2003}, for schemes without stabilizing terminal constraints --- considered in this paper --- corresponding results are more recent and less elaborated. Concerning stability, the papers \cite{AlaB95,GMTT05,JadH05} show (under different types of controllability or detectability conditions) that stability can be expected if the optimization horizon is chosen sufficiently large, without, however, aiming at giving precise estimates for these horizons.

Closed loop performance of MPC controlled systems is measured by evaluating an infinite horizon functional along the closed loop trajectory. Suboptimality estimates, which typically allow to conclude stability of the closed loop, are then obtained by comparing this value with the optimal value of the infinite horizon problem. In \cite{ShaX97} an estimation method of this type for discrete time linear systems is presented which relies on a numerical approximation of the finite time optimal value function. Since for nonlinear or large scale systems this function is usually not computable, in \cite{GruR08} a method for finite or infinite dimensional discrete time nonlinear systems using ideas from relaxed dynamic programming has been presented. This approach allows for performance estimates based on controllability properties. Motivated by these results, in \cite{G07} a linear program has been developed whose solution precisely estimates the degree of suboptimality from exponential or finite time controllability.

The present paper builds upon \cite{G07} extending the analysis from this reference to MPC schemes with time varying {\em control horizon}, i.e., the interval between two consecutive optimizations or, equivalently, the interval on which each resulting open loop optimal control is applied. This setting is motivated by networked control systems in which the network performance determines the control horizon, see \cite{GPW08,GPW09} and the discussion after Remark \ref{minrem}, below. In particular, we thoroughly investigate the impact of different --- possibly time varying --- control horizons on the closed loop behavior. 

Moreover, we give an analytic solution to the linear program from \cite{G07} and -- as a consequence -- an explicit formula for the suboptimality estimate based on the $\KL_0$-function characterizing our controllability assumption. This allows for a much more detailed analysis which is the main contribution of this paper. We investigate -- among others -- the impact of the {\em optimization horizon}, i.e., the interval on which the predicted trajectory is optimized (and which we choose identical to the prediction horizon), on the suboptimality and stability of the MPC closed loop. Especially, we prove conjectures from \cite{G07} with respect to minimal stabilizing horizons which were based on numerical observations. Furthermore, we analyze the influence of adding a final weight in the finite horizon cost functional.

The paper is organized as follows. In Section \ref{setupsec} we describe the setup and problem formulation. In Section \ref{controlsec} we introduce our controllability assumption and briefly summarize the needed results from \cite{G07}. In Section \ref{stabsec} we show that our suboptimality result can be used to conclude stability, extending \cite[Section 5]{G07} to time varying control horizons. In Section \ref{alphasec} we present the explicit formula for our suboptimality index $\alpha$ in Theorem \ref{alpha_formula_thm}. In the ensuing sections we examine effects of different parameters on $\alpha$. In particular, in Section \ref{alphasec1} we investigate the impact of the optimization horizon and in Sections \ref{alphasec2} and \ref{alphasec3} we scrutinize qualitative and quantitative effects, respectively, of different control horizons. Finally, in Section \ref{exsec} we illustrate our results with numerical examples. A number of technical lemmata and their proofs can be found in the appendix in Section \ref{appendix}.

\section{Setup and Preliminaries}\label{setupsec}
We consider a nonlinear discrete time control system given by
\begin{equation}
	\label{2:eq:discrete control system}
	x(n + 1) = f(x(n), u(n)), \quad x(0) = x_0
\end{equation}
with $x(n) \in X$ and $u(n) \in U$ for $n \in \N_0$. Here the state space $X$ 
and the control value space $U$ are arbitrary metric spaces. We denote
the space of control sequences $u: \N_0 \rightarrow U$ by $\UU$ and
the solution trajectory for given $u \in \UU$ by $x_u(\cdot)$. Note that
constraints can be incorporated by replacing $X$ and $U$ by appropriate
subsets of the respective spaces. For simplicity of exposition,
however, we will not address feasibility issues in this paper. 

A typical class of such discrete time systems are sampled--data systems induced by a controlled --- finite or infinite dimensional --- differential equation with sampling period $T>0$. In this situation, the discrete time $n$ corresponds to the continuous time $t = n T$.

Our goal is to minimize the infinite horizon cost $J_\infty(x_0, u) = \sum_{n = 0}^\infty l(x_u(n), u(n))$ with running cost $l: X \times U \rightarrow \R_0^+$ by a multistep state feedback control (rigorously defined below in Definition \ref{2:def:m_i feedback law}). We denote the optimal value function for this problem by $V_\infty(x_0) := \inf_{u \in \UU} J_\infty(x_0,u)$. Since infinite horizon optimal control problems are in general computationally infeasible, we use a receding horizon approach in order to compute an approximately optimal controller. To this end, we consider the finite horizon functional
\begin{equation} \label{2:eq:finite cost functional}
	J_N(x_0, u) = \sum_{n = 0}^{N - 1} l(x_u(n), u(n))
\end{equation}
with {\em optimization horizon} $N\in\N$ inducing the optimal value function
\begin{equation} \label{2:eq:finite value function}
	V_N(x_0) = \inf_{u\in\UU} J_N(x_0,u).
\end{equation}
By solving this finite horizon optimal control problem we obtain $N$ control values $\mu(x_0,0),\linebreak[0] \mu(x_0,1),\linebreak[0]  \ldots,\linebreak[0]  \mu(x_0,N-1)$ depending on the state $x_0$. Implementing the first $m_0 \in \{1,\ldots,N-1\}$ elements of this sequence yields a new state $x(m_0)$. Iterative application of this construction then provides a control sequence on the infinite time interval, whose properties we intend to investigate in this paper. To this end, we introduce a more formal description of this construction. 
\begin{definition}
	Given a set $M\subseteq\{1,\ldots,m^\star\}$, $m^\star \in \mathbb{N}$, we call a control horizon sequence $(m_i)_{i \in \N_0}$ {\em admissible} if $m_i\in M$ holds for all $i\in \N_0$. Furthermore, for $k, n\in\N_0$ we define
	\begin{eqnarray*}
		\sigma(k) & := & \sum_{j=0}^{k-1} m_i \quad \mbox{(using the convention $\sum_{j=0}^{-1}=0$)}\\ \varphi(n) & := & \max \{\sigma(k)\,|\, k\in\N_0, \sigma(k)\le n\}.
	\end{eqnarray*}
\end{definition}

Using this notation, the applied control sequence can be expressed as
\[
\ldots, \mu(x(\sigma(k)),0), \ldots, \mu(x(\sigma(k)),m_k-1), \mu(x(\sigma(k+1)),0), \ldots
\] 
A closed loop interpretation of this construction can be obtained via multistep feedback laws. 
\begin{definition} \label{2:def:m_i feedback law}
        For $m^\star \geq 1$ and $M\subseteq \{1,\ldots,m^\star\}$ a multistep feedback law is a map $\mu : X \times \{ 0,\linebreak[0]  \ldots, \linebreak[0] m^\star - 1 \} \rightarrow U$ which for an admissible control horizon sequence $(m_i)_{i \in \N_0}$ is applied according to the rule $x_\mu(0)=x_0$, 
	\begin{equation} \label{2:def:m_i feedback law:eq1}
		x_\mu(n + 1) = f(x_\mu(n), \mu(x_\mu ( \varphi(n) ), n - \varphi(n))).
	\end{equation}
\end{definition}
Using this definition, the above construction is equivalent to the following definition.
\begin{definition} \label{2:def:mpc m-step feedback law}
        For $m^\star\ge 1$ and $N\ge m^\star+1$ we define the multistep MPC feedback law $\mu_{N, m^\star}(x_0,n) := u^\star(n)$, where $u^\star$ is a minimizing control for \eqref{2:eq:finite value function} with initial value $x_0$. 
\end{definition}
\begin{remark}
	For simplicity of exposition here we assume that the infimum in \eqref{2:eq:finite value function} is a minimum, i.e., that a minimizing control sequence $u^*$ exists.
\label{minrem}\end{remark}

Note that in ``classical'' MPC only the first element of the obtained finite horizon optimal sequence of control values is used. Our main motivation for considering this generalized feedback concept with varying control horizons $m_i$ are networked control systems (NCS) in which the transmission channel from the controller to the plant is subject to packet dropouts. In order to compensate these dropouts, at each successful transmission time $\sigma(k)$ a whole sequence of control values is transmitted to the plant. This sequence is then used until the next successful transmission at time $\sigma(k+1) = \sigma(k)+m_k$, for details see \cite{GPW08}. Note that in this application the control horizon $m_k$ is not known at time $\sigma(k)$.

In this paper we consider the conceptually simplest MPC approach imposing neither terminal costs nor terminal constraints. 
In order to measure the suboptimality degree of the multistep feedback for the infinite horizon problem we define 
\begin{equation*}
	V_\infty^{\mu,(m_i)}(x_0) := \sum_{n = 0}^\infty l(x_{\mu}(n), \mu(x_\mu(\varphi(n)), n - \varphi(n))).
\end{equation*}
Our approach relies on the following result from relaxed dynamic programming \cite{LinR06,Rant06}, which is a straightforward generalization of proposition \cite[Proposition 2.4]{G07}, cf.\ \cite{GPW08} for a proof. 
\begin{proposition} \label{2:prop:m-step suboptimality estimate}
	Consider a multistep feedback law $\wmu: X \times \{ 0, \ldots, m^\star - 1\} \rightarrow U$, a set $M\subseteq\{1,\linebreak[0] \ldots,\linebreak[0] m^\star\}$ and a function $\wV: X \rightarrow \R_0^+$ and assume that for each admissible control horizon sequence $(m_i)_{i\in\N_0}$ and each $x_0\in X$ the corresponding solution $x_\wmu(n)$ with $x_\wmu(0) = x_0$ satisfies 
	\begin{equation}
		\label{2:prop:m-step suboptimality estimate:eq1}
		\wV(x_0) \geq \wV(x_\wmu(m_0)) + \alpha \sum_{k = 0}^{m_0 - 1} l(x_\wmu(k), \wmu(x_0, k))
	\end{equation}
	for some $\alpha\in(0,1]$. Then for all $x_0 \in X$ and all admissible $(m_i)_{i\in\N_0}$ the estimate $\alpha V_\infty(x_0) \leq \linebreak[0] \alpha \linebreak[0] V_\infty^{\wmu,(m_i)}(x_0)\linebreak[0]  \leq \linebreak[0] \wV(x_0)$ holds.
\end{proposition}

\section{Controllability and performance bounds}\label{controlsec}
In this section we introduce an asymptotic controllability assumption and deduce several consequences for our optimal control problem. In order to facilitate this relation we will formulate our basic controllability assumption, below, not in terms of the trajectory but in terms of the running cost $l$ along a trajectory.\\
To this end, we say that a continuous function $\rho: \R_{\geq 0} \rightarrow \R_{\geq 0}$ is of class $\KK_\infty$ if it satisfies $\rho(0) = 0$, is strictly increasing and unbounded. Furthermore, we say that a continuous function $\beta: \R_{\geq 0} \times \R_{\geq 0} \rightarrow \R_{\geq 0}$ is of class $\KL_0$ if for each $r > 0$ we have $\lim_{t \rightarrow \infty} \beta(r, t) = 0$ and for each $t \geq 0$ we either have $\beta(\cdot, t) \in \KK_\infty$ or $\beta(\cdot, t) \equiv 0$. Note that in order to allow for tighter bounds for the actual controllability behavior of the system we use a larger class than the usual class $\KL$. It is, however, easy to see that each $\beta \in \KL_0$ can be overbounded by a $\tilde{\beta} \in \KL$, e.g., by setting $\tilde \beta(r, t) = \sup_{\tau \geq t} \beta(r, \tau) + e^{-t} r$. Moreover, we define $l^\star(x) := \min_{u \in U} l(x, u)$.

\begin{assumption} \label{3:ass:controllability}
	Given a function $\beta \in \KL_0$, for each $x_0 \in X$ there exists a control function $u_{x_0} \in \UU$ satisfying $l(x_{u_{x_0}}(n), u_{x_0}(n)) \leq \beta(l^\star(x_0), n)$ for all $n \in \N_0$.
\end{assumption}

Special cases for $\beta \in \KL_0$ are
\begin{equation}
	\label{3:eq:exponential controllability}
	\beta(r, n) = C \sigma^n r
\end{equation}
for real constants $C \geq 1$ and $\sigma \in (0, 1)$, i.e., {\em exponential controllability}, and
\begin{equation}
	\label{3:eq:finite time controllability}
	\beta(r,n) = c_n r
\end{equation}
for some real sequence $(c_n)_{n \in \N_0}$ with $c_n \geq 0$ and $c_n = 0$ for all $n \geq n_0$, i.e., {\em finite time controllability} (with linear overshoot).\\
For certain results it will be useful to have the property
\begin{equation}
	\label{3:eq:submultiplicativity}
	\beta(r, n + m) \leq \beta( \beta(r, n), m) \quad \mbox{for all } r \geq 0, n, m \in \N_0.
\end{equation}
Property \eqref{3:eq:submultiplicativity} ensures that any sequence of the form $\lambda_n = \beta(r, n)$, $r > 0$, also fulfills $\lambda_{n + m} \leq \beta(\lambda_n, m)$. It is, for instance, always satisfied in case \eqref{3:eq:exponential controllability} and satisfied in case \eqref{3:eq:finite time controllability} if and only if $c_{n + m} \leq c_n c_m$. If needed, this property can be assumed without loss of generality, cf. \cite[Section 3]{G07}.

In order to ease notation, we define the value
\begin{equation}
	\label{3:eq:definition B_N}
	B_N(r) := \sum_{n = 0}^{N - 1} \beta(r,n).
\end{equation}
for any $r \geq 0$ and any $N \in \mathbb{N}_{\geq 1}$ . An immediate
consequence of Assumption \ref{3:ass:controllability} and Bellman's
optimality principle $V_N(x) = \min_{u\in U}\{l(x,u) +
V_{N-1}(f(x,u))\}$ are the following lemmata from \cite{G07}.

\begin{lemma}\label{VNboundlemma}
	For each $N \geq 1$ the inequality
	\begin{equation}
		\label{cont_bound}
		V_N(x_0) \leq B_N(l^\star(x_0))
	\end{equation}
	holds.
\end{lemma}

\begin{lemma}\label{lemma}
	Suppose Assumption \ref{3:ass:controllability} holds and consider $x_0 \in X$ and an optimal control $u^\star$ for the finite horizon optimal control problem \eqref{2:eq:finite value function} with optimization horizon $N \geq 1$. Then for each $j = 0, \ldots, N - 2$ the inequality
	\begin{equation}
		J_{N-j}(x_{u^\star}(j), u^\star(j + \cdot)) \leq B_{N - j}(l^\star(x_{u^\star}(j))
	\end{equation}
	and for each $m = 1, \ldots, N - 1$ and each $j = 0, \ldots, N - m - 1$ the inequality
	\begin{equation}
		V_{N}(x_{u^\star}(m)) \leq J_j(x_{u^\star}(m), u^\star(m + \cdot)) + B_{N - j} (l^\star (x_{u^\star}(m + j)))
	\end{equation}
	holds for $B_{N - j}$ from \eqref{3:eq:definition B_N}. 
\end{lemma}

Now we provide a constructive approach in order to compute $\alpha$ in
\eqref{2:prop:m-step suboptimality estimate:eq1} for systems
satisfying Assumption \ref{3:ass:controllability}. Note that
\eqref{2:prop:m-step suboptimality estimate:eq1} only depends on $m_0$
and not on the remainder of the control horizon sequence. Hence, we
can perform the computation separately for each control horizon $m$ and
obtain the desired $\alpha$ for variable $m$ by minimizing over the
$\alpha$-values for all admissible $m$. 

For our computational approach we consider arbitrary values
$\lambda_0, \ldots, \lambda_{N - 1} > 0$ and $\nu > 0$ and start by
deriving necessary conditions under which these values coincide with
an optimal sequence $l(x_{u^\star}(n), u^\star(n))$ and an optimal
value $V_N(x_{u^\star}(m))$, respectively. 

\begin{proposition}\label{4:prop:necessary conditions}
	Suppose Assumption \ref{3:ass:controllability} holds and consider $N
        \geq 1$, $m \in \{ 1, \ldots, N - 1 \}$, a sequence $\lambda_n
        > 0$, $n = 0, \ldots, N - 1$, and a value $\nu > 0$. Consider $x_0
        \in X$ and assume that there exists a minimizing control
        $u^\star \in \UU$ for \eqref{2:eq:finite value function} such
        that $\lambda_n$ equals $l(x_{u^\star}(n), u^\star(n))$ for all
        $n \in \{0, \ldots, N - 1\}$. Then 
	\begin{equation}
		\label{4:prop:necessary conditions:eq1}
		\sum_{n = k}^{N - 1} \lambda_n \leq B_{N - k}(\lambda_k), \quad	k = 0, \ldots, N - 2
	\end{equation}
	holds true and if furthermore $\nu = V_N(x_{u^\star}(m))$ we have
	\begin{equation}
		\label{4:prop:necessary conditions:eq2}
		\nu \leq \sum_{n = 0}^{j - 1} \lambda_{n + m} + B_{N - j}(\lambda_{j + m}), \quad j = 0, \ldots, N - m - 1.
	\end{equation}
\end{proposition}
\begin{proof}
	If the stated conditions hold, then $\lambda_n$ and $\nu$ meet the inequalities given in Lemma \ref{lemma}, which is exactly \eqref{4:prop:necessary conditions:eq1} and \eqref{4:prop:necessary conditions:eq2}. 
\end{proof}

Using this proposition a sufficient condition for suboptimality of the
MPC feedback law $\mu_{N,m}$ is given in Theorem
\ref{4:thm:optimality} which is proved in \cite{G07}.

\begin{theorem} \label{4:thm:optimality}
	Consider $\beta \in \KL_0$, $N \geq 1$, $m \in \{ 1, \ldots, N - 1 \}$, and assume that all sequences $\lambda_n > 0$, $n = 0, \ldots, N - 1$ and values $\nu > 0$ fulfilling \eqref{4:prop:necessary conditions:eq1}, \eqref{4:prop:necessary conditions:eq2} satisfy the inequality 
	\begin{equation}
		\label{4:thm:optimality:eq1}
		\sum_{n = 0}^{N - 1} \lambda_n - \nu \geq \alpha \sum_{n = 0}^{m - 1} \lambda_n
	\end{equation}
	for some $\alpha\in(0,1]$. Then for each optimal control problem \eqref{2:eq:discrete control system}, \eqref{2:eq:finite value function} satisfying Assumption \ref{3:ass:controllability} the assumptions of Proposition \ref{2:prop:m-step suboptimality estimate} are satisfied for the multistep MPC feedback law $\mu_{N,m}$ and in particular the inequality
$\alpha V_\infty(x) \leq \alpha V_\infty^{\mu_{N, m}}(x) \leq V_N(x)$
	holds for all $x \in X$.
\end{theorem}

In view of Theorem \ref{4:thm:optimality}, the value $\alpha$ can be interpreted as a performance bound which indicates how good the receding horizon MPC strategy approximates the infinite horizon problem. In the remainder of this section we present an optimization based approach for computing $\alpha$. To this end, consider the following optimization problem.

\begin{problem}\label{optprob}
	Given $\beta \in \KL_0$, $N \geq 1$ and $m \in \{ 1, \ldots, N - 1 \}$, compute 
	\begin{equation*}
		\alpha_{N,m}^1 := \inf_{\lambda_0, \ldots, \lambda_{N - 1}, \nu} \frac{\sum_{n = 0}^{N - 1} \lambda_n - \nu}{\sum_{n = 0}^{m - 1} \lambda_n}
	\end{equation*}
	subject to the constraints \eqref{4:prop:necessary conditions:eq1}, \eqref{4:prop:necessary conditions:eq2}, and $\lambda_0, \ldots, \lambda_{N - 1}, \nu > 0$.
\end{problem}

The following is a straightforward corollary from Theorem \ref{4:thm:optimality}.

\begin{corollary}\label{optcor}
	Consider $\beta \in \KL_0$, $N \geq 1$, $m \in \{ 1, \ldots, N
        - 1 \}$, and assume that the optimization problem
        \ref{optprob} has an optimal value $\alpha \in (0, 1]$. Then
        for each optimal control problem \eqref{2:eq:discrete control
          system}, \eqref{2:eq:finite value function} satisfying
        Assumption \ref{3:ass:controllability} the assumptions of
        Theorem \ref{4:thm:optimality} are satisfied and the
        assertions from Theorem \ref{4:thm:optimality} hold.
\end{corollary}

As already mentioned in \cite[Remark 4.3]{G07}, our setting can be easily extended to the setting including an additional weight $\omega \geq 1$ on the final term, i.e., altering our finite time cost functional by adding $(\omega - 1) l(x_u(N-1), u(N-1))$. Note that the original form of the functional $J_N$ is obtained by setting $\omega = 1$. All results in this section remain valid if the statements are suitably adapted. In particular, \eqref{2:eq:finite cost functional} and \eqref{3:eq:definition B_N} become
\begin{eqnarray}
	J_N(x_0, u) & := & \sum_{n = 0}^{N - 2} l(x_u(n), u(n)) + \omega l(x_u(N-1), u(N-1)) \nonumber \\
	\label{3:eq:definition B_N_omega}
	B_N(r) & := & \sum_{n = 0}^{N - 2} \beta(r,n) + \omega \beta(r,N-1).
\end{eqnarray}
and the formula in Problem \ref{optprob} alters to
\begin{equation}\label{opt_with_omega}
	\alpha_{N,m}^\omega := \inf_{\lambda_0, \ldots, \lambda_{N - 1}, \nu} \frac{\sum_{n = 0}^{N - 2} \lambda_n + \omega \lambda_{N-1}- \nu}{\sum_{n = 0}^{m - 1} \lambda_n}.
\end{equation}

\section{Asymptotic stability}\label{stabsec}
In this section, which extends \cite[Section 5]{G07} to varying control horizons, we show how the performance bound $\alpha=\alpha_{N,m}^\omega$ can be used in order to conclude asymptotic stability of the MPC closed loop. More precisely, we investigate the asymptotic stability of the zero set of
$l^\star$. To this end, we make the following assumption.

\begin{assumption}\label{A-ass}
	There exists a closed set $A \subset X$ satisfying:
	\begin{itemize}
		\item[(i)] For each $x \in A$ there exists $u \in U$ with $f(x, u) \in A$ and $l(x, u) = 0$, i.e., we can stay inside $A$ forever at zero cost.
		\item[(ii)] There exist $\KK_\infty$--functions	$\alpha_1$, $\alpha_2$ such that the inequality
		\begin{equation}
			\label{lbound}
			\alpha_1(\| x \|_A) \leq l^\star(x) \leq \alpha_2(\| x \|_A)
		\end{equation}
		holds for each $x \in X$ where $\| x \|_A := \min_{y \in A}\| x - y \|$. 
	\end{itemize}
\end{assumption}
This assumption assures global asymptotic stability of $A$ under the
optimal feedback for the infinite horizon problem, provided $\beta(r,
n)$ is summable. We remark that condition (ii) can be relaxed in
various ways, e.g., it could be replaced by a detectability condition
similar to the one used in \cite{GMTT05}. However, in order to keep
the presentation in this paper technically simple we will work with
Assumption \ref{A-ass}(ii) here. Our first stability result is
formulated in the following theorem. Here we say that a multistep
feedback law $\mu$ asymptotically stabilizes a set $A$ if there exists
$\tilde\beta\in\KL_0$ such that for all admissible control horizon
sequences the closed 
loop system satisfies $\| x_\mu(n) \|_A \leq \tilde{\beta}(\| x_0
\|_A, n)$. 

\begin{theorem} \label{stabthm}
Consider $\beta \in \KL_0$, $m^\star \ge 1$ and $N \geq m^\star+1$ and
a set $M\subseteq \{1,\ldots,m^\star\}$. Assume that $\alpha^\star :=
\min_{m \in M} \{ \alpha_{N,m}^\omega \} > 0$ where $\alpha_{N,m}^\omega$ denotes
the optimal value of optimization Problem \ref{optprob}. Then for each
optimal control problem \eqref{2:eq:discrete control system},
\eqref{2:eq:finite value function} satisfying the Assumptions
\ref{3:ass:controllability} and \ref{A-ass} the multistep MPC
feedback law $\mu_{N, m^\star}$ asymptotically stabilizes the set
$A$ for all admissible control horizon sequences
$(m_i)_{i\in\N_0}$. Furthermore, the function $V_N$ is a 
Lyapunov function at the transmission times $\sigma(k)$ in the sense
that 
	\begin{eqnarray}
		\label{lyapineq}
		V_N(x_{\mu_{N, m^\star}}(\sigma(k+1)))
                & \leq & 
                V_N(x_{\mu_{N, m^\star}}(\sigma(k))) \\ 
                &&  - \,\, \nonumber 
                \alpha^\star V_{m_k}(x_{\mu_{N,
                    m^\star}}(\sigma(k)))  
	\end{eqnarray}
holds for all $k \in \N_0$ and $x_0\in X$.
\end{theorem}
\begin{proof}
	From \eqref{lbound} and Lemma \ref{VNboundlemma} we
        immediately obtain the inequality 
	\begin{equation}
		 \label{lyapbound}
		\alpha_1(\| x \|_A) \leq V_N(x) \leq B_N(\alpha_2(\| x \|_A)).
	\end{equation}
	Note that $B_N \circ \alpha_2$ is again a
        $\KK_\infty$--function. The stated Lyapunov inequality
        \eqref{lyapineq} follows immediately from the definition of
        $\alpha^\star$ and \eqref{2:prop:m-step suboptimality
          estimate:eq1} which holds according to Corollary
        \ref{optcor} for all $m \in M$. Again, using \eqref{lbound} we
        obtain $V_m(x) \geq \alpha_1(\| x \|_A)$ and thus a standard
        construction (see, e.g., \cite{NeTe04}) yields a
        $\KL$--function $\rho$ for which the inequality
        $V_N(x_{\mu_{N, m^\star}}(\sigma(k))) \leq \rho
        (V_N(x), k)\leq \rho
        (V_N(x), \lfloor \sigma(k)/m^\star\rfloor)$ holds. In
        addition, using the definition of 
        $\mu_{N, m^\star}$, for $p = 1, \ldots,
        m_k-1$, $k\in\N_0$, and abbreviating $x(n)=x_{\mu_{N, m^\star}}(n)$
        we obtain 
	\begin{eqnarray*} 
		&& V_N(x(\sigma(k)+p)) \\
		&& \le  \sum_{n = \sigma(k)+p}^{\sigma(k+1)-1} l(x(n),
                \mu_{N, m^\star}(x(\varphi(n)), n-\varphi(n)))\\ 
		&& \quad + \,\, V_{N - m_k + p}(x(\sigma(k+1))) \\
		&& \le  \sum_{n = \sigma(k)}^{\sigma(k+1)-1} l(x(n),
                \mu_{N, m^\star}(x(\varphi(n)), n-\varphi(n)))\\ 
		&& \quad + \,\, V_{N - m_k + p}(x(\sigma(k+1))) \\
		&& \leq  \,\, V_N(x(\sigma(k))) + V_N(x(\sigma(k+1)))\,\,
                \leq \,\, 2 V_N(x(\sigma(k))) 
	\end{eqnarray*}
	where we have used \eqref{lyapineq} in the last inequality. Hence, we obtain the estimate $V_N(x_{\mu_{N,
            m^\star}}(n)) \leq 2 \rho(V_N(x),\lfloor
        \varphi(n)/m^\star\rfloor)$ which implies 
	\begin{eqnarray*}
		\|x_{\mu_{N,m^\star}}(n)\|_A & \leq & \alpha_1^{-1}(V_N(x_{\mu_{N,m^\star}}(n))) \\
		& \leq & \alpha_1^{-1}(2\rho(V_N(x),\lfloor
        \varphi(n)/m^\star\rfloor)) \\
		& \leq & \alpha_1^{-1}(2 \rho(B_N( \alpha_2(\| x
                \|_A)), \lfloor
        (n-m^\star)/m^\star\rfloor)) 
	\end{eqnarray*}
	and thus asymptotic stability with $\KL$-function given by, e.g., $\tilde{\beta}(r, n) = \linebreak[0] \alpha_1^{-1}(2 \rho(B_N( \alpha_2(r)),\linebreak[0] \lfloor
        (n-m^\star)/m^\star\rfloor)) + r e^{-n}$.
\end{proof}
\begin{remark}\label{rem:tight}
(i) For the ``classical'' MPC case $m^\star=1$ and $\beta$ satisfying 
\eqref{3:eq:submultiplicativity}
it is shown in
\cite[Theorem 5.3]{G07} that the criterion from Theorem \ref{stabthm}
is tight in the sense that if $\alpha^\star<0$ holds then there exists
a control system which satisfies Assumption \ref{3:ass:controllability}
but which is not stabilized by the MPC scheme. We conjecture that the
same is true for the general case $m^\star \ge 2$.

(ii) Note that in Theorem \ref{stabthm} we use a criterion for arbitrary but fixed $m\in M$ in order to conclude asymptotic stability for time varying $m_i\in M$. This is possible since our proof yields $V_N$ as a common Lyapunov function for all $m\in M$, cf.\ also \cite[Section 2.1.2]{Lib03}.
\end{remark}

\section{Calculation of $\alpha_{N,m}^\omega$}\label{alphasec}

In this section we continue the analysis of Problem \ref{optprob} in the extended version \eqref{opt_with_omega}, i.e., including an additional terminal weight. Although this is an optimization problem of much lower complexity than the original MPC optimization problem, still, it is in general nonlinear. However, it becomes a linear program if $\beta(r,n)$ (and thus $B_k(r)$ from \eqref{3:eq:definition B_N}) is linear in $r$.
\begin{lemma}\label{linlemma}
	Let $\beta(r, t)$ be linear in its first argument. Then Problem \ref{optprob} yields the same optimal value $\alpha_{N,m}^\omega$ as 
	\begin{equation}\label{minexplin}
		\min_{\lambda_0, \lambda_1, \ldots, \lambda_{N-1}, \nu} \sum_{n = 0}^{N - 2} \lambda_n + \omega \lambda_{N-1} - \nu
	\end{equation}
	subject to the (now linear) constraints \eqref{4:prop:necessary conditions:eq1}, \eqref{4:prop:necessary conditions:eq2} with $B_N(k)$ from \eqref{3:eq:definition B_N_omega} and
	\begin{equation}\label{posval_lin}
		\lambda_0, \ldots, \lambda_{N - 1}, \nu \geq 0, \quad \sum_{n = 0}^{m - 1} \lambda_n = 1.
	\end{equation}
\end{lemma}
For a proof we refer to \cite[Remark 4.3 and Lemma 4.6]{G07}, observing that this proof is easily extended to $\omega \geq 1$.
\begin{proposition}\label{Alpha_Formula_LP_Lemma}
	Let $\beta(\cdot,\cdot)$ be linear in its first argument and define $\gamma_k:=B_k(r)/r$. Then the optimal value of Problem \ref{optprob} equals the optimal value of the optimization problem
	\begin{equation*}
		\min_{\lambda}\quad 1 - (\gamma_{m+1} - \omega)\, \lambda_{N-1}
	\end{equation*}
	subject to $\lambda = (\lambda_1,\ldots,\lambda_{N-1})^T \geq 0$ componentwise and the linear constraints
	\begin{eqnarray}
		\gamma_N \sum\limits_{n=1}^{m-1} \lambda_n + \sum_{n=m}^{N-2} \lambda_n + \rlap{$\omega$}\phantom{\gamma_{m+1}} \lambda_{N-1} & \leq & \gamma_N -1 \label{Alpha_Formula_LP_Lemma_Restriction1} \\
		\sum_{n=j}^{N-2} \lambda_n - \gamma_{N-j\phantom{+m}}\, \lambda_j + \rlap{$\omega$}\phantom{\gamma_{m+1}} \lambda_{N-1} & \leq & 0 \qquad\qquad(j=\phantom{m} 1,\ldots,N-2) \label{Alpha_Formula_LP_Lemma_Restriction2} \\
		\sum_{n=j}^{N-2} \lambda_n - \gamma_{N-j+m}\, \lambda_j + \gamma_{m+1} \lambda_{N-1} & \leq & 0 \qquad\qquad (j=\phantom{1} m,\ldots,N-2). \label{Alpha_Formula_LP_Lemma_Restriction3}
	\end{eqnarray}
\end{proposition}
\begin{proof}
	We proceed from the linear optimization problem stated in Lemma \ref{linlemma} and show that Inequality \eqref{4:prop:necessary conditions:eq2}, $j=N-m-1$, is active in the optimum. To this end, we assume the opposite and deduce a contradiction. $\lambda_{N-1}>0$ allows -- due to the continuity of $B_{m+1}(\lambda_{N-1})$ with respect to $\lambda_{N-1}$ -- for reducing this variable without violating Inequality \eqref{4:prop:necessary conditions:eq2}, $j=N-m-1$. As a consequence the objective function decreases strictly whereas all other constraints remain valid. Hence, $\lambda_{N-1}=0$ holds. Since $\lambda_{N-2} \leq B_{m+2}(\lambda_{N-2})$ Inequalities \eqref{4:prop:necessary conditions:eq2}, $j=N-m-2$, and \eqref{4:prop:necessary conditions:eq1}, $k=N-2$, hold trivially. Analogously to $\lambda_{N-1}>0$ we show $\lambda_{N-2}=0$. Iterative application of this observation provides $\lambda_m=\ldots,\lambda_{N-1}=0$. But then the right hand side of \eqref{4:prop:necessary conditions:eq2}, $j=N-m-1$, is equal to zero which -- in combination with $\nu \geq 0$ -- leads to the claimed contradiction. 

	This enables us to treat Inequality \eqref{4:prop:necessary conditions:eq2}, $j=N-m-1$, as an equality constraint. In conjunction with the non-negativity conditions imposed on $\lambda_m,\ldots,\lambda_{N-1}$ this ensures $\nu \geq 0$. Moreover, $\lambda_0 \geq 0$ is satisfied for all feasible points due to Inequality \eqref{4:prop:necessary conditions:eq1}, $k=0$, and the linearity of $B_N$. Next, we utilize Equalities \eqref{posval_lin} and \eqref{4:prop:necessary conditions:eq2}, $j=N-m-1$, in order to eliminate $\nu$ and $\lambda_0$ from the considered optimization problem. Using these equalities and the the definition of $\gamma_{m+1}$ converts the objective function from Lemma \ref{Alpha_Formula_LP_Lemma} into the desired form. Furthermore, Equality \eqref{posval_lin} provides the equivalence of Inequalities \eqref{4:prop:necessary conditions:eq1}, $k=0$, and \eqref{Alpha_Formula_LP_Lemma_Restriction1}. Taking Equality \eqref{4:prop:necessary conditions:eq2}, $j=N-m-1$, into account yields
	\begin{equation*}
		\sum_{n=m+j}^{N-2} \lambda_n + \gamma_{m+1} \lambda_{N-1} - \gamma_{N-j}\lambda_{m+j} \leq 0
	\end{equation*}
	for \eqref{4:prop:necessary conditions:eq2}, $j=0,\ldots,N-m-2$. Shifting the control variable $j$ shows the equivalence to \eqref{Alpha_Formula_LP_Lemma_Restriction3}, $j=m,\ldots,N-2$. Paraphrasing \eqref{4:prop:necessary conditions:eq1} provides \eqref{Alpha_Formula_LP_Lemma_Restriction2} for $k=1,\ldots,N-2$.
\end{proof}

Before we proceed, we formulate Problem \ref{Alpha_Formula_Relaxed_Problem} by dropping Inequalities \eqref{Alpha_Formula_LP_Lemma_Restriction2}, $j=m,\ldots,N-2$. The solution of this relaxed (optimization) problem paves the way for dealing with Problem \ref{optprob}.
\begin{problem}\label{Alpha_Formula_Relaxed_Problem}
	Minimize $1 - (\gamma_{m+1} - \omega)\, \lambda_{N-1}$ subject to $\lambda = (\lambda_1,\ldots,\lambda_{N-1})^T \geq 0$ componentwise and $A \lambda \leq \bar{b}$, where
	\begin{equation*}
		A:=\left( \begin{array} {cccccc}
			a_1    & a_2    & \dots  & a_{N-2} & \omega  \\
			d_1    & 1      & \dots  & 1       & b_1     \\
			0      & d_2    & \ddots & \vdots  & \vdots  \\
			\vdots & \ddots & \ddots & 1       & b_{N-3} \\
			0      & \dots  & 0      & d_{N-2} & b_{N-2} \\
		\end{array} \right) \quad\mbox{and}\quad \bar{b}:=\left( \begin{array}{c}
			\gamma_{N} - 1 \\
			0 \\
			\vdots \\
			0 \\
			0 \\
		\end{array} \right)
	\end{equation*}
	with
	{\begin{equation*}
		a_j = \left\{ \begin{array}{cl}
			\gamma_N & \mbox{for $j<m$} \\
			1 & \mbox{otherwise}
		\end{array} \right. \quad
		b_j = \left\{ \begin{array}{cl}
			\omega & \mbox{for $j<m$} \\
			\gamma_{m+1} & \mbox{otherwise}
		\end{array} \right. \quad
		d_j = \left\{ \begin{array}{cl}
			1- \gamma_{N-j} & \mbox{for $j<m$} \\
			1- \gamma_{N-j+m} & \mbox{otherwise}
		\end{array} \right.
	\end{equation*}}
\end{problem}
\begin{theorem}\label{alpha_formula_thm}
	Let $\beta(\cdot,\cdot)$ be linear in its first argument and satisfy \eqref{3:eq:submultiplicativity}. Then the optimal value ${\alpha} = {\alpha}_{N,m}^\omega$ of Problem \ref{optprob} for given optimization horizon $N$, control horizon $m$, and weight $\omega$ on the final term satisfies ${\alpha}_{N,m}^\omega=1$ if and only if $\omega \geq \gamma_{m+1}$. Otherwise, we get
	{\begin{equation}\label{alpha_formula}
		{\alpha}_{N,m}^\omega =	1 - \frac {(\gamma_{m+1}-\omega) \prod\limits_{i=m+2}^N (\gamma_i - 1) \prod\limits_{i=N-m+1}^{N} (\gamma_i - 1)} {\left(\prod\limits_{i=m+1}^N \gamma_{i} - (\gamma_{m+1}-\omega) \prod\limits_{i=m+2}^{N} (\gamma_{i}-1)\right) \left( \prod\limits_{i=N-m+1}^N \gamma_{i} - \prod\limits_{i=N-m+1}^N (\gamma_{i}-1) \right)}.
	\end{equation}}
\end{theorem}
\begin{proof}
	We have shown that the linear optimization problem stated in Proposition \ref{Alpha_Formula_LP_Lemma} yields the same optimal value as Problem \ref{optprob} for $\mathcal{KL}_0$-functions which are linear in their first argument. Technically, this is posed as a minimization problem. Taking the restriction $\lambda_{N-1} \geq 0$ into account leads to the determinable question whether the coefficient of $\lambda_{N-1}$ is positive or not. As a consequence, the aim is either minimizing or maximizing $\lambda_{N-1}$. In the first case, i.e., $\gamma_{m+1} - \omega \leq 0$, choosing $\lambda_1=\ldots=\lambda_{N-1} = 0$ solves the considered task and provides the optimal value $\alpha_{N,m}^\omega = 1$.

	In order to prove the assertion we solve the relaxed Problem \ref{Alpha_Formula_Relaxed_Problem} and show that its optimum is also feasible for Problem \ref{optprob}. Suppose that $\lambda_{m+1} - \omega > 0$ holds, then Lemma \ref{appendix_lgs} shows the optimum's crucial characteristic to satisfy the linear system of equations $A \lambda = \bar{b}$ with $A$ and $\bar{b}$ from Problem \ref{Alpha_Formula_Relaxed_Problem}. We proceed by deriving formulae for $\lambda_{N-2},\ldots,\lambda_1$ depending (only) on $\lambda_{N-1}$. These allow for an explicit calculation of $\lambda_{N-1}$ from $A_1 \lambda = \bar{b}_1$. To this end, define $\delta_i :=-d_i > 0$ and begin with showing the equality
	\begin{equation}\label{Alpha_Formula_Representation_Lambda_m_bis_N-2}
		\lambda_{N-1-i} = \left( \prod_{j=1}^{i-1} ( 1 + \delta_{N-1-j} )/\delta_{N-1-j} \right) \gamma_{m+1} \lambda_{N-1} / \delta_{N-1-i}
	\end{equation}
	for $i=1,\ldots,N-1-m$ by induction which is obvious for $i=1$. Thus, we continue with the induction step using Lemma \ref{appendix_technical_lemma:1}:
	{\small\begin{eqnarray*}
		\lambda_{N-1-i} & = & \frac {1}{\delta_{N-1-i}} \left[ \gamma_{m+1} \lambda_{N-1} + \sum_{k=1}^{i-1} \lambda_{N-1-k} \right] \stackrel{I.A.}{=} \frac {\gamma_{m+1} \lambda_{N-1}}{\delta_{N-1-i}} \left[ 1 + \sum_{k=1}^{i-1} \frac {\prod_{j=1}^{k-1} ( 1 + \delta_{N-1-j} )} {\prod_{j=1}^k \delta_{N-1-j}} \right] \\
		& = & \frac {\gamma_{m+1} \lambda_{N-1}} {\prod_{j=1}^i \delta_{N-1-j}} \sum_{k=0}^{i-1} \left( \prod_{j=1}^{k-1} (1+\delta_{N-1-j}) \prod_{j=k+1}^{i-1} \delta_{N-1-j} \right)
		\stackrel{\eqref{appendix_technical_lemma:1_eq1}} {=} \frac {\gamma_{m+1} \prod_{j=1}^{i-1} ( 1 + \delta_{N-1-j} )} {\prod_{j=1}^i \delta_{N-1-j}} \lambda_{N-1}.
	\end{eqnarray*}}
	Similarly, in consideration of \eqref{appendix_technical_lemma:1_eq1} applied with $N-1 = m$ one obtains the representation $\lambda_{m-i} = \left( \prod_{j=1}^{i-1} (1+\delta_{m-j}) / \delta_{m-j} \right) (\omega \lambda_{N-1} + \sum_{j=m}^{N-2} \lambda_j) / \delta_{m-i}$ for $i=1,\ldots,m-1$. We consider the left hand side of $A_1 \lambda = \bar{b}_1$:
	\begin{eqnarray*}
		& & \gamma_N \sum_{i=1}^{m-1} \lambda_i + \sum_{i=m}^{N-2} \lambda_i + \omega \lambda_{N-1} = \gamma_N \sum_{i=1}^{m-1} \lambda_{m-i} + \sum_{i=1}^{N-1-m} \lambda_{N-1-i} + \omega \lambda_{N-1} \\
		& = & \left[ \gamma_N \left( \omega + \gamma_{m+1} \sum_{i=1}^{N-1-m} \frac {\prod_{j=1}^{i-1} (1+\delta_{N-1-j})}{\prod_{j=1}^i \delta_{N-1-j}} \right) \sum_{i=1}^{m-1} \frac {\prod_{j=1}^{i-1} (1+\delta_{m-j})}{\prod_{j=1}^{i} \delta_{m-j}} \right] \lambda_{N-1} \\
			& + & \left[ \gamma_{m+1} \sum_{i=1}^{N-1-m} \frac {\prod_{j=1}^{i-1} (1+\delta_{N-1-j})}{\prod_{j=1}^i \delta_{N-1-j}} + \omega \right] \lambda_{N-1} \\
		& = & \left[ \gamma_N \left( \omega + \gamma_{m+1} \sum_{i=1}^{N-1-m} \frac {\prod_{j=1}^{i-1} \gamma_{m+1+j}}{\prod_{j=1}^i (\gamma_{m+1+j}-1)} \right) \sum_{i=1}^{m-1} \frac {\prod_{j=1}^{i-1} \gamma_{N-m+j}} {\prod_{j=1}^{i} (\gamma_{N-m+j}-1)} \right] \lambda_{N-1} \\
			& + & \left[ \gamma_{m+1} \sum_{i=1}^{N-1-m} \frac {\prod_{j=1}^{i-1} \gamma_{m+1+j}}{\prod_{j=1}^i (\gamma_{m+1+j} - 1)} + \omega \right] \lambda_{N-1}
	\end{eqnarray*}
	The common denominator of this expression is $\prod_{i=1}^{N-1-m} (\gamma_{m+1+i} - 1) \prod_{i=1}^{m-1} (\gamma_{N-m+i} - 1)$ which is equal to $\prod_{i=m+2}^N (\gamma_i - 1) \prod_{i=N+1-m}^{N-1} (\gamma_i - 1)$. Thus, the nominator equals $\lambda_{N-1}$ with the coefficient
	{\footnotesize\begin{eqnarray*}
	& & \Big( \omega \prod_{i=m+2}^N (\gamma_i - 1) + \gamma_{m+1} \underbrace{\sum_{i=m+2}^N \prod_{j=m+2}^{i-1} \gamma_j \prod_{j=i+1}^N (\gamma_j - 1)}_{\stackrel {\eqref{appendix_technical_corollary:1_eq1}} = \prod_{i=m+2}^N \gamma_i - \prod_{i=m+2}^N (\gamma_i - 1)} \Big) \Big( \gamma_N \underbrace{\sum_{i=1}^{m-1} \left( \prod_{j=1}^{i-1} \gamma_{N-m+j} \prod_{j=i+1}^{m-1} (\gamma_{N-m+j} - 1) \right)}_{\stackrel {\eqref{appendix_technical_lemma:1_eq1}} = \prod_{j=N-m+1}^{N-1} \gamma_j - \prod_{j=N-m+1}^{N-1} (\gamma_j - 1)} + \prod_{i=N-m+1}^{N-1} (\gamma_i - 1) \Big) 
	\end{eqnarray*}}%
	where we used \eqref{appendix_technical_lemma:1_eq1} from Lemma \ref{appendix_technical_lemma:1} with $\delta_{N-1-j} = \gamma_{N-m+j} - 1$. Hence, taking the coefficient $(\gamma_{m+1} - \omega)$ of $\lambda_{N-1}$ in the objective function and $\bar{b}_1 = \gamma_N - 1$ into account, we obtain formula \eqref{alpha_formula} as the optimal value of Problem \ref{Alpha_Formula_Relaxed_Problem}.

	However, the assertion claims this to be the optimal value for Problem \ref{optprob} as well. In order to prove this it suffices to show that the optimum of Problem \ref{Alpha_Formula_Relaxed_Problem} satisfies the Inequalities \eqref{Alpha_Formula_LP_Lemma_Restriction2}, $j=m\ldots,N-2$. As a consequence, it solves the optimization problem stated in Proposition \ref{Alpha_Formula_LP_Lemma} which is equivalent to Problem \ref{optprob}. As a byproduct, this covers the necessity of the previously considered condition $\gamma_{m+1}-\omega \leq 0$ in order to obtain $\alpha_{N,m}^\omega = 1$.

	We perform a pairwise comparison of Inequality \eqref{Alpha_Formula_LP_Lemma_Restriction3} and \eqref{Alpha_Formula_LP_Lemma_Restriction2} for $j \in \{m,\ldots,N-2\}$ in order to show that the Inequalities \eqref{Alpha_Formula_LP_Lemma_Restriction2}, $j=m,\ldots,N-2$ are dispensable. To this end, it suffices to show
	\begin{equation}\label{Alpha_Formula_Inequality_Lemma_Proof1}
		(\gamma_{m+1}-\omega) \lambda_{N-1} \geq (\gamma_{N-j+m} - \gamma_{N-j}) \lambda_j \qquad j=m,\ldots,N-2.
	\end{equation}
	Equation \eqref{Alpha_Formula_Representation_Lambda_m_bis_N-2} characterizes the components $\lambda_j$, $j=m,\ldots,N-2$, in the optimum of Problem \ref{Alpha_Formula_Relaxed_Problem} by means of the equation $\left( \prod_{i=m+2}^{N-j+m} (\gamma_i-1) \right) \lambda_j = \gamma_{m+1} \left( \prod_{i=m+2}^{N-j+m-1} \gamma_i \right) \lambda_{N-1}$, $j = m,\ldots,N-2$. Using this representation $\lambda_j$ which (only) depends on $\lambda_{N-1}$ Inequality \eqref{Alpha_Formula_Inequality_Lemma_Proof1} is equivalent to
	\begin{equation*}
		(\gamma_{m+1} - \omega) \prod_{i=m+2}^{N-j+m} (\gamma_i - 1) \geq (\gamma_{N-j+m} - \gamma_{N-j}) \prod_{i=m+1}^{N-j+m-1} \gamma_i,\qquad j = m,\ldots,N-2.
	\end{equation*}
	Since the left hand side of this expression is equal to
	\begin{equation*}
		(\gamma_{m+1}-\omega) \prod_{i=m+2}^{N-j+m-1} (\gamma_i - 1) (c_0 - 1) + (\gamma_{m+1}-\omega) \prod_{i=m+2}^{N-j+m-1} (\gamma_i - 1) \left[ \sum_{n=1}^{N-j+m-2} c_n + \omega c_{N-j+m-1} \right],
	\end{equation*}
	$(c_0 - 1) \geq 0 $, and $(\gamma_{N-j+m} - \gamma_{N-j}) = \sum_{n=N-j-1}^{N-j+m-2} c_n + \omega c_{N-j+m-1} - \omega c_{N-j-1}$ Lemma \ref{Appendix_Lemma_Inequalities} applied for $k=1$ completes the proof.
\end{proof}
\begin{remark}
	If condition \eqref{3:eq:submultiplicativity} is not satisfied the $\alpha_{N,m}^\omega$-value which has been deduced in Theorem \ref{alpha_formula_thm} may still be used as a lower bound for the optimal value of Problem \ref{optprob} for $\mathcal{KL}_0$-functions which are linear in their first arguments, cf. Corollary \ref{convergenceNtoINFINITY}.
\end{remark}

At first glance, exponential controllability with respect to the stage costs may seem to be restrictive. However, since the stage costs can be used as a design parameter, cf. \cite[Section 7]{G07}, this includes even systems which are only asymptotically controllable. In order to illustrate this assertion we consider the control system defined by $x(n+1) = x(n) + u(n)x(n)^3$ -- which is the Euler approximation of the differential equation $\dot{x}(t) = u(t) x(t)^3$ with time step $1$ -- with $U = [-1,1]$ on $X = (-1,1) \subset \mathbb{R}$.\footnote{The state and control restrictions are necessary to preserve the characteristics of the continuous time system for the Euler approximation.} This system is asymptotically stabilizable with control function $u(\cdot) \equiv -1$, i.e., $x(n+1) = x(n) - x(n)^3$. However, it is not exponentially stabilizable. Defining
\begin{equation*}
	l(x(n),u(n)) := e^{- \frac {1}{2x(n)^2}}
\end{equation*}
for $0 < \| x(n) \| < 1$ and $l(x(n),u(n)) := \| x(n) \|$ otherwise allows for choosing $\beta(r,t) = r e^{-t/e}$, i.e., a $\mathcal{KL}$-function of type \eqref{3:eq:exponential controllability}. We have to establish the inequality
\begin{equation*}
	l(x(n+1)) = l(x(n)-x(n)^3) = e^{- \frac {1}{2x(n)^2(1-x(n)^2)^2}} \leq e^{-1} l(x(n)) = e^{-1} e^{-\frac{1}{2x(n)^2}}
\end{equation*}
which implies Assumption \ref{3:ass:controllability} inductivly and is equivalent to
\begin{equation*}
	1 \geq 2x(n)^2(1-x(n)^2)^2 + (1-x(n))^2 = 1 - 3x(n)^4 + 2x(n)^6.
\end{equation*}
Since $\| x(n) \| \leq 1$ this inequality holds. Thus, we have obtained exponential controllability with respect to suitably chosen stage costs.
\begin{remark}
	Note that Assumption \ref{3:ass:controllability} is not merely an abstract condition. Rather, in connection with Formula \eqref{alpha_formula} it can be used for analyzing differences in the MPC closed loop performance for different stage costs $l(\cdot,\cdot)$ and thus for developing design guidelines for selecting good cost functions $l(\cdot,\cdot)$. This has been carried out, for instance, for the linear wave equation with boundary control in \cite{Altmueller}, for a semilinear parabolic PDE with distributed and boundary control in \cite{AltmuellerBFG} (see also \cite{G07} for a preliminary study), and for a discrete time 2d test example in \cite{G07}.
\end{remark}

\section{Characteristics of $\alpha_{N,m}^\omega$ depending on the optimization horizon $N$}\label{alphasec1}

Theorem \ref{alpha_formula_thm} enables us to easily compute the performance bounds $\alpha_{N,m}^\omega$ which are needed in Theorem \ref{stabthm} in order to prove stability provided $\beta$ is known. However, even if $\beta$ is not known exactly, we can deduce valuable information. The following corollary is obtained by a careful analysis of the fraction in \eqref{alpha_formula}.
\begin{corollary} \label{convergenceNtoINFINITY}
	For each fixed $m$, $\beta$ of type \eqref{3:eq:exponential controllability} or \eqref{3:eq:finite time controllability} and $\omega\ge 1$ we have $\lim_{N\to\infty} \alpha_{N,m}^\omega=1$. In particular, for sufficiently large $N$ the assumptions of Theorem \ref{stabthm} hold and hence the closed loop system is asymptotically stable.
\end{corollary}
\begin{proof}
	Since $\beta(r,n)$ is summable, i.e., $\sum_{n=0}^\infty \beta(r,n) < \infty$, there exists an index $\widetilde{m}$ such that $\omega \sum_{n=\widetilde{m}}^\infty c_n \leq \varepsilon < 1$. It suffices to investigate the case $\gamma_{m+1} - \omega > 0$ because otherwise the assertion holds trivially. We have to show that the subtrahend of the difference in formula \eqref{alpha_formula} converges to zero as the optimization horizon $N$ tends to infinity. To this aim, we divide the term under consideration into two factors. One of them is the following which is bounded for sufficiently large $N$, i.e., $N > \widetilde{m} + m$,
	{\small\begin{equation*}
		\frac {\prod_{N-m+1}^N (\gamma_i - 1)} {\prod_{N-m+1}^N \gamma_i - \prod_{N-m+1}^N (\gamma_i - 1)} < \frac {m (\gamma_{\widetilde{m}} + \varepsilon - 1)} {m(\gamma_{\widetilde{m}} - (\omega - 1) c_{\widetilde{m}} - (\gamma_{\widetilde{m}} - (\omega - 1) c_{\widetilde{m}} + \varepsilon - 1))} = \frac {\gamma_{\widetilde{m}} + \varepsilon - 1} {1 - \varepsilon} < \infty.
	\end{equation*}}
	Hence, we focus on the other factor, i.e., 
	\begin{eqnarray*}
		\frac {(\gamma_{m+1} - \omega) \prod_{m+2}^N (\gamma_i - 1)} {\prod_{m+1}^N \gamma_i - (\gamma_{m+1} - \omega) \prod_{m+2}^N (\gamma_i - 1)} 
		& = & \frac {\prod_{m+1}^N \gamma_i} {\prod_{m+1}^N \gamma_i - (\gamma_{m+1} - \omega) \prod_{m+2}^N (\gamma_i - 1)} - 1 \\
		& = & \frac {\gamma_{m+1}} {\omega + (\gamma_{m+1} - \omega) \left( \frac{ \prod_{m+2}^N \gamma_i - \prod_{m+2}^N (\gamma_i - 1)} {\prod_{m+2}^N \gamma_i} \right)} - 1.
	\end{eqnarray*}
	Showing the convergence of this term to zero for $N$ tending to infinity completes the proof. Thus, it suffices to prove $\prod_{m+2}^N (\gamma_i - 1)/ \gamma_i \longrightarrow 0$ for $N$ tending to infinity. 
	Taking into account $\gamma_{\widetilde{m}} - (\omega - 1) c_{\widetilde{m}} \leq \gamma_i$ for all $i \geq \widetilde{m}$, we derive the desired convergence by the estimate
	\begin{equation*}
		\prod_{m+2}^N \frac {\gamma_i - 1}{\gamma_i} \leq \prod_{m+2}^{\widetilde{m}} \frac {\gamma_i - 1}{\gamma_i} \left( \frac {\gamma_{\widetilde{m}} - (\omega - 1) c_{\widetilde{m}} + \varepsilon - 1} {\gamma_{\widetilde{m}} - (\omega - 1) c_{\widetilde{m}}} \right)^{N-\widetilde{m}} \stackrel {N \rightarrow \infty} {\longrightarrow} 0.
	\end{equation*}
\end{proof}

Corollary \ref{convergenceNtoINFINITY} ensures stability for sufficiently large optimization horizons $N$ which has already been shown in \cite{GMTT05} under similar conditions (see also \cite{JadH05} for an analogous result in continuous time). Our result generalizes this assertion to arbitrary, but fixed control horizons $m$. Furthermore, similar to \cite{GruR08} for $\omega=1$, it also implies that for $N\to\infty$ the infinite horizon cost $V_\infty^{\mu_{N,m}}$ will converge to the optimal value $V_\infty$ (using the inequality $\alpha_{N,m}^1 V_\infty^{\mu_{N,m}} \le V_N$ from Theorem \ref{4:thm:optimality} and the obvious inequality $V_N\le V_\infty$ for $\omega=1$).    

However, compared to these references, our approach has the significant advantage that we can also investigate the influence of different quantitative characteristics of $\beta$, e.g., the overshoot $C$ and decay rate $\sigma$ in the exponentially controllable case \eqref{3:eq:exponential controllability}. For instance, the task of calculating all parameter combinations $(C,\, \sigma)$ implying a nonnegative $\alpha_{N,m}^\omega$ and thus stability for a given optimization horizon $N$ can be easily performed, cf. Figure \ref{stability_region_m1}\footnote{The idea to visualize the parameter dependent stability regions in this way goes back to \cite{Voit08}.}.

\begin{figure}[!ht]
	\begin{center}
		\includegraphics[width=8.cm]{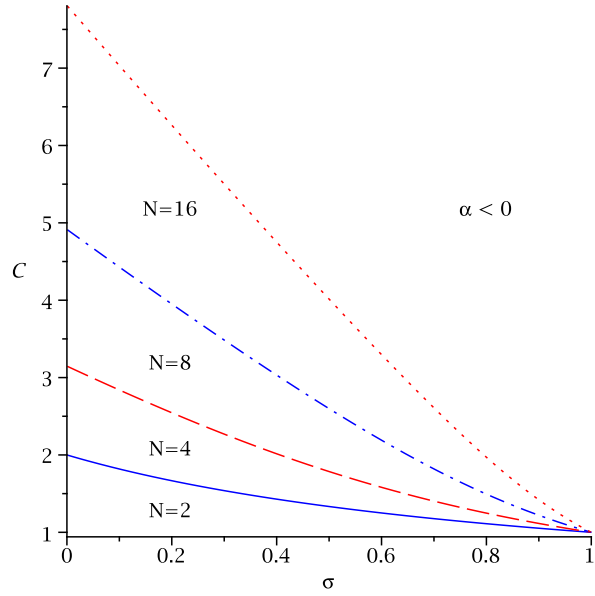}
	\end{center}
	\caption{Illustration of the stability region guaranteed by Theorem \ref{alpha_formula_thm} for various optimization horizons $N$ given a $\mathcal{KL}$-function of type \eqref{3:eq:exponential controllability} for ``classical'' MPC, i.e., $m=1$.}
	\label{stability_region_m1}
\end{figure}

As expected, the stability region grows with increasing optimization horizon $N$. Moreover, Theorem \ref{alpha_formula_thm} enables us to quantify the observed enlargement, e.g., doubling $N=2$ increases the considered area by $129.4$ percent. Furthermore, we observe that for a given decay rate $\sigma$ there always exists an overshoot $C$ such that stability is guaranteed. Indeed, Theorem \ref{alpha_formula_thm} enables us to prove this. To this end, we deal with the special case $C=1$ exhibiting a significantly simpler expression for $\alpha_{N,m}^\omega$.

\begin{proposition}\label{exp_stab_C1}
	Let the $\mathcal{KL}_0$-function be of type \eqref{3:eq:exponential controllability} and $C=1$. Then the optimal value $\alpha_{N,m}^\omega$ is equal to $\min\{1,\, 1 - (1+\sigma\omega-\omega) \sigma^{N-1}\} > 0$.
\end{proposition}
\begin{proof}
	We define the auxiliary quantity $\eta := 1 + \sigma \omega - \omega$. Then, we obtain the equalities $\gamma_i = (1-\eta \sigma^{i-1})/(1-\sigma)$, $\gamma_{i} - 1 = \sigma (1 - \eta \sigma^{i-2}) / (1-\sigma)$, and $\gamma_{m+1} - \omega = \eta (1-\sigma^{m}) / (1-\sigma)$. Thus, the necessary and sufficient condition $(\gamma_{m+1}-\omega) \leq 0$ from Theorem \ref{alpha_formula_thm} holds if and only if $\eta \leq 0$. Hence, we restrict ourselves to $\eta > 0$ and the right hand side of formula \eqref{alpha_formula} is equal to
	{\small\begin{eqnarray*}
		\alpha_{N,m}^\omega & = & 1 - \frac {
			\frac {\eta (1-\sigma^m)} {1-\sigma} \prod_{m+2}^N \frac {\sigma (1-\sigma^{i-2}\eta)} {1-\sigma} \prod_{N-m+1}^N \frac {\sigma (1-\sigma^{i-2}\eta)} {1-\sigma}
		} {
			\left( \prod_{m+1}^N \frac {1-\sigma^{i-1}\eta}{1-\sigma} - \frac {(1-\sigma^m)\eta}{1-\sigma} \prod_{m+2}^N \frac {\sigma(1-\sigma^{i-2}\eta)}{1-\sigma} \right) \left( \prod_{N-m+1}^N \frac {1-\sigma^{i-1} \eta}{1-\sigma} - \prod_{N-m+1}^N \frac {\sigma(1-\sigma^{i-2}\eta)}{1-\sigma} \right)
		}\\
		& = & 1 - \frac {
			\eta (1-\sigma^m) \sigma^{N-1} \prod_{N-m+1}^{N-m+1} (1-\sigma^{i-2}\eta)
		} {
			\underbrace{\left( (1-\sigma^{N-1} \eta) - \eta (1-\sigma^m) \sigma^{N-m-1} \right)}_{=1-\sigma^{N-m-1} \eta} \cdot \underbrace{\left( (1-\sigma^{N-1} \eta) - (1-\sigma^{N-m-1}\eta) \sigma^m \right)}_{= 1-\sigma^m}
		}\\
		& = & 1 - \eta \sigma^{N-1},
	\end{eqnarray*}}
	where we have omitted the control index.
\end{proof}
\begin{remark}
	Note that the optimal value $\alpha_{N,m}^\omega$, i.e., the solution of Problem \ref{optprob}, does not depend on the control horizon $m$ for $C=1$. Consequently, the control horizon $m$ does not play a role for this special case.
\end{remark}

Proposition \ref{exp_stab_C1} states that we always obtain a strictly positive value $\alpha_{N,m}^\omega$ for $C=1$. Due to continuity of the involved expressions this remains true for $C = 1 + \varepsilon$ for sufficiently small $\varepsilon$. Thus, for any decay rate $\sigma\in(0,1)$ and sufficiently small $C>1$ (depending on $N$, $m$ and $\omega$) we obtain $\alpha_{N,m}^\omega>0$ and thus asymptotic stability. However, this property does not hold if we exchange the roles of $\sigma$ and $C$, i.e., for a given overshoot $C>1$ stability cannot in general be concluded for a sufficiently small decay rate $\sigma>0$. 

Next, we investigate the relation between $\gamma = \sum_{n=0}^\infty c_n$ and the optimization horizon $N$ for finite time controllability in one step, i.e., for a $\mathcal{KL}_0$-function of type \eqref{3:eq:finite time controllability} satisfying \eqref{3:eq:submultiplicativity} defined by $c_0 = \gamma$ and $c_n = 0$ for all $n \in \mathbb{N}_{\geq 1}$. For this purpose, let $\gamma$ be strictly greater than $\omega \geq 1$. Otherwise Theorem \ref{alpha_formula_thm} provides $\alpha_{N,m}^\omega = 1$ regardless of the optimization horizon $N$. In this case, Formula \eqref{alpha_formula} yields 
\begin{equation*}
	\alpha_{N,m}^\omega = 1 - \frac {(\gamma - \omega) (\gamma - 1)^{N-1}}{(\gamma^{N-m} - (\gamma - \omega) (\gamma - 1)^{N-m-1}) (\gamma^{m} - (\gamma - 1)^{m})}.
\end{equation*}
We aim at determining the minimal optimization horizon $N$ guaranteeing stability for a given parameter $\gamma$. In order to ensure stability, we have to show $\alpha_{N,m}^\omega \geq 0$. We begin our examination with the smallest possible control horizon $m=1$. This leads to the inequality
\begin{equation*}
	\alpha_{N,1}^\omega = 1 - \frac {(\gamma - \omega) (\gamma - 1)^{N-1}}{\gamma^{N-1} - (\gamma - \omega) (\gamma - 1)^{N-2}} = \frac {\gamma^{N-1} - (\gamma - \omega) (\gamma - 1)^{N-2} \gamma}{\gamma^{N-1} - (\gamma - \omega) (\gamma - 1)^{N-2}} \geq 0.
\end{equation*}
Since the logarithm is monotonically increasing this is in turn equivalent to 
\begin{equation*}
	N \geq 2 + \frac { \ln (\gamma - \omega) } {\ln \gamma - \ln (\gamma - 1)} =: f(\gamma).
\end{equation*}
We show that $f(\gamma)$ tends to $\gamma \ln \gamma$ asymptotically. To this end, we consider
\begin{equation*}
	\lim_{\gamma \rightarrow \infty} \frac {f(\gamma)}{\gamma \ln \gamma} = \underbrace{\lim_{\gamma \rightarrow \infty} \frac {2}{\gamma \ln \gamma}}_{= 0} + \underbrace{\lim_{\gamma \rightarrow \infty} \frac {\ln (\gamma - \omega)}{\ln \gamma}}_{= 1} \cdot \lim_{\gamma \rightarrow \infty} \frac {\frac 1 \gamma}{\ln \gamma - \ln (\gamma - 1)} = \lim_{\gamma \rightarrow \infty} \frac {\gamma (\gamma - 1)}{\gamma^2} = 1
\end{equation*}
where we have used L'Hospital's rule. Clearly, ceiling the derived expression for the optimization horizon $N$ doesn't change the obtained result.

We continue with analysing the coherancy between $\gamma$ and $N$ for control horizons $m$ which provide the largest optimal value, i.e., $m = \lfloor N/2 \rfloor$, cf. Section \ref{alphasec2} below. Analogously, $\alpha_{N,\lfloor N/2 \rfloor}^\omega \geq 0$ induces lower bounds
\begin{equation*}
	N \geq \begin{cases}
		2 \ln \left( \frac {2\gamma - \omega - 1}{\gamma - 1}\right) / (\ln \gamma - \ln (\gamma - 1)) & \text{for even $N$} \\[2mm]
		\left( \ln \left( \frac {2 \gamma - \omega}{\gamma} \right) + \ln \left( \frac {2 \gamma - \omega}{\gamma - 1} \right) \right) / (\ln \gamma - \ln (\gamma - 1)) & \text{for odd $N$}
	\end{cases}
\end{equation*}
for the optimization horizon $N$. Again in consideration of L'Hospital's rule, the investigated expression exhibits asymptotically a behaviour like $2 \ln 2 \cdot \gamma$, cf. Figure \ref{finite_time_one_step_asymptotic}. Since the obtained approximation $2 \ln 2 \cdot \gamma$ holds for both estimates corresponding to even and odd natural numbers $N$ for $m=\lfloor N/2 \rfloor$, we have illustrated the resulting horizon lengths for given $\gamma$ with respect to both. Moreover, these estimates coincide with the numerical results derived in \cite[Section 6]{G07}.

\begin{figure}[!ht]
	\begin{center}
		\includegraphics[width=6.2cm]{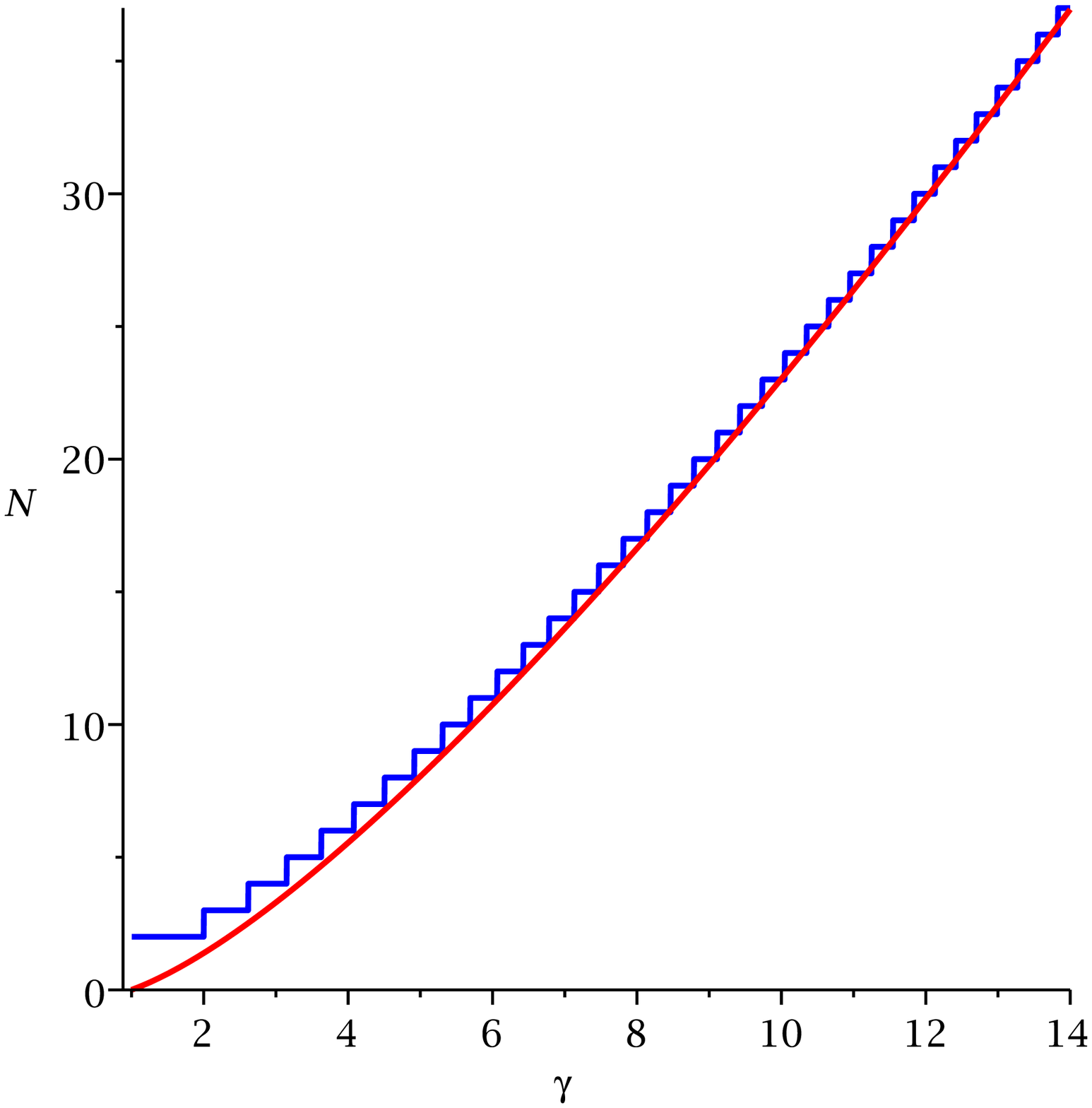}
		\includegraphics[width=6.2cm]{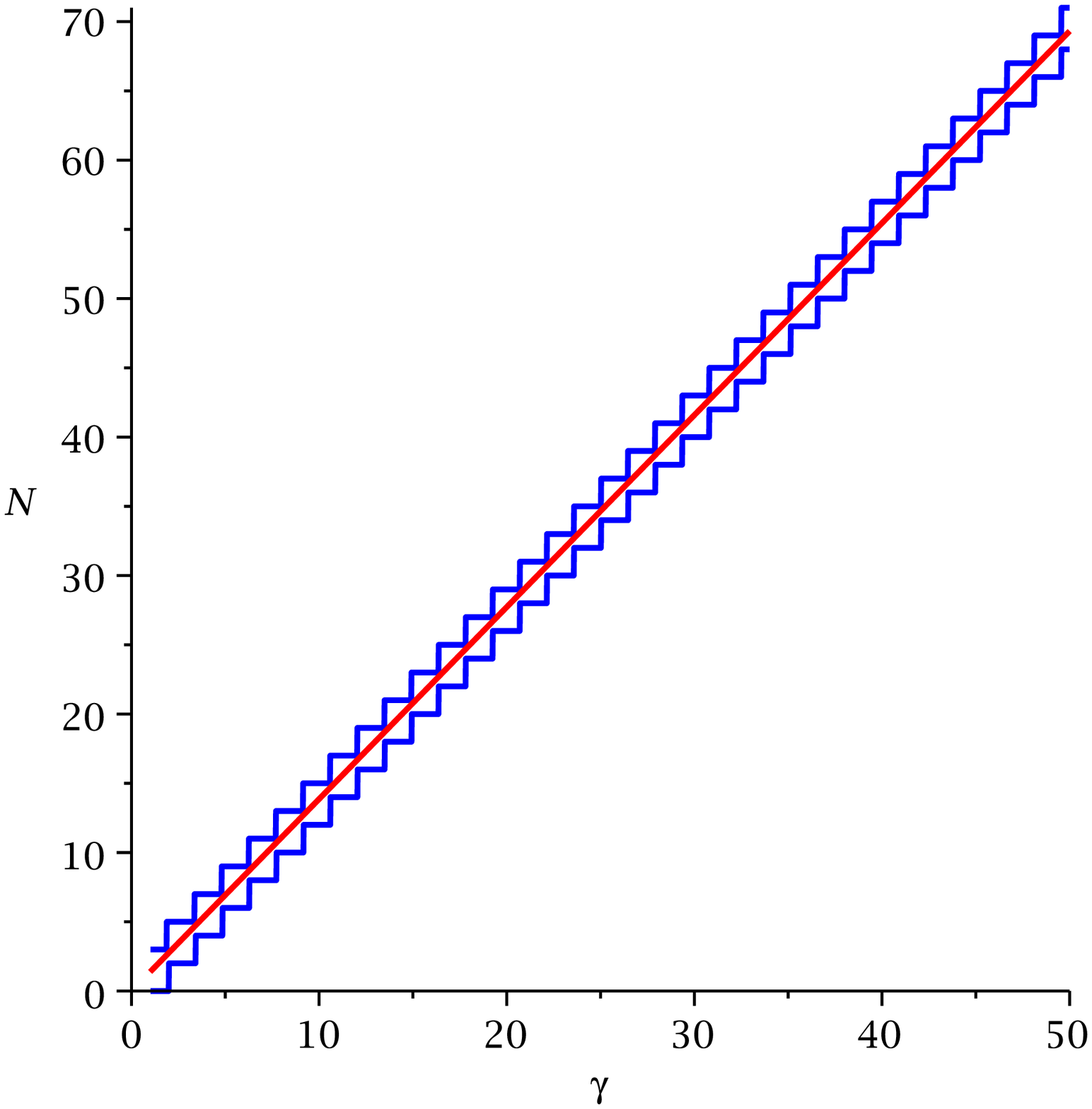}
	\end{center}
	\caption{Minimal stabilizing optimization horizons for one step finite time controllability for $m=1$ and $m = \lfloor N/2 \rfloor$ in comparison with their asymptotic approximations.}
	\label{finite_time_one_step_asymptotic}
\end{figure}

\begin{remark}
	As a consequence of Lemma \ref{linlemma} it follows that these estimates provide upper bounds for the minimal stabilizing horizons for $\mathcal{KL}_0$-functions $\beta(\cdot,\cdot)$ which are linear in their first argument and satisfy \eqref{3:eq:submultiplicativity}, e.g., for $c_0 = \gamma = C \sum_{n=0}^\infty \sigma^n$ with $C \geq 1$, $\sigma \in (0,1)$.
\end{remark}

\section{Qualitative characteristics of $\alpha_{N,m}^\omega$ depending on varying control horizon $m$}\label{alphasec2}

In the previous section we have investigated the influence of the optimization horizon $N$ on the optimal value $\alpha_{N,m}^\omega$ of Problem \ref{optprob} in the extended version. E.g., we have proven that Theorem \ref{alpha_formula_thm} ensures stability for sufficiently large optimization horizons $N$. Thus, choosing $N$ appropriately remains crucial in order to obtain suitable $\alpha_{N,m}^\omega$-values. However, Theorem \ref{stabthm} assumes the positivity of several $\alpha_{N,m}^\omega$-values with different control horizons $m$. Section \ref{alphasec1} already indicated that, e.g., the minimal stabilizing horizon depends sensitively on the parameter $m$. Thus, the question arises whether changing the control horizon persistently causes additional difficulties in order to guarantee stability.

Before proceeding, we state results concerning symmetry and monotonicity properties of the optimal value $\alpha_{N,m}^\omega$ with respect to the control horizon $m$. These results -- which are proven in Subsections \ref{subsectionSymmetry}, \ref{subsectionMonotonicity} -- do not only pave the way to answer the asked question but are also interesting in their own rights.
\begin{proposition}\label{propositionSymmetry}
	Let $\beta$ be of type \eqref{3:eq:exponential controllability} or of type \eqref{3:eq:finite time controllability} with $c_n = 0$ for $n \geq 3$. Then $\alpha_{N,m}^\omega \leq \alpha_{N,N-m}^\omega$ holds for $m \in \{1,\ldots,\lfloor N/2 \rfloor\}$, $N \in \mathbb{N}$, and $\omega \geq 1$.
\end{proposition}
\begin{proposition}\label{propositionMonotonicity}
	Let $\beta$ be of type \eqref{3:eq:exponential controllability} and $\omega \in \{1\} \cup [1/(1-\sigma),\infty)$ or of type \eqref{3:eq:finite time controllability} with $c_n = 0$ for $n \geq 2$ and $\omega \geq 1$. Then $\alpha_{N,m+1}^\omega \geq \alpha_{N,m}^\omega$ holds for $m \in \{1,\ldots,\lfloor N/2 \rfloor - 1\}$, $N \in \mathbb{N}$.
\end{proposition}

These symmetry and monotonicity properties have the following remarkable consequence for our stabilization problem.
\begin{theorem}\label{thm:stabmain} 
	Let $\beta$ be of type \eqref{3:eq:exponential controllability} and $\omega \in \{1\} \cup [1/(1-\sigma),\infty)$ or of type \eqref{3:eq:finite time controllability} with $c_n = 0$ for $n \ge 2$. Then for each $N \geq 2$ the stability criterion from Theorem \ref{stabthm} is satisfied for $m^\star=N-1$ if and only if it is satisfied for $m^\star=1$.
\end{theorem}
\begin{proof} 
	Proposition \ref{propositionSymmetry} and \ref{propositionMonotonicity} imply $\alpha_{N,m}^\omega \geq \alpha_{N,1}^\omega$ for all $m \in M$ which yields the assertion.
\end{proof}

In other words, for exponentially controllable systems without or with sufficiently large final weight and for systems which are finite time controllable in at most two steps, we obtain stability for our proposed networked MPC scheme under exactly the same conditions as for ``classical'' MPC, i.e., $m^\star=1$. In this context we recall once again that for $m^\star=1$ the stability condition of Theorem \ref{stabthm} is tight, cf.\  Remark \ref{rem:tight}.

\subsection{Symmetry Analysis}\label{subsectionSymmetry}

In this subsection we carry out a complete symmetry analysis of the optimal value $\alpha_{N,m}^\omega$ given in Theorem \ref{alpha_formula_thm} with respect to the control horizon $m$. To this end, we distinguish the special case $\omega = 1$ from $\omega > 1$, i.e., the szenario including an additional weight on the final term. The following symmetry property for $\omega = 1$ follows immediately from  Formula \eqref{alpha_formula}.
\begin{corollary}\label{alpha_symm_cor}
	For $m=1,\ldots,\lfloor \frac N2 \rfloor$ the values from Theorem \eqref{alpha_formula_thm} satisfy $\alpha_{N,m}^1 = \alpha_{N,N-m}^1$.
\end{corollary}

Depicting the corresponding $\alpha_{N,m}^\omega$-values for some $\mathcal{KL}_0$-functions, Fig. \ref{typical_values} illustrates the assertion of Corollary \ref{alpha_symm_cor}. In addition, we observe that increasing the control horizon improves the optimal value $\alpha_{N,m}^\omega$ quantitatively. Hereby, Figure \ref{typical_values} shows the interaction between symmetry and monotonicity properties which is essential for the proof of Theorem \ref{thm:stabmain}. Moreover, adding an additional weight on the final term may lead to a further improvement of the guaranteed stability behaviour. But instead of the equality $\alpha_{N,m}^1 = \alpha_{N,N-m}^1$ we observe the inequality $\alpha_{N,m}^\omega \leq \alpha_{N,N-m}^\omega$ for the setting including an additional weight on the final term for $\mathcal{KL}_0$-functions which are linear in their first argument and satisfy \eqref{3:eq:submultiplicativity}
\begin{figure}[!ht]
	\begin{center}
		\includegraphics[width=6.3cm]{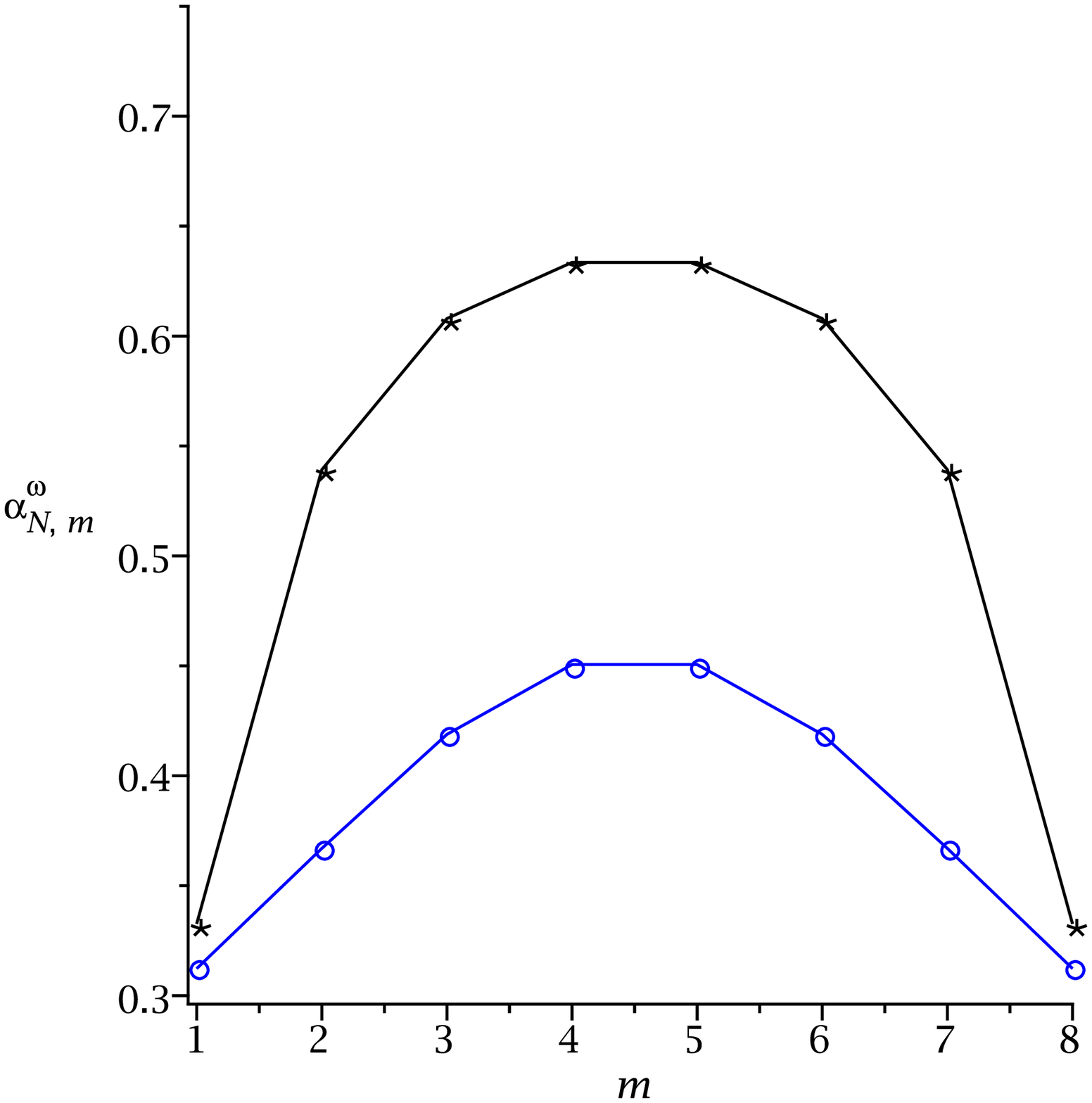}
		\includegraphics[width=6.3cm]{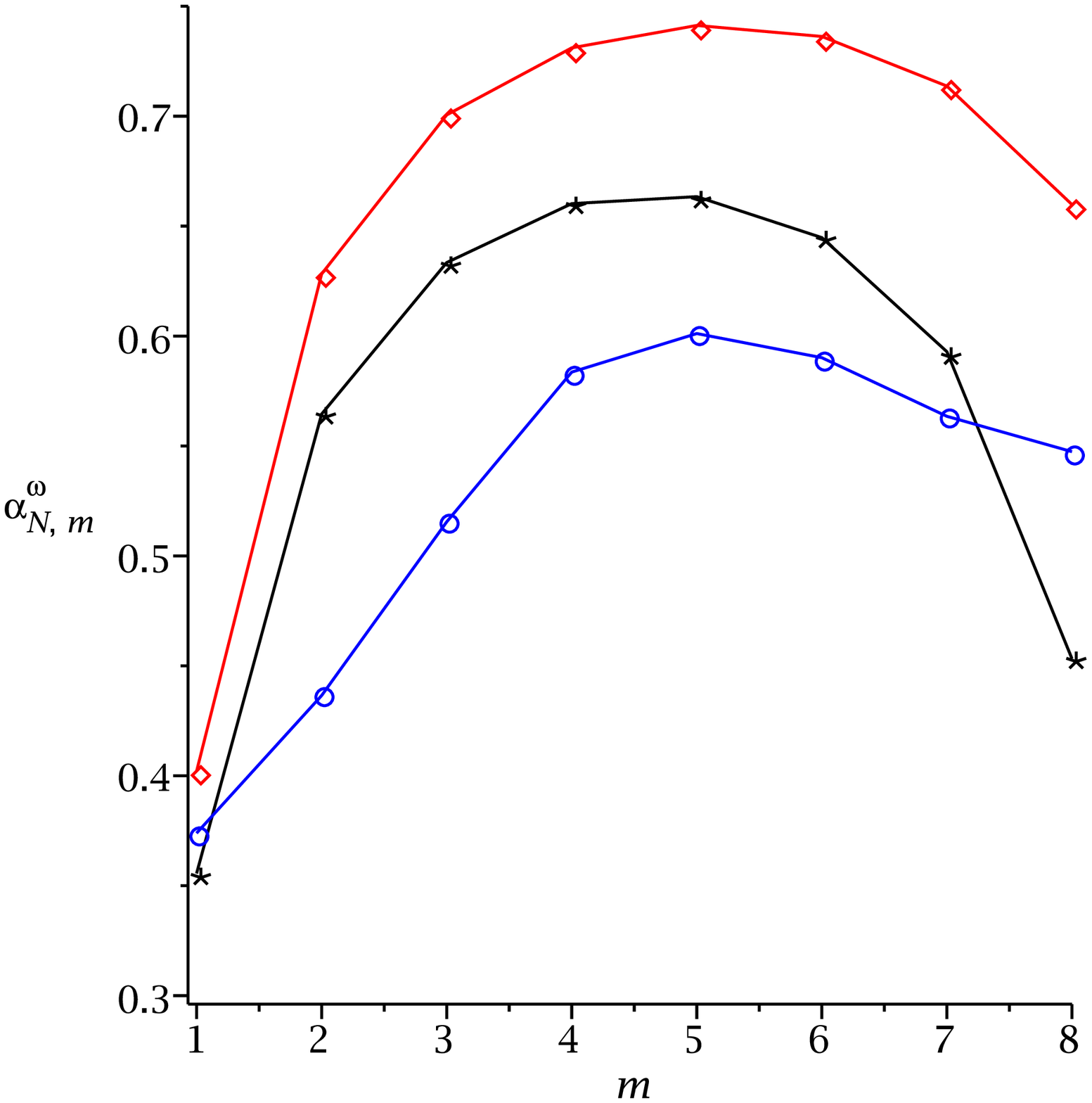}
	\end{center}
	\caption{We consider a $\mathcal{KL}$-function of type \eqref{3:eq:exponential controllability} with parameters $(C,\sigma) = (2,5/8)$, i.e., exponential controllability, and a $\mathcal{KL}_0$-function of type \eqref{3:eq:finite time controllability} satisfying \eqref{3:eq:submultiplicativity} defined by $c_0=1,\, c_1=5/4,\, c_2=3/2,\, c_3=5/4,\, c_4=1/2,\, c_5=1/4$, and $c_6=1/16$, i.e., finite time controllability (*,$\circ$). On the left we have illustrated the corresponding $\alpha_{9,\cdot}^1$-values. On the right we have depicted the same functions with final weights, i.e., $\omega = 5/4,\, 2$ for the exponentially stable case and $\omega = 2$ for finite time controllability (*,$\diamond$,$\circ$).}
	\label{typical_values}
\end{figure}

In the remainder of this subsection we prove Proposition \ref{propositionSymmetry} and demonstrate that a generalization including $\mathcal{KL}_0$-functions of type \eqref{3:eq:finite time controllability} not satisfying the given assumptions is not possible. Consequently, we aim at establishing $\alpha^\omega_{N,m} \leq \alpha^\omega_{N,N-m}$ for $\omega > 1$. Note that, if the necessary and sufficient condition $\gamma_{m+1} - \omega \leq 0$ holds, then Formula \eqref{alpha_formula} solely provides values $\alpha_{N,m}^\omega$ greater or equal one. Thus, showing the desired inequality using the expressions provided by Formula \eqref{alpha_formula} covers the assertion in the absense of the condition $\gamma_{m+1} - \omega \leq 0$. Hence, $\alpha^\omega_{N,N-m}-\alpha^\omega_{N,m} \geq 0$ is equivalent to
{\small\begin{eqnarray*}
	& & (\gamma_{N-m+1} - 1)(\gamma_{m+1}-\omega) \left[ \prod_{i=N-m+1}^N \gamma_i - (\gamma_{N-m+1}-\omega) \prod_{i=N-m+2}^N (\gamma_i-1) \right] \left[ \prod_{i=m+1}^N \gamma_i - \prod_{i=m+1}^N (\gamma_i - 1) \right] \\
	& \geq & (\gamma_{N-m+1} - \omega)(\gamma_{m+1}-1) \left[ \prod_{i=m+1}^N \gamma_i - (\gamma_{m+1}-\omega) \prod_{i=m+2}^N (\gamma_i-1) \right] \left[ \prod_{i=N-m+1}^N \gamma_i - \prod_{i=N-m+1}^N (\gamma_i - 1) \right].
\end{eqnarray*}}
Expanding these terms and reducing by $(\omega - 1) > 0$ leads to
{\small\begin{equation} \label{symmetry_ineq}
	\prod_{i=m+1}^{N-m} \gamma_i \prod_{i=N-m+1}^N (\gamma_i - 1) (\gamma_{N-m+1} - \omega) + (\gamma_{m+1} - \gamma_{N-m+1}) \prod_{i=m+1}^N \gamma_i - (\gamma_{m+1} - \omega) \prod_{i=m+1}^N (\gamma_i - 1) \geq 0.
\end{equation}}
These preliminary considerations enable us to derive results referring to the above mentioned symmetry property of $\alpha_{N,m}^\omega$ with respect to the control horizon $m$. First, we assume finite time controllability. This case is completely characterized by Lemma \ref{symmetry_omega_c1c2c3} and Remark \ref{symmetry_omega_c1c2c3_neg}.
\begin{lemma}\label{symmetry_omega_c1c2c3}
	Assume that the $\mathcal{KL}_0$-function is of type \eqref{3:eq:finite time controllability} satisfying \eqref{3:eq:submultiplicativity}. In addition, let $c_0,\, c_1,\, c_2 \geq 0$ and $c_n = 0$ for all $n \in \mathbb{N}_{\geq 3}$. Then $\alpha^\omega_{N,N-m} - \alpha^\omega_{N,m} \geq 0$ holds for $m < N-m$, $m,N \in \mathbb{N}$.
\end{lemma}
\begin{proof}
	We show that assuming $\gamma_{m+1} - \omega \leq 0$ implies $\gamma_{N-m+1} - \omega \leq 0$ for $m < N-m$. Since $\gamma_{m+1} = \sum_{n=0}^{m-1} c_n + \omega c_m$ the assertion holds for $m \geq 2$. Thus, we restrict ourselves to $m=1$, i.e., $c_0 + c_1 \omega - \omega \leq 0$. Then it follows
	\begin{equation*}
		\gamma_{N-m+1} - \omega \leq c_0 + c_1 + c_2 \omega - \omega + c_1 \omega - c_1 \omega \stackrel {\eqref{3:eq:submultiplicativity}} {\leq} (c_0 + c_1 \omega - \omega) + c_1 (c_0 + c_1 \omega - \omega) \leq 0. 
	\end{equation*}
	This covers the case $\gamma_{m+1} - \omega \leq 0$. Hence, we assume $\gamma_{m+1} - \omega > 0$. Note that Corollary \ref{alpha_symm_cor} shows the assertion for $\omega = 1$. Thus, we restrict ourselves to the case $\omega > 1$. Moreover, we suppose $\gamma_{N-m+1} - \omega > 0$ because otherwise $\alpha^\omega_{N,N-m} = 1$ and the assertion holds trivially. We make a case differentiation with respect to $m$. For $m>2$ we have $\gamma_{m+1} = \gamma_{N-m+1}$. Thus, \eqref{symmetry_ineq} is equivalent to
	\begin{equation*}
		(\gamma_{m+1} - \omega) \prod_{i=N-m+1}^N (\gamma_i - 1) \left[ \prod_{i=m+1}^{N-m} \gamma_i - \prod_{i=m+1}^{N-m} (\gamma_i - 1) \right] \geq 0
	\end{equation*}
	implying the assertion. For $m=2$ we obtain $\gamma_{m+1} = \gamma_{3} = c_0 + c_1 + \omega c_2$ and $\gamma_{N-m+1} = \sum_{n=0}^2 c_n$. Consequently, $(\gamma_{m+1} - \gamma_{N-m+1}) = (\omega - 1) c_2$ and $(\gamma_{m+1}-\omega) = (\omega - 1) c_2 + (\gamma_{N-m+1} - \omega)$ holds. Using these expressions for the corresponding terms in \eqref{symmetry_ineq} provides $\alpha^\omega_{N,N-m} \geq \alpha^\omega_{N,m}$. For $m=1$ we have $\gamma_{m+1} = \gamma_{2} = c_0 + \omega c_1$ and $\gamma_{N-m+1} = \gamma_{N}$. Thus, \eqref{symmetry_ineq} is equivalent to
	\begin{equation*}
		\left( \gamma_2(\gamma_N-1)(\gamma_N-\omega)+(\gamma_2-\gamma_N)\gamma_2 \gamma_N \right) \prod_{i=3}^{N-1} \gamma_i - \left( (\gamma_2 - \omega)(\gamma_2 - 1)(\gamma_N-1) \right) \prod_{i=3}^{N-1} (\gamma_i - 1) \geq 0.
	\end{equation*}
	We show that the coefficient of $\prod_{i=3}^{N-1} \gamma_i$ is positive, i.e., $\gamma_2 (\gamma_N(\gamma_2 - \omega - 1) + \omega) \geq 0$.
	To this end, it suffices to investigate the case $\gamma_2 - \omega - 1 < 0$. In consideration of \eqref{3:eq:submultiplicativity}, we estimate this expression as follows
	\begin{eqnarray*}
		\gamma_2 [\gamma_N(\gamma_2 - \omega - 1) + \omega] & \geq & \gamma_2 [ (c_0 + c_1 + c_1^2 \omega)(c_0 + c_1 \omega - \omega - 1) + \omega] \\
		& = & \gamma_2 [ (c_0 - 1)(c_0 + c_1 \omega - \omega + c_1 ) + c_1^2 \omega (c_0 + c_1 \omega - \omega) ] \geq 0.
	\end{eqnarray*}
	Hence, it suffices to show that the coefficient of $\prod_{i=3}^{N-1} \gamma_i$ is greater or equal than the coefficient of $\prod_{i=3}^{N-1} (\gamma_i - 1)$. This leads to the inequality $\gamma_2(\gamma_2-1) - \omega(\gamma_N - 1) \geq 0$. Here, we estimate $\gamma_N \leq c_0 + c_1 + \omega c_2$ and obtain $\omega^2 (c_1^2 - c_2) + (c_0 - 1)(c_0 + c_1 \omega - \omega + c_1 \omega)$ which holds according to the assumptions $\gamma_2 - \omega > 0$ and \eqref{3:eq:submultiplicativity}.
\end{proof}
\begin{remark}\label{symmetry_omega_c1c2c3_neg}
	Note that Lemma \ref{symmetry_omega_c1c2c3} does not hold for arbitrary $\mathcal{KL}_0$-functions of type \eqref{3:eq:finite time controllability} which satisfy \eqref{3:eq:submultiplicativity}. Consider, e.g., $c_0 = 1,\, c_1 = 3/2,\, c_2 = 2/3,\, c_3 = 1$ and $c_n = 0$ for $n \geq 4$. For $N=5$ and $m=2$ the necessary and sufficient condition for $\alpha^\omega_{5,2} = 1$ are satisfied for $\omega \geq 15/2$, i.e., $\gamma_2 - \omega \leq 0$. However, it holds $\gamma_{N-m+1} - \omega = \sum_{n=0}^2 c_n > 0$ which implies $\alpha^\omega_{5,3} < 1 = \alpha^\omega_{5,2}$.
\end{remark}

In the sense of Remark \ref{symmetry_omega_c1c2c3_neg} the assertion of Lemma \ref{symmetry_omega_c1c2c3} is strict. Hence, the deduced results only hold for a subset of the class of finite time controllable systems which satisfy \eqref{3:eq:submultiplicativity}. In contrast to that, we are able to derive much more general results for the exponentially controllable case. To this end, assume that the $\mathcal{KL}_0$-function is of type \eqref{3:eq:exponential controllability}. We begin our analysis with the case $\gamma_{m+1} - \omega \leq 0$, i.e., $\alpha_{N,m}^\omega = 1$. This condition guarantees the preservation of the symmetry property stated in Corollary \ref{alpha_symm_cor} and implies $\alpha_{N,\widetilde{m}}^\omega = 1$ for all $\widetilde{m} \in \{m+1,\ldots,N-1\}$.
\begin{lemma}\label{suff_cond_exp}
	Let the $\mathcal{KL}_0$-function $\beta$ be of type \eqref{3:eq:exponential controllability}. Then $\gamma_{m+1} - \omega \leq 0$ implies $\gamma_{\widetilde{m}+1} - \omega \leq 0$ for all $\widetilde{m} \in \{m+1,\ldots,N-1\}$.
\end{lemma}
\begin{proof}
	We have $\gamma_{m+1} - \omega = \sum_{n=0}^{m-1} c_n + (c_m - 1) \omega < 0$ which implies $c_m = C \sigma^m < 1$ and is equivalent to $\omega > C (1-\sigma^m) / (1-\sigma)(1-C\sigma^m) > 1$. Thus, it suffices to establish the following inequality based on the obtained estimate for $\omega$:
	\begin{equation*}
		\gamma_{\widetilde{m}+1} - \omega \leq \frac {C(1-\sigma^{\widetilde{m}})}{1-\sigma} - \frac {C (1-\sigma^m)} {(1-\sigma)(1-C\sigma^m)} (1 - C \sigma^{\widetilde{m}}) < 0.
	\end{equation*}
	The second inequality is equivalent to $(C-1) \sigma^{\widetilde{m}} < (C-1) \sigma^m$ showing the assertion for $\widetilde{m} > m$.
\end{proof}

Proposition \ref{exp_stab_C1} and Lemma \ref{suff_cond_exp} enable us to carry out a complete symmetry analysis of the optimal values obtained in Theorem \ref{alpha_formula_thm} (for a proof of the following lemma we refer to Subsection \ref{AppendixProofSymmetry}). Note that we do not distinguish whether an additional weight on the final term is included and do note impose any restrictions with respect to the considered $\mathcal{KL}$-functions.
\begin{lemma}\label{symmetry_omega_exp}
	Let the $\mathcal{KL}_0$-function $\beta$ be of type \eqref{3:eq:exponential controllability}. Then $\alpha^\omega_{N,N-m} - \alpha^\omega_{N,m} \geq 0$ holds for $m < N-m$, $m \in \mathbb{N}_{>0}$, $N \in \mathbb{N}_{\geq 3}$.
\end{lemma}

\subsection{Monotonicity Properties}\label{subsectionMonotonicity}

Apart from the symmetry proven in Corollary \ref{alpha_symm_cor} and its generalized counterparts for the szenario including an additional weight on the final term, cf. Lemmata \ref{symmetry_omega_c1c2c3} and \ref{symmetry_omega_exp}, one also observes a certain monotonicity property: we have $\alpha_{N,m+1}^\omega \ge \alpha_{N,m}^\omega$ for $m=1,\ldots, \lfloor N/2 \rfloor-1$. This is a very desirable feature because it implies -- in combination with the derived symmetry results -- that if the stability condition in Theorem \ref{stabthm} holds for $m^\star=1$ then is also holds for all $m^\star\le N-1$, cf.\ Theorem \ref{thm:stabmain}. However, the next example shows that this monotonicity property does not always hold.
\begin{example}\label{noMonotonicity_ex}
	We consider the $\mathcal{KL}_0$-functions $\beta_1$ and $\beta_2$ of type \eqref{3:eq:finite time controllability} defined by $c_0 = 1.24,\, c_1 = 1.14,\, c_2 = 1.04$ and $c_i = 0$ for all $i \geq 3$ for $\beta_1$ and $c_0 = 1,\, c_1=1.2,\, c_2=1.1,\, c_3=1.1,\, c_4=1.2,\, c_5=1,\, c_6=0.75,\, c_7=0.25$ and $c_i = 0$ for all $i \geq 8$ for $\beta_2$. Both functions satisfy condition \eqref{3:eq:submultiplicativity} and $\beta_1$ is, in addition, monotonically decreasing. The corresponding values $\alpha_{N,m}^1$ in Fig. \ref{picture_noMonotonicity} show that neither function satisfies $\alpha_{N,m+1}^1 \ge \alpha_{N,m}^1$ for $m=1,\ldots, \lfloor N/2 - 1 \rfloor$. 
\end{example}
\begin{figure}[!ht]
	\begin{center}
		\includegraphics[width=4.8cm]{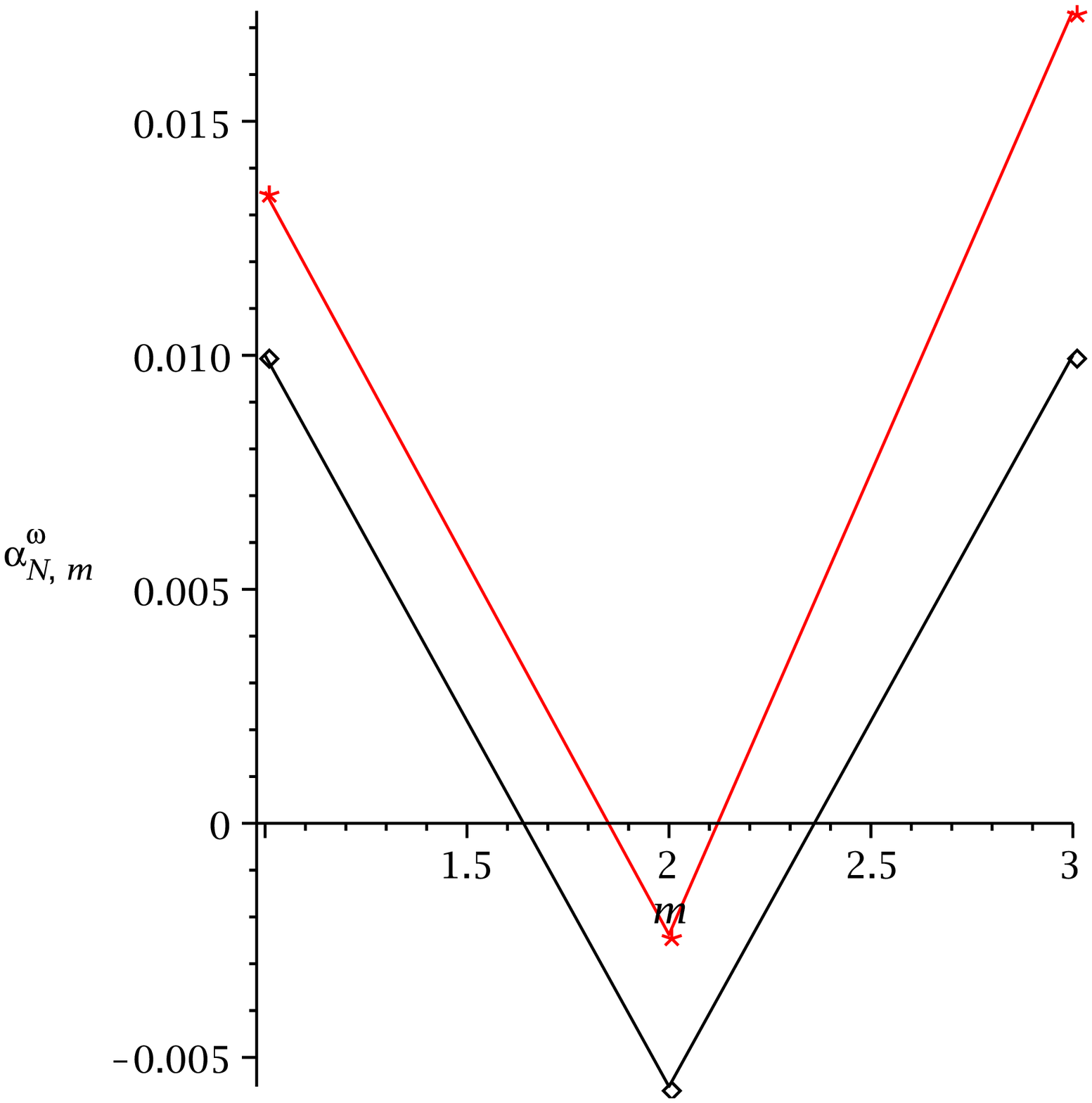}
		\hspace{12mm}
		\includegraphics[width=4.8cm]{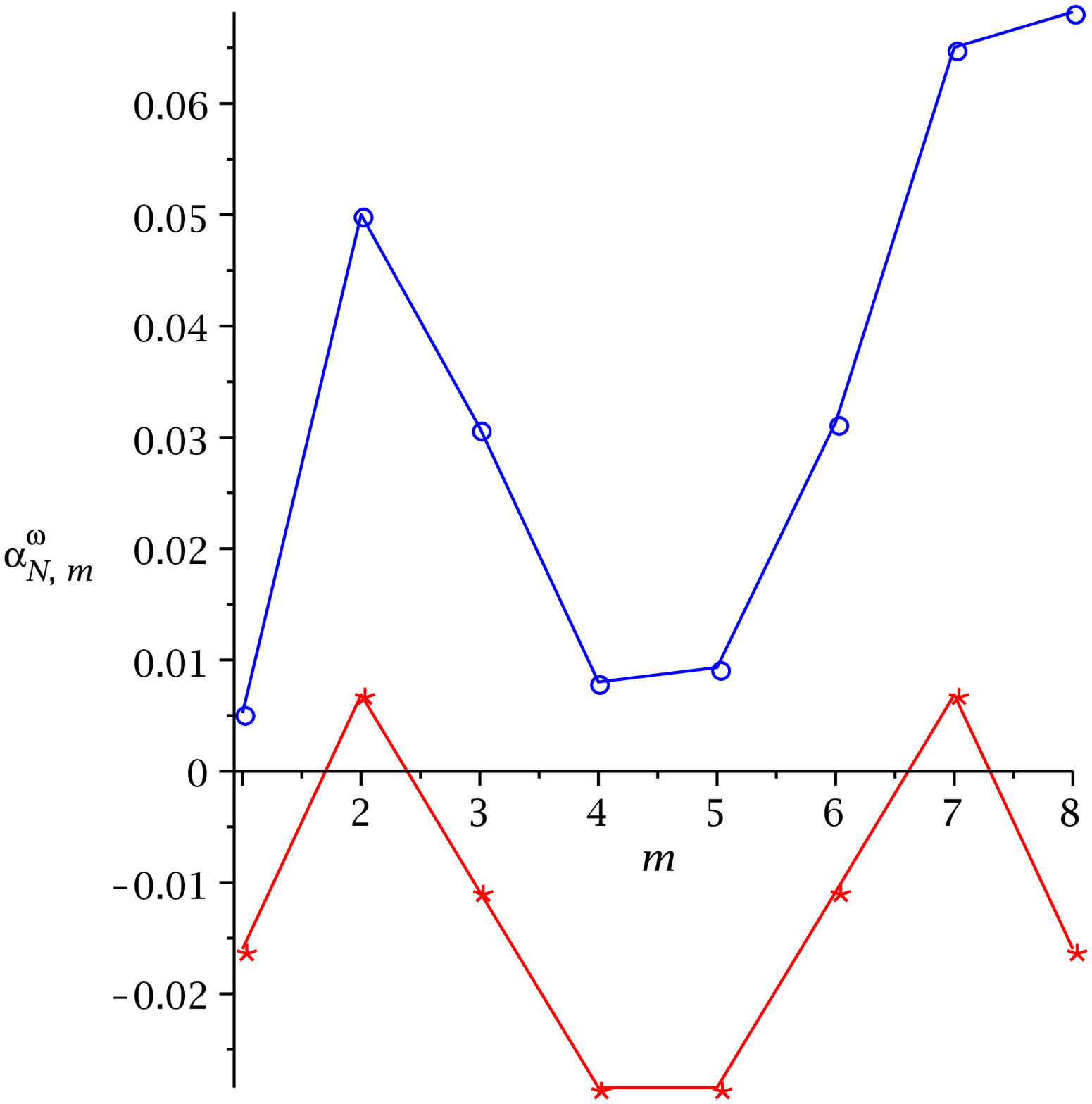}
	\end{center}
	\caption{Visualization of the corresponding $\alpha_{4,m}^1$, $m=1,\ldots,3$ values for $\beta_1$ (on the left with $\omega = 1,\, 1.01$) and $\alpha_{9,m}^1$, $m=1,\ldots,8$ for $\beta_2$ (on the right with $\omega = 1,\, 4/3$) from Example \ref{noMonotonicity_ex}.}
	\label{picture_noMonotonicity}  
\end{figure}

Example \ref{noMonotonicity_ex} shows that the desired monotonicity property does not hold for arbitrary $\mathcal{KL}_0$-functions $\beta$. However, the following lemmata show that monotonicity holds for $\beta$ of type \eqref{3:eq:exponential controllability} and at least for a subset of $\beta$ of type \eqref{3:eq:finite time controllability}. First, we adress monotonicity properties with respect to the control horizon $m$ for $\mathcal{KL}_0$-functions of type \eqref{3:eq:finite time controllability}. More precisely, we aim at establishing $\alpha_{N,m+1}^\omega \geq \alpha_{N,m}^\omega$. As seen in Example \ref{noMonotonicity_ex}, even in the setting without an additional weight on the final term it is not possible to prove this inequality for arbitrary $\mathcal{KL}_0$-functions of type \eqref{3:eq:finite time controllability}. Thus, the following lemma gives a complete analysis.
\begin{lemma}\label{monotonicity_omega_c1c2}
	Assume that the $\mathcal{KL}_0$-function $\beta$ is of type \eqref{3:eq:finite time controllability}. In addition, let $c_0,\, c_1 \geq 0$ and $c_n = 0$ for all $n \in \mathbb{N}_{\geq 2}$. Then $\alpha^\omega_{N,m+1} - \alpha^\omega_{N,m} \geq 0$ holds for $m \in \{1,\ldots,\lfloor N/2 \rfloor - 1\}$ in consideration of condition \eqref{3:eq:submultiplicativity}.
\end{lemma}
\begin{proof}
	Since $\gamma_{m+1} \geq \gamma_{i}$ for all $i \in \mathbb{N}_{\geq 3}$ we assume w.l.o.g. $\gamma_{m+1} - \omega > 0$. Thus, we use Formula \ref{alpha_formula} to show the assertion. Then, $\alpha_{N,m+1}^\omega - \alpha_{N,m}^\omega \geq 0$ is equivalent to
	{\small\begin{eqnarray}
		& & (\gamma_{m+1}-\omega)(\gamma_{m+2}-1) \left[ \prod_{i=m+2}^N \gamma_i - (\gamma_{m+2}-\omega) \prod_{i=m+3}^N (\gamma_i - 1) \right] \left[ \prod_{i=N-m}^N \gamma_i - \prod_{i=N-m}^N (\gamma_i-1) \right] \label{appendix_monotonicity_ineq1} \\
		& - & (\gamma_{m+2} - \omega)(\gamma_{N-m}-1) \left[ \prod_{i=m+1}^N \gamma_i - (\gamma_{m+1}-\omega) \prod_{i=m+2}^N (\gamma_i-1) \right] \left[ \prod_{i=N-m+1}^N \gamma_i - \prod_{i=N-m+1}^N (\gamma_i-1) \right] \geq 0. \nonumber
	\end{eqnarray}}
	Note that $\gamma_{3} = \gamma_{i}$ holds for all $i \in \mathbb{N}_{\geq 4}$ which implies $\gamma_{m+2} = \gamma_{N-m}$. Thus, the inequality in consideration simplifies to
	\begin{eqnarray*}
		& & \omega (\gamma_{m+1} - \gamma_{N-m}) \prod_{i=m+2}^N \gamma_i \left[ \prod_{i=N-m+1}^N \gamma_i - \prod_{i=N-m+1}^N (\gamma_i-1) \right] \\
		& + & \prod_{i=N-m+1}^N (\gamma_i-1) \prod_{i=N-m+1}^N \gamma_i (\gamma_{m+1}-\omega) \left[ \prod_{i=m+2}^{N-m} \gamma_i - (\gamma_{m+2}-\omega) \prod_{i=m+3}^{N-m} (\gamma_i - 1) \right] \geq 0
	\end{eqnarray*}
	and shows the assertion.
\end{proof}

Now, we aim at deriving monotonicity properties assuming exponential controllability. To this end, we consider inequality \eqref{appendix_monotonicity_ineq1}, i.e., $\alpha_{N,m+1}^\omega - \alpha_{N,m}^\omega$, which is in turn equivalent to

{\footnotesize\begin{equation}\label{appendix_exp_monotonicity_inequality1}
	 - a \prod_{i=m+2}^N \gamma_i + (a + (\gamma_{m+1} - \omega)(\gamma_{m+2} - 1)) \prod_{i=m+2}^{N-m} \gamma_i \prod_{i=N-m+1}^N (\gamma_i - 1) - (\gamma_{m+1} - \omega)(\gamma_{m+2} - \omega) \prod_{i=m+2}^N (\gamma_i - 1) \geq 0
\end{equation}}
with $a := - [ \omega \gamma_{N-m} (\gamma_{m+1} - \gamma_{m+2}) + \gamma_{m+1} (\gamma_{m+2} - \gamma_{N-m}) + \omega (\gamma_{N-m} - \gamma_{m+1}) ]$. The following corollary shows the desired monotonicity properties for the exponentially controllable case for final weights which are chosen sufficiently large, i.e., $\omega \geq (1-\sigma)^{-1}$.
\begin{corollary}\label{qualitativ_properties_monotonicity_eta}
	Let the $\mathcal{KL}_0$-function be of type \eqref{3:eq:exponential controllability} and $\eta := 1 + \sigma \omega - \omega \leq 0$. Then it holds $\alpha^\omega_{N,m+1} - \alpha^\omega_{N,m} \geq 0$ for $m \in \mathbb{N}_{\geq 1}$ satisfying $2m + 2 \leq N$, $N \in \mathbb{N}_{\geq 4}$.
\end{corollary}
\begin{proof}
	Using the equality $\gamma_{i+1} = C (\sum_{n=0}^{i-1} \sigma^n + \omega \sigma^{i} = C (1-\eta \sigma^i) (1-\sigma)^{-1}$ we obtain
	\begin{equation*}
		-a = \frac {-C \eta}{1-\sigma} \Big[ (C-1) \omega \sigma^m + \sigma^{N-m-1} \omega (1 - C \eta \sigma^{m}) + \gamma_{m+1} (\sigma^{m+1} - \sigma^{N-m-1}) \Big] \geq 0.
	\end{equation*}
	As a consequnece, the term $-a \prod_{i=m+2}^{N-m} \gamma_i (\prod_{i=N-m+1}^N \gamma_i - \prod_{i=N-m+1}^N (\gamma_i - 1))$ from \eqref{appendix_exp_monotonicity_inequality1} is positive. Since the remaining terms in \eqref{appendix_exp_monotonicity_inequality1} are also positive the assertion follows.
\end{proof}

The following lemma deals with the special case $N = 2m + 2$ and is proven in Subsection \ref{AppendixProofMonotonicity}.
\begin{lemma}\label{appendix_exp_monotonicity_induciton_start_lemma}
	Let the $\mathcal{KL}_0$-function be of type \eqref{3:eq:exponential controllability} and $\eta := 1 + \sigma \omega - \omega > 0$. Then it holds $\alpha^\omega_{N,m+1} - \alpha^\omega_{N,m} \geq 0$ for $m \in \mathbb{N}_{\geq 1}$ satisfying $2m + 2 = N$, $N \in \mathbb{N}_{\geq 4}$.
\end{lemma}

Assume that \eqref{appendix_exp_monotonicity_inequality1} holds for $N$. We aim at establishing the considered inequality for $N+1$ using the assertion for $N$ as an induction assumption. Moreover, $a$ is defined as above and $\tilde{a}$ stands for the same expression with $N+1$ instead of $N$. Then, it suffices to show
\begin{eqnarray*}
	& & (\gamma_{N+1} - 1)(\gamma_{m+1} - \omega) (\gamma_{m+2} - \omega) \prod_{m+2}^N (\gamma_i - 1) \\
	& \leq & (\gamma_{N+1} - 1) \left( -a \prod_{m+2}^N \gamma_i + (a + (\gamma_{m+1} - \omega)(\gamma_{m+2} - 1)) \prod_{m+2}^{N-m} \gamma_i \prod_{N-m+1}^N (\gamma_i - 1) \right) \\
	& \leq & -\tilde{a} \prod_{m+2}^{N+1} \gamma_i + (\tilde{a} + (\gamma_{m+1} - \omega)(\gamma_{m+2} - 1)) \prod_{m+2}^{N-m+1} \gamma_i \prod_{N-m+2}^{N+1} (\gamma_i - 1)
\end{eqnarray*}
where we have omitted the control variable. This in turn is equivalent to
{\small \begin{eqnarray*}
	& & [(\tilde{a} + (\gamma_{m+1}-\omega)(\gamma_{m+2}-1)) \gamma_{N-m+1} - (a+(\gamma_{m+1}-\omega)(\gamma_{m+2}-1))(\gamma_{N-m+1}-1)] \prod_{i=N-m+2}^{N+1} (\gamma_i - 1) \\
	& \geq & (\tilde{a} \gamma_{N+1} - a (\gamma_{N+1} - 1)) \prod_{i=N-m+1}^N \gamma_i.
\end{eqnarray*}}
Taking into account the following equations
\begin{eqnarray*}
	\tilde{a} + (\gamma_{m+1}-\omega)(\gamma_{m+2}-1) & = & (\gamma_{N-m+1}-1)(\omega \gamma_{m+2} - \omega \gamma_{m+1} + \gamma_{m+1} - \omega),\\
	-a - (\gamma_{m+1} - \omega)(\gamma_{m+2} - 1) & = & - (\gamma_{N-m}-1)(\omega \gamma_{m+2} - \omega \gamma_{m+1} + \gamma_{m+1} - \omega),\\
	\tilde{a} - a & = & (\gamma_{N-m+1} - \gamma_{N-m})(\omega \gamma_{m+2} - \omega \gamma_{m+1} + \gamma_{m+1} - \omega),\\
	a & = & \gamma_{N-m}(\omega \gamma_{m+2} - \omega \gamma_{m+1} + \gamma_{m+1} - \omega) - \gamma_{m+1}(\gamma_{m+2}-\omega).
\end{eqnarray*}
the inequality in consideration is equivalent to
{\footnotesize\begin{equation}\label{appendix_exp_monotonicity_inequality2}
	(\gamma_{N-m+1} - (\gamma_{N-m} - 1)) \prod_{N-m+1}^{N+1} (\gamma_i - 1) \geq \left( \gamma_{N-m+1} - \gamma_{N-m} + \frac {\gamma_{N-m}}{\gamma_{N+1}} - \frac {\gamma_{m+1}(\gamma_{m+2}-\omega)}{\gamma_{N+1}(\omega \gamma_{m+2} - \omega \gamma_{m+1} + \gamma_{m+1} - \omega)} \right) \prod_{N-m+1}^{N+1} \gamma_i.
\end{equation}}
\begin{lemma}\label{qualitativ_properties_mon_omega1}
	Let the $\mathcal{KL}_0$-function $\beta$ be of type \eqref{3:eq:exponential controllability} and $\omega = 1$. Then $\alpha^\omega_{N,m+1} - \alpha^\omega_{N,m} \geq 0$ holds for $N \in \mathbb{N}_{\geq 4}$, $m \in \{1,\ldots,\lfloor N/2 \rfloor - 1\}$.
\end{lemma}
\begin{proof}
	Due to Lemma \ref{appendix_exp_monotonicity_induciton_start_lemma}, it suffices to show inequality \eqref{appendix_exp_monotonicity_inequality2} to prove $\alpha^\omega_{N,m+1} \geq \alpha^\omega_{N,m}$. $\omega = 1$ implies $(\omega \gamma_{m+2} - \omega \gamma_{m+1} + \gamma_{m+1} - \omega) = (\gamma_{m+2} - 1)$ and $\eta = \sigma$. Thus, the right hand side of this inequality can be simplified significantly. Using the definition of $\gamma_i$, we obtain the equivalent inequality
	\begin{equation*}
		p(C):=(C + \sigma^{-(N-m)}) \prod_{i=N-m+1}^{N+1} \left( C - \frac {1-\sigma}{1-\sigma^i} \right) \geq C^{m+1} \left( C + \frac {\sigma^{m+1}-\sigma^{N-m}}{(1-\sigma^{N+1})\sigma^{N-m}} \right)=:q(C).
	\end{equation*}
	Proposition \ref{exp_stab_C1} provides $p(1) = q(1)$. Furthermore, note that $p$ has exactly one negative root and $m+1$ strictly positive roots which are located in the open interval $(0,1)$. Additionaly, $q$ can be represented as $C^{m+1}(C+c)$ with $c > 0$. Thus, the only condition needed to be verified for the application of Lemma \ref{lemma_polynomial} is the one with respect to the ${(m+1)}^{st}$-derivative. To this end, we have to show the inequality
	\begin{equation*}
		\sigma^{-(N-m)} - \sum_{i=N-m+1}^{N+1} \frac {1-\sigma}{1-\sigma^i} > \frac {\sigma^{m+1}-\sigma^{N-m}}{(1-\sigma^{N+1})\sigma^{N-m}}
	\end{equation*}
	which is equivalent to
	\begin{equation*}
		\underbrace{\frac {1-\sigma^{m+1}}{1-\sigma}}_{> (m+1) \sigma^m} (1+\sigma^{N-m}) > (1-\sigma^{N+1}) \sigma^{N-m} \underbrace{\sum_{i=N-m+1}^{N+1} \frac 1 {1-\sigma^i}}_{< (m+1) \frac 1 {1-\sigma^{N-m+1}}}.
	\end{equation*}
	Hence, the inequality
	\begin{equation*}
		(1-\sigma^{N-m+1}) \sigma^m (1+\sigma^{N-m}) \geq \sigma^m - \sigma^{2N - m + 1} \geq \sigma^{N-m} - \sigma^{N-m} \sigma^{N+1} = \sigma^{N-m} (1-\sigma^{N+1})
	\end{equation*}
	implies the assertion.
\end{proof}

\section{Quantitative characteristics of $\alpha_{N,m}^\omega$ depending on varying control horizon $m$}\label{alphasec3}

In the previous section we derived qualitative characteristics of the obtained stability bounds $\alpha_{N,m}^\omega$ with respect to the control horizon $m$, e.g. symmetry or monotonicity properties, which can be exploited according to Theorem \ref{thm:stabmain}. Nevertheless, the decisive question remains whether we are able to guarantee stability or not, i.e., whether the sign of the $\alpha_{N,m}^\omega$-values corresponding to a given $\mathcal{KL}_0$-function $\beta(r,t)$ is positive or not. Clearly, choosing the optimization horizon $N$ sufficiently large ensures stability, cf. Corollary \ref{convergenceNtoINFINITY}. On the other hand, enlarging the optimization horizon $N$ implies higher computational costs in order to find an optimal control with respect to \eqref{2:eq:finite cost functional}. In contrast to that, changing the control horizon $m$ or adding an additional weight on the final term does not significantly change the effort needed to solve the finite horizon optmization problem induced by \eqref{2:eq:finite cost functional}. Hence, we aim at analyzing the influence of these set screws more closely.

\subsection{Control horizon}
Assuming exponential controllability, i.e., a $\mathcal{KL}_0$-function of type \eqref{3:eq:exponential controllability}, Theorem \ref{alpha_formula_thm} enables us to determine the maximal feasible overshoot $C$ for a given decay rate $\sigma$ ensuring stablity, i.e., a positive $\alpha_{N,m}^\omega$-value. For simplicity of exposition, we restrict ourselves to the case without an additional weight on the final term, i.e., $\omega = 1$. In order to determine the impact of the control horizon $m$ for a given optimization horizon $N$, we carefully examine the stability region, i.e., the set of all parameter combinations $C \geq 1$, $\sigma \in (0,1)$ guaranteeing stability of the underlying discrete time systems. Due to the assertion of Corollary \ref{alpha_symm_cor} it suffices to deal with $m \in \{1,\ldots,\lfloor N/2 \rfloor\}$. Exemplarily, Figure \ref{plot_stability_region} shows the stability regions for $N=7$ and $N=11$ respectively.
\begin{figure}[!ht]
	\begin{center}
		\includegraphics[width=6.cm]{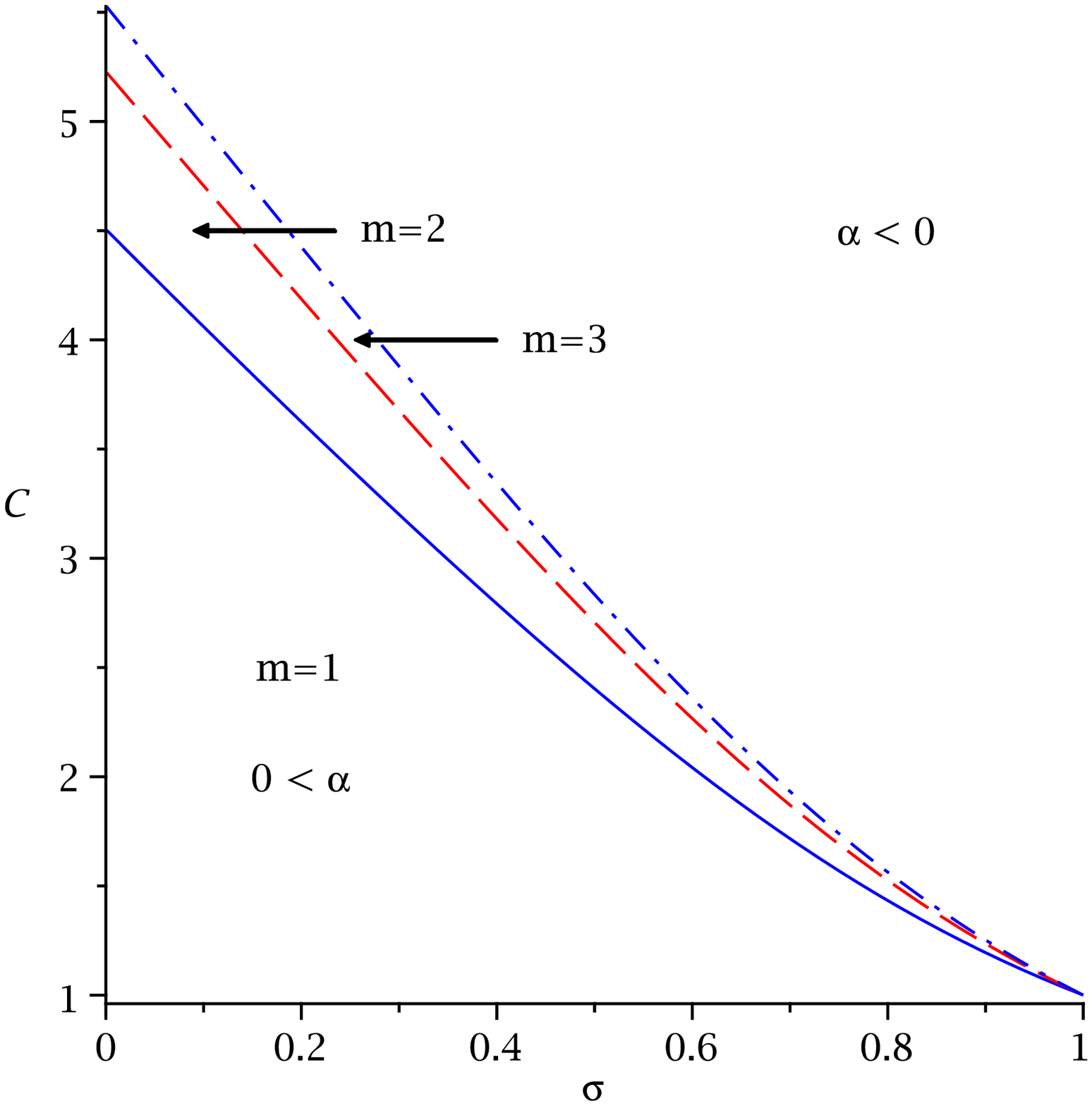}
		\hspace{5mm}
		\includegraphics[width=6.cm]{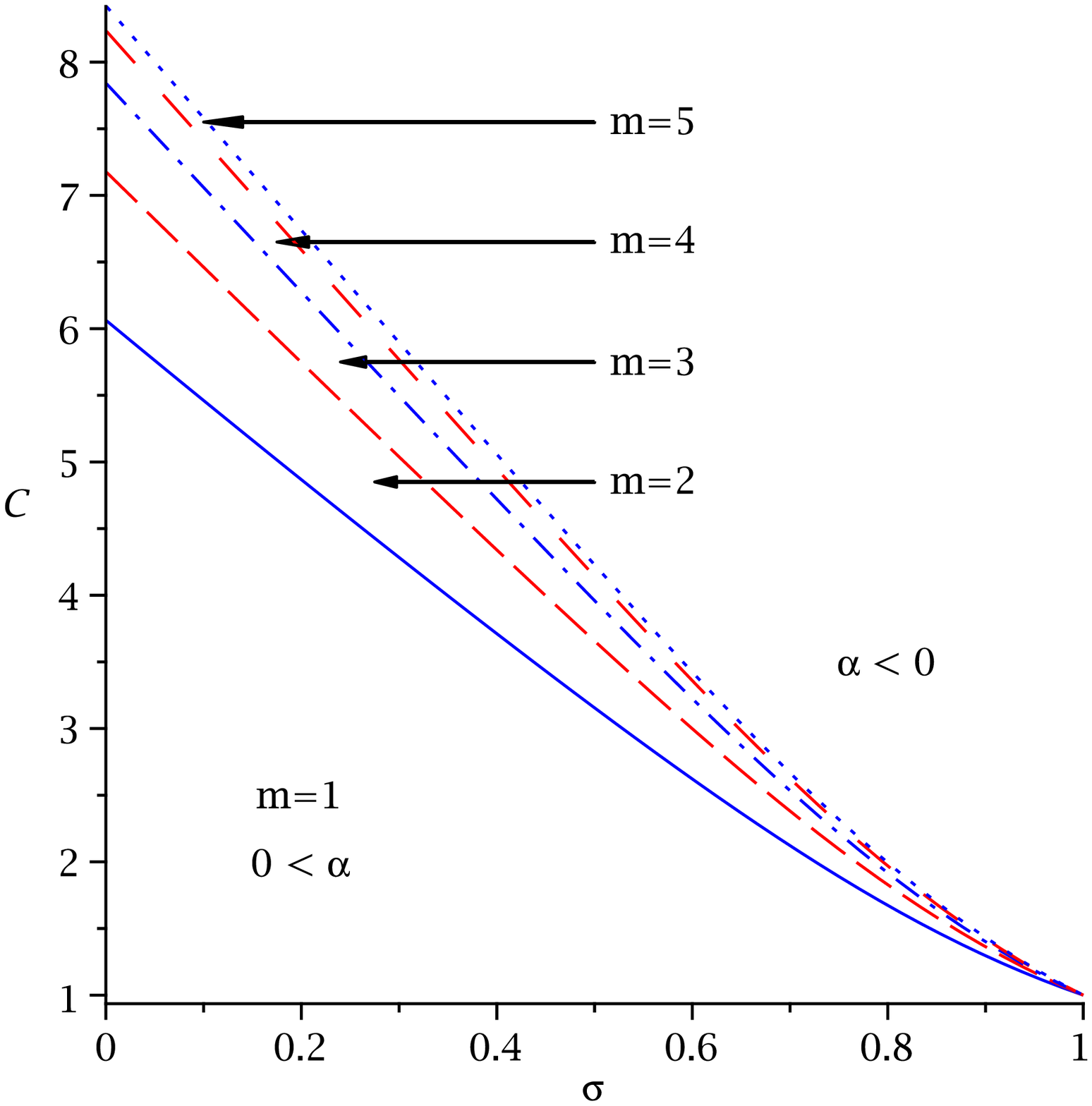}
	\end{center}
	\caption{Illustration of parameter combinations $(C,\sigma)$ ensuring stability in dependency of the control horizon $m$ for optimization horizons $N=7$ and $N=11$ respectively.}
	\label{plot_stability_region}
\end{figure}

Apparently, increasing the control horizon $m$ enlarges the stability region, i.e., allows for larger overshoots $C$ for given decay rates $\sigma$. To be more precise, incrementing the control horizon augments the area of interest. Here, the monotonicity property according to Lemma \ref{qualitativ_properties_mon_omega1} is reflected. Moreover, we are able to quantify this aspect. E.g., for $N=7$, the area containing feasible $(C,\sigma)$ pairs is scaled up by $21$ and $30$ percent. For longer optimization horizons increasing the control horizon enhances the attainable gain even further, e.g. $m=2$ and $m=5$ enlarge the stability region by $23$ and $48$ percent respectively.

As we have seen, the exponential controllable case provides various desirable properties, cf. Theorem \ref{thm:stabmain}, which do not hold for finite time controllability, i.e., $\mathcal{KL}_0$-functions of type \eqref{3:eq:finite time controllability} satisfying \eqref{3:eq:submultiplicativity}, in general. Still, we expended the effort to give a complete characterization referring to this setting, cf. Lemmata \ref{symmetry_omega_c1c2c3} and \ref{monotonicity_omega_c1c2}. The reason for putting so much emphasis on this is given in the following example.
\begin{example}
\label{QuantitativePropertiesComparison}
	Consider the $\mathcal{KL}_0$-function of type \eqref{3:eq:finite time controllability} defined by $c_0 = 5/2,\, c_1 = 2,\, c_2 = 3/2,\, c_3 = 32/25,\, c_4 = 1,\, c_5 = 1/2,\, c_6 = 1/8$, and $c_i = 0$ for all $i \in \mathbb{N}_{\geq 7}$. We construct an upper bound by choosing $C=5/2$ and $\sigma=4/5$, i.e., a $\mathcal{KL}$-function of type \eqref{3:eq:exponential controllability}, cf. Figure \ref{comparison} on the right. Although this seems to be a good approximation, the corresponding $\alpha_{N,\cdot}^\omega$-values are significantly worse, cf. Figure \ref{comparison} on the left. E.g., in contrast to the finite time controllable case, stability cannot be guaranteed for control horizons $m \in \{2,3,4,12,13,14\}$ using exponential controllability.
\end{example}

Thus, it is in general favorable to work with a $\mathcal{KL}_0$-function ensuring finite time controllability in contrast to using an upper bound provided by an estimated $\mathcal{KL}$-function of type \eqref{3:eq:exponential controllability}. Particularly, this is the case since positivity of the needed $\alpha_{N,\cdot}^\omega$-values can easily be checked by means of Theorem \ref{alpha_formula_thm}. Moreover, note that even for $\mathcal{KL}_0$-functions which do not satisfy the assumptions of Theorem \ref{thm:stabmain}, the assertion with respect to symmetry and monotonicity often holds, cf. Figure \ref{comparison}.

\begin{figure}[!ht]
	\begin{center}
		\includegraphics[width=4.8cm]{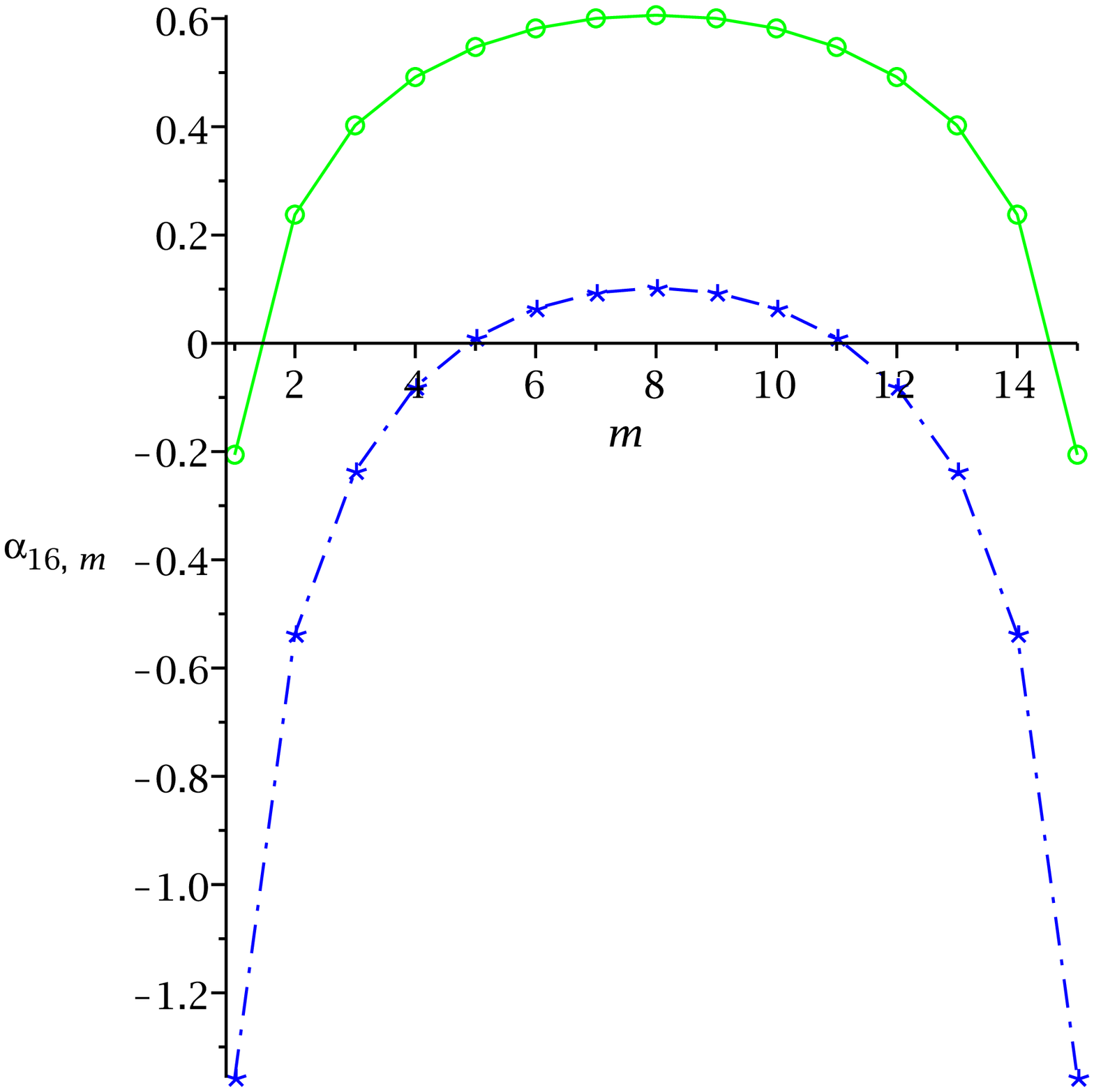}
		\hspace{12mm}
		\includegraphics[width=4.8cm]{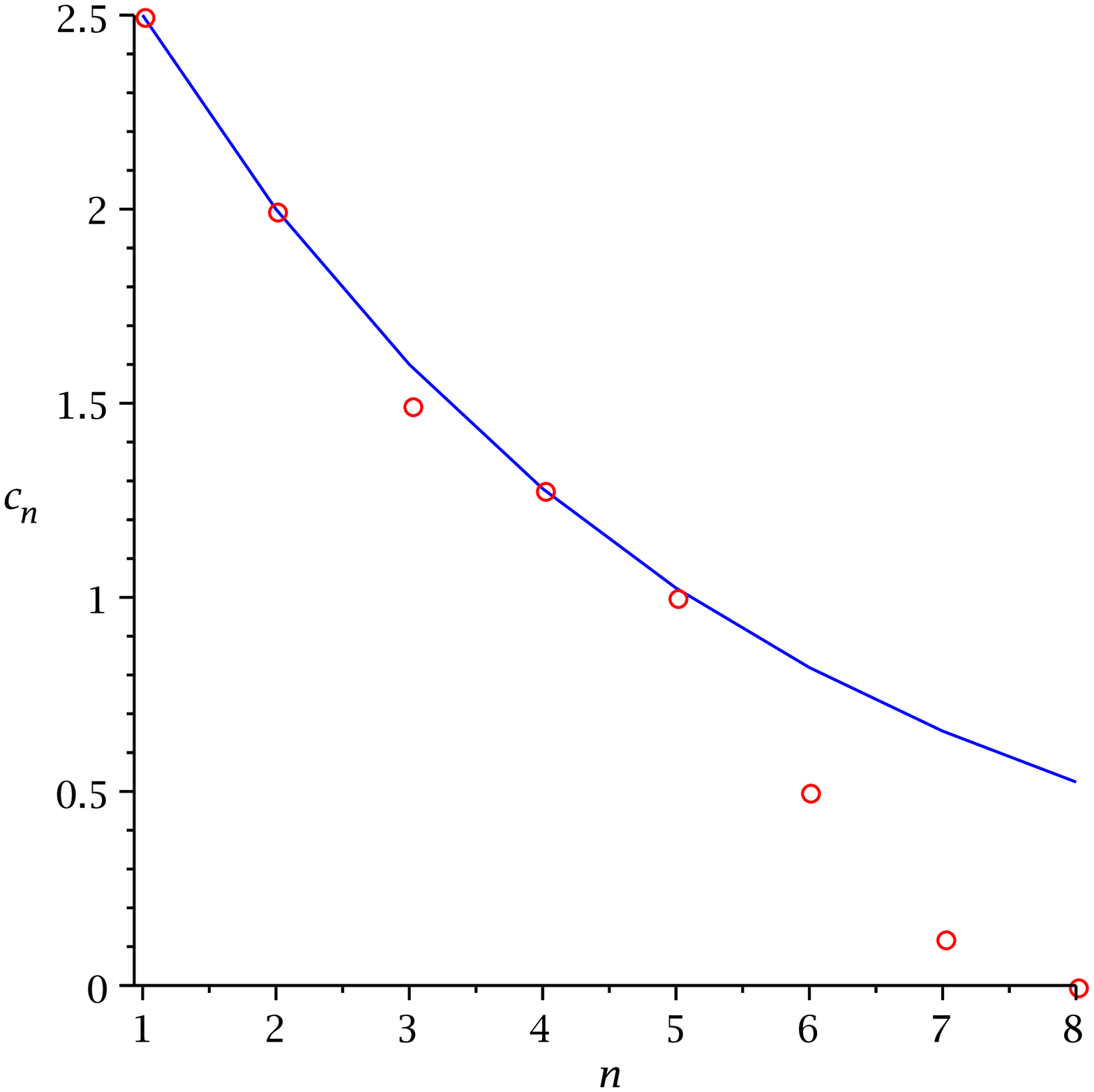}
	\end{center}
	\caption{Comparison of $\alpha_{16,\cdot}^1$ for the $\mathcal{KL}_0$-functions of type (\ref{3:eq:exponential controllability},$\circ$) and (\ref{3:eq:finite time controllability},$\star$) respectively. On the right we depict the corresponding $\mathcal{KL}_0$-functions.}
	\label{comparison}
\end{figure}

\subsection{Influence of an additional weight on the final term}

In order to evaluate the benefit provided by a final weight, we discern whether the coefficient $c_m$ is strictly smaller than one or not. Since $c_m < 1$ allows for guaranteeing the necessary and sufficient condition $\gamma_{m+1} - \omega \leq 0$, cf. Theorem \ref{alpha_formula_thm}, for sufficiently large $\omega$, adding an appropriate additional weight on the final term enables us to ensure stability. Moreover, note that the probability of being able to fulfill this condition increases with longer control horizons in general, cf. $\mathcal{KL}_0$-functions of type \eqref{3:eq:exponential controllability}.

However, without this condition being satisfied analyzing the effects of including a final weight is much more subtle. Thus, we start our investigation by looking at the following example which demonstrates positive effects of adding a final weight to the considered setting.
\begin{example} \label{example_impact_final_weight1}
	We assume finite time controllability characterized by the coefficients $c_0=1,\, c_1=3/2,\, c_2=39/20,\, c_4=7/5$, and $c_i=0$ for all $i \in \mathbb{N}_{\geq 5}$ which ensure \eqref{3:eq:submultiplicativity}. The resulting $\alpha_{7,\cdot}^\omega$-values for $\omega = 1$ and $\omega = 3/2$ are illustrated in Fig. \ref{impact_final_weight}.
\end{example}

At first, note that the generalized symmetry as well as the monotonicity property hold in Example \ref{example_impact_final_weight1} although the assumptions of Lemmata \ref{symmetry_omega_c1c2c3} and \ref{monotonicity_omega_c1c2} are not satisfied $(c_3,\, c_4 > 0)$. Furthermore, we observe the interplay of adding a final weight and the two mentioned properties. Together, these adjusting screws imply our stability condition for $m=4$ which is not fulfilled for $\omega = 1$.

The next example points out a possible pitfall as well as an adequate approach to overcome it.
\begin{example}
\label{example_impact_final_weight2}
	Again, we assume finite time controllability, i.e., a $\mathcal{KL}_0$-function of type \eqref{3:eq:finite time controllability} given by $c_0=1,\, c_1=3/2,\, c_2=2/3,\, c_3=1$, and $c_i=0$ for all $i \in \mathbb{N}_{\geq 4}$. Note that these coefficients satisfy condition \eqref{3:eq:submultiplicativity}. In Fig. \ref{impact_final_weight} we have depicted $\alpha_{5,m}^\omega$, $m=1,2,3,4$ with several different additional weights on the final term.
\end{example}
\begin{figure}[!ht]
	\begin{center}
		\includegraphics[width=4.8cm]{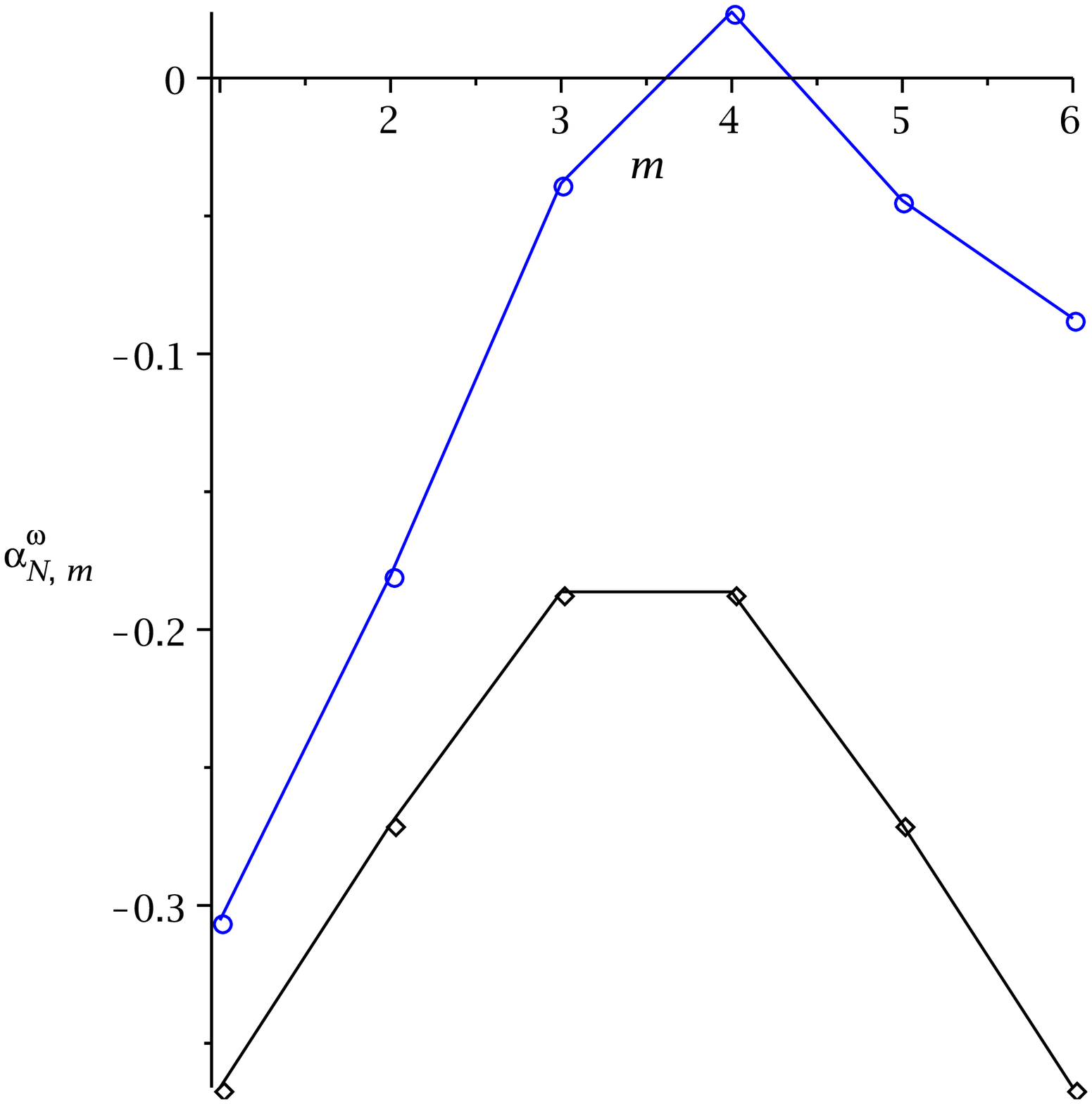}
		\hspace{12mm}
		\includegraphics[width=4.8cm]{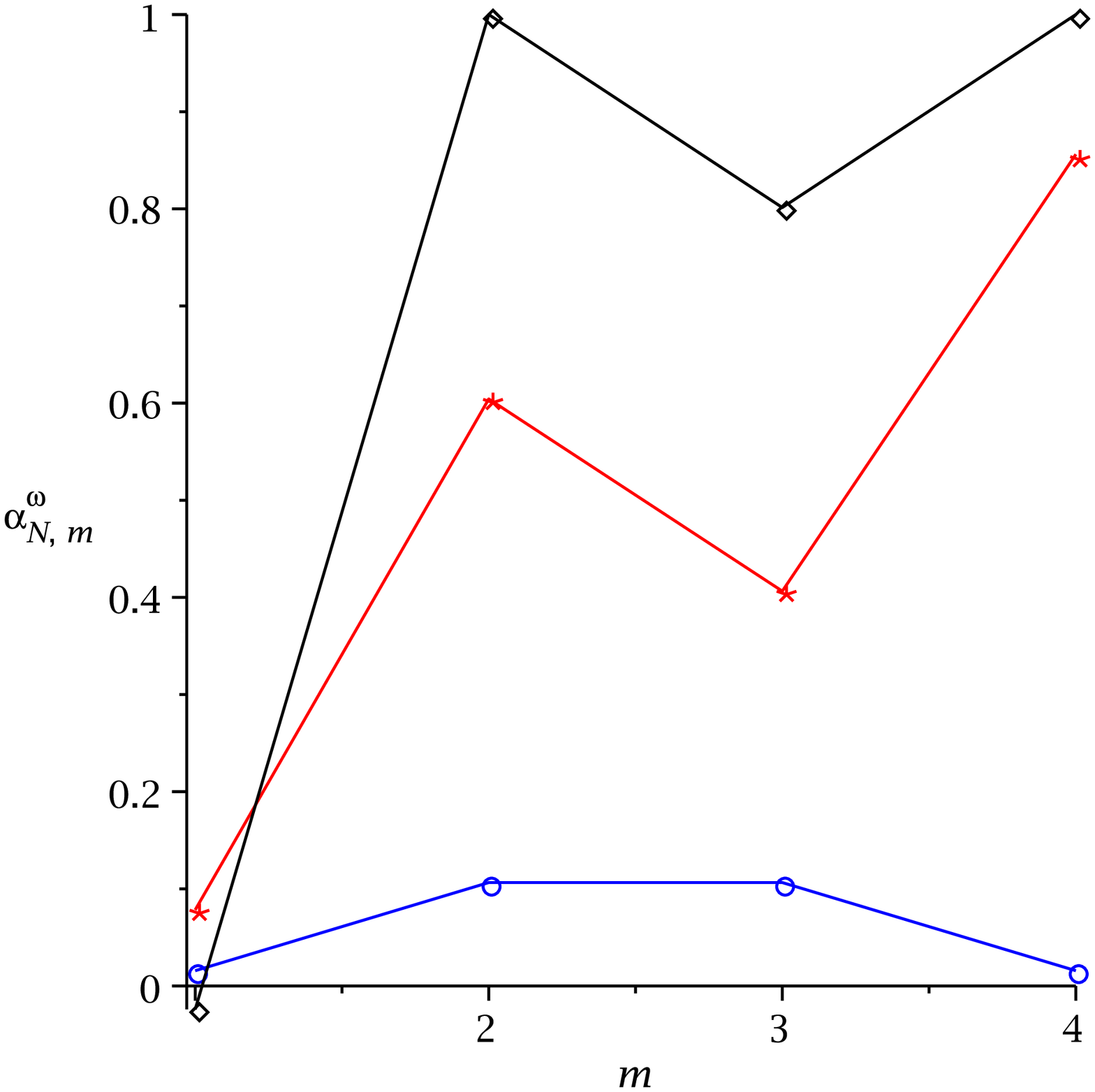}
	\end{center}
	\caption{$\alpha_{7,\cdot}^\omega$ for the $\mathcal{KL}_0$-function from Example \ref{example_impact_final_weight1} for $\omega = 1$ ($\diamond$) and $\omega=3/2$ ($\circ$). On the right $\alpha_{5,\cdot}^\omega$ corresponding to Example \ref{example_impact_final_weight2} for $\omega = 1,\, 7/2$, and $55/2$ ($\circ,*,\diamond$).}
	\label{impact_final_weight}
\end{figure}

Although increasing $\omega$ seems to significantly improve the guaranteed stability behaviour in general, an additional weight on the final term -- chosen too big -- may even invalidate our stability condition, cf. Figure \ref{impact_final_weight}. However, in the szenario of Example \ref{example_impact_final_weight2} we are able to compensate this drawback by shifting to a larger control horizon. Once more, we stress the fact that Theorem \ref{alpha_formula_thm} allows for easily calculating the $\alpha_{N,m}^\omega$-values which has to be taken into account.

\section{Example}\label{exsec}
In this section we compare our analytical results with numerical MPC simulations. The first example is a linear inverted pendulum which is solved for a grid of initial values. The second example is a nonlinear form of the inverted pendulum. We use it to show that enlarging the control horizon $m$ exhibits a stabilizing effect even for long optimization horizons $N$. Since numerical optimization for large control horizons appears to be difficult, we added a third example, a nonlinear arm--rotor--platform model, which allows us to show numerical results for $m=1, \ldots, N-1$.
\subsection{Linear Inverted Pendulum}
Our first example is a linear inverted pendulum on a cart given by
\begin{equation*}
	\dot{x}(t) = \left( \begin{array}{cccc} 0 & 1 & 0 & 0 \\ g &
            -k & 0 & 0 \\ 0 & 0 & 0 & 1 \\ 0 & 0 & 0 & 0 \end{array}
        \right) x + \left( \begin{array}{c} 0 \\ 1 \\ 0 \\ 1
          \end{array} \right) u, 
\end{equation*}
in which we want to stabilize the upright position $x^\star = (0, 0, 0, 0)$ using linear MPC. Here, we used the optimization horizon $N = 10$, the sampling interval $T = 0.7$ and the cost functional $J_N(x_0, u) = \sum_{n = 0}^{N - 1} \|Qx_u(n)\|_1 + \|Ru(n)\|_1$ with $Q = 2 \, \text{Id}$ and $R = 4\, \text{Id}$. Moreover, we use the constants $g = 9.81$ and $k = 0.1$ for gravitation and friction respectively.\\
Since Assumption \ref{3:ass:controllability} exhibits a set--valued nature we consider a uniform grid $\mathcal{G}$ of initial values from the set $[-0.05, 0.05]^2 \times [-1, 1] \times [-0.05, 0.05]$ which contains the origin $x^\star$. For each $m=1,\ldots,9$ we simulate the MPC closed loop trajectories $x_{\mu_{N,m}}$ with control horizon $m_i\equiv m$ and initial value $x \in \mathcal{G}$. Along each trajectory we then compute $\alpha_{N,m}^{\omega}$ as the minimum of the suboptimality degrees from Formula \eqref{2:prop:m-step suboptimality estimate:eq1} applied with $x_0=x_{\mu_{N,m}}(n)$, $n=0,m,2m,\ldots,19$. A selection of these values is plotted in Fig.\ \ref{5:fig:curve}, in which each dashed line represents the values $\alpha_{N,1}^{\omega},\ldots,\alpha_{N,N-1}^{\omega}$ for a corresponding initial value. In addition, the minima over all trajectories are plotted as a solid line.
\begin{figure}[!ht]
	\begin{center}
		\includegraphics[width=0.48\textwidth]{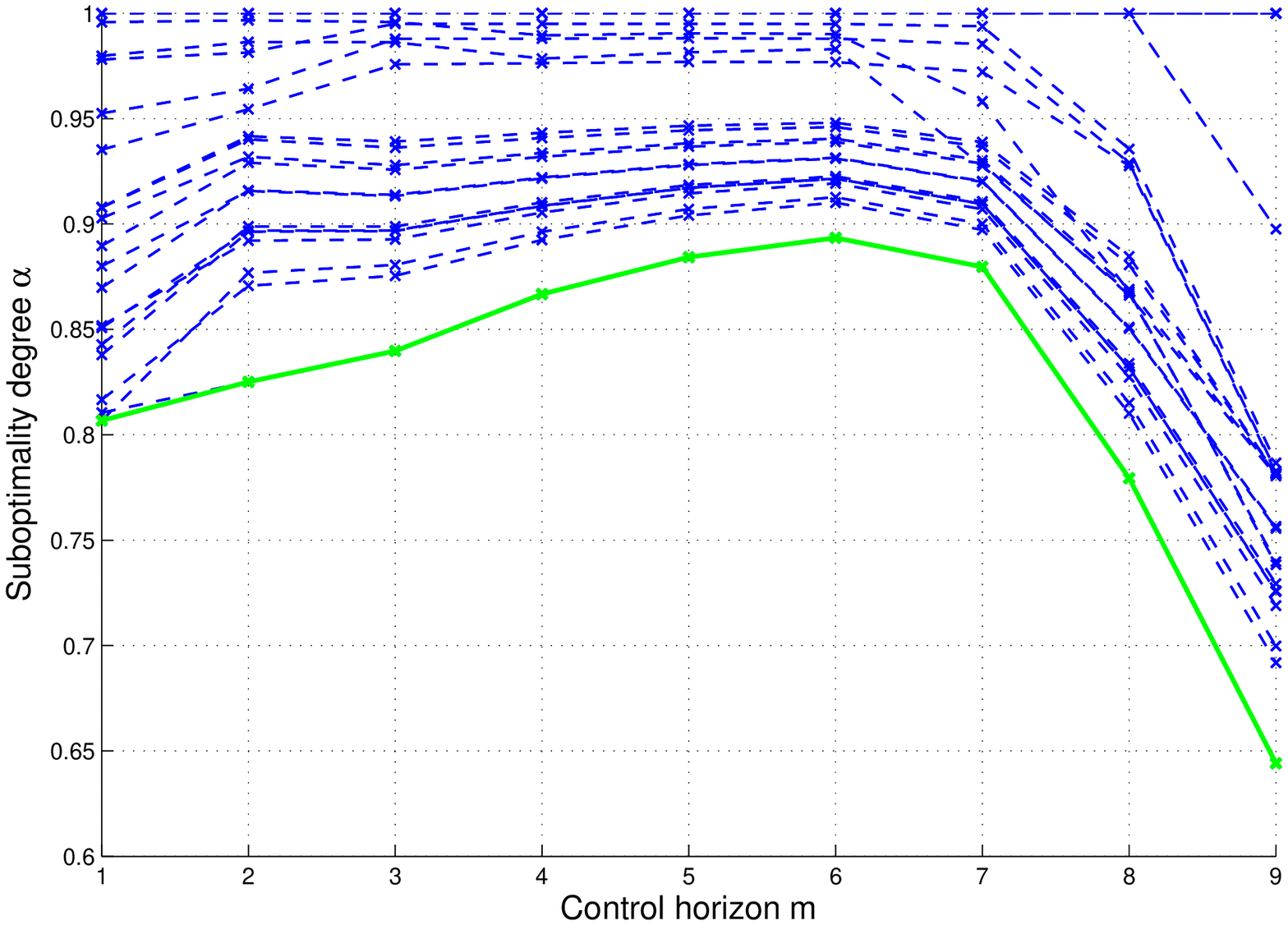}
		\hfill
		\includegraphics[width=0.48\textwidth]{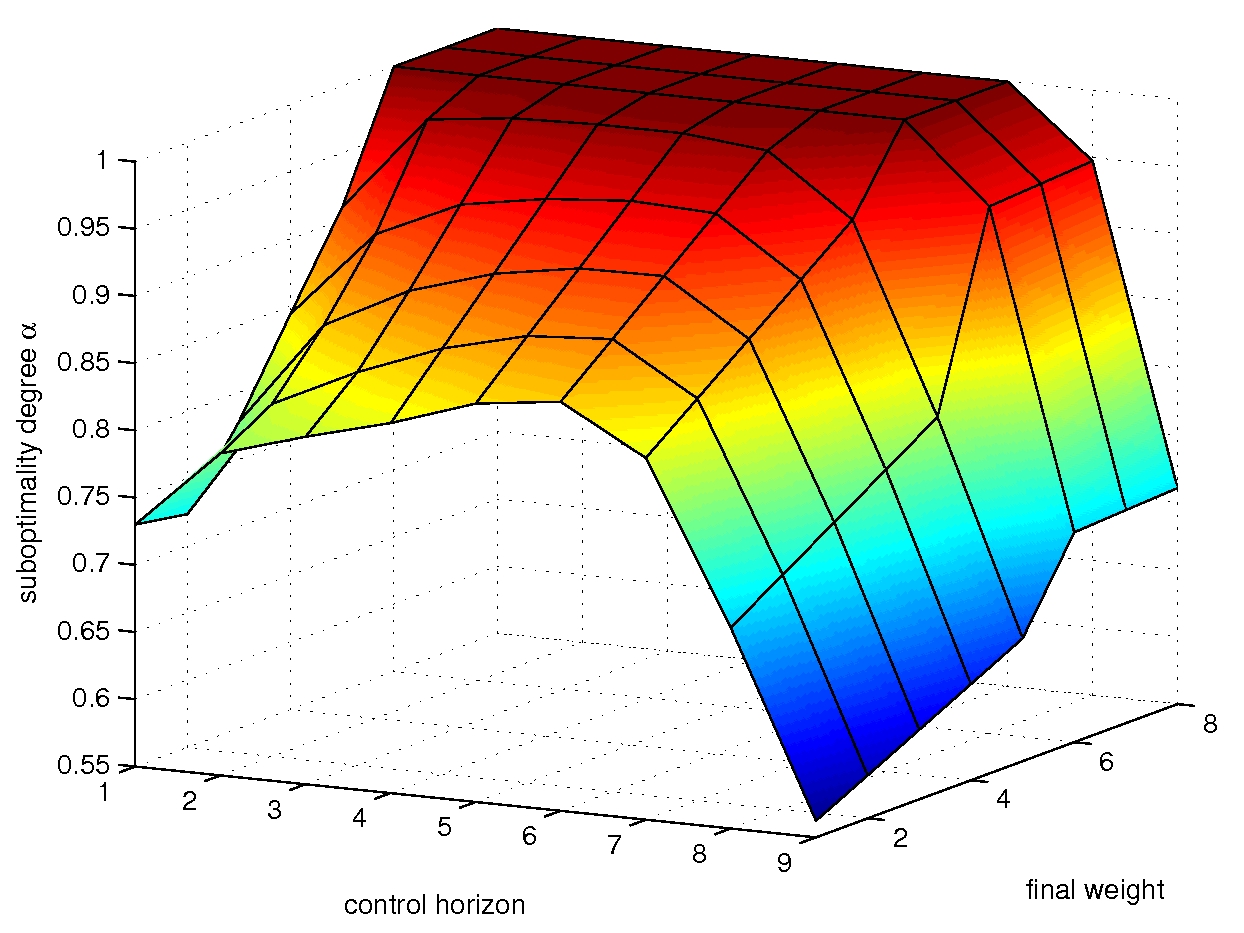}
	\end{center}
	\caption{Approximation of $\alpha_{10,m}^\omega$, $\omega \in \{1, \ldots, 8\}$, for the linear inverted pendulum.}
	\label{5:fig:curve}
\end{figure}

The results indicate that the closed loop is asymptotically stable for each $m_i$ and confirm that choosing control horizons $m_i > 1$ may indeed improve the suboptimality bound. Moreover, it is interesting to compare Fig.\ \ref{5:fig:curve} with Fig.\ \ref{typical_values}. While Fig.\ \ref{typical_values} shows the minimal $\alpha$-values for a set of exponentially controllable systems, the curves in Fig.\ \ref{5:fig:curve} represent the numerically computed $\alpha$-values for one particular system and a grid of initial values. Yet, the curves in Fig.\ \ref{5:fig:curve} at least approximately resemble the shape of the curves in Fig.\ \ref{typical_values}.\\
The second part of Fig.\ \ref{typical_values} treats the enlarged weights on the final term. Comparing the presented analytical results to our simulations depicted in the right part of Fig.\ \ref{5:fig:curve} we obtain the same tendency: If the weight on the final term is increased, then the degree of suboptimality $\alpha$ is growing faster for small control horizons $m$.
\subsection{Nonlinear Inverted Pendulum}
In order to show that similar effects can be experienced for nonlinear control problems, we consider the presented inverted pendulum on a cart problem in the nonlinear form
\begin{align*}
	\dot{x}_1(t) & = x_2(t)\\
	\dot{x}_2(t) & = -\frac{g}{l}\sin(x_1(t) + \pi) - \frac{k_L}{l} x_2(t) | x_2(t) | - u(t) \cos(x_1(t) + \pi) - k_R \mbox{sgn}(x_2(t))\\
	\dot{x}_3(t) & = x_4(t)\\
	\dot{x}_4(t) & = u(t)
\end{align*}
with gravitation constant $g = 9.81$, length $l = 10$ and friction terms $k_R = k_L = 0.01$. Again, we aim at stabilizing the upright position $x^\star = (0, 0, 0, 0)$. Here, we consider the stage cost
\begin{align*}
	l(x(i), u(i)) := & \int\limits_{t_i}^{t_{i+1}} 10^{-4} u(t)^2 + \Big( 3.51 \sin(x_1(t))^2 + 4.82 \sin(x_1(t)) x_2(t) + 2.31 x_2(t)^2 \\
	& \quad + 0.1 \left( (1 - \cos(x_1(t))) \cdot (1 + \cos(x_2(t))^2) \right)^2 + 0.01 x_3(t)^2 + 0.1 x_4(t)^2 \Big)^2 dt
\end{align*}
which gives us the cost functional
\begin{align*}
	J_N(x_0, u) = &\sum\limits_{i = 0}^{N - 2} l(x(i), u(i)) + \omega l(x(N-1), u(N-1))
\end{align*}
over an optimization horizon $N = 70$ and sampling length $T = 0.05$. Moreover, the tolerances of the optimization routine and the differential equation solver are set to $10^{-6}$ and $10^{-7}$ respectively. Since this cost functional is $2 \pi$--periodic, we add box--constraints limiting $X$ to the interval $[-2 \pi + 0.01, 2 \pi - 0.01]$. Using the initial value $x_0 = (\pi + 1.5, 0, 0, 0)$, we simulated the MPC closed--loop trajectories $x_{\mu_{N,m}}$ for $m = 1, \ldots, 10$ and final weights $\omega = 1, \ldots, 10$ with control horizon $m_i \equiv m$.\\
Since the numerical optimization is not reliable on start of the NMPC algorithm, a startup sequence of 20 NMPC iterations using $m=1$ is implemented to obtain an initial guess of the control which can be assumed to be close to the global optimum. Moreover, the problem appears to be practically stabilizable only. To compensate this issue, we used the estimation constant $\varepsilon = 10^{-4}$, cf. \cite[Theorem 21]{GruP09} for details.\\
In Fig.\ \ref{fig:nonlinear}, the minimal $\alpha$--value along a closed--loop trajectory is shown for a variety of final weights $\omega$ and control horizons $m$.

\begin{figure}[!ht]
	\begin{center}
		\includegraphics[width=0.48\textwidth]{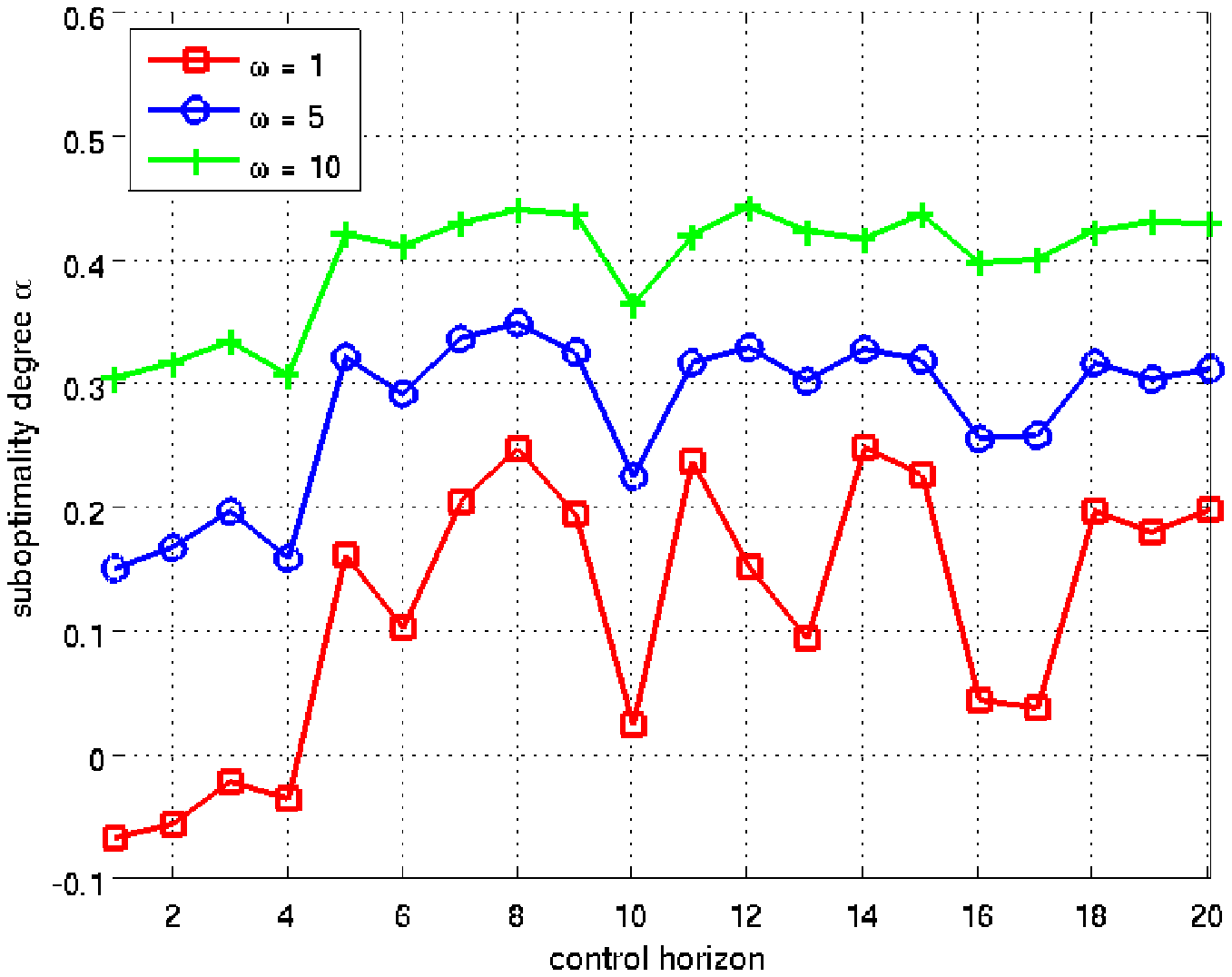}
		\hfill
		\includegraphics[width=0.48\textwidth]{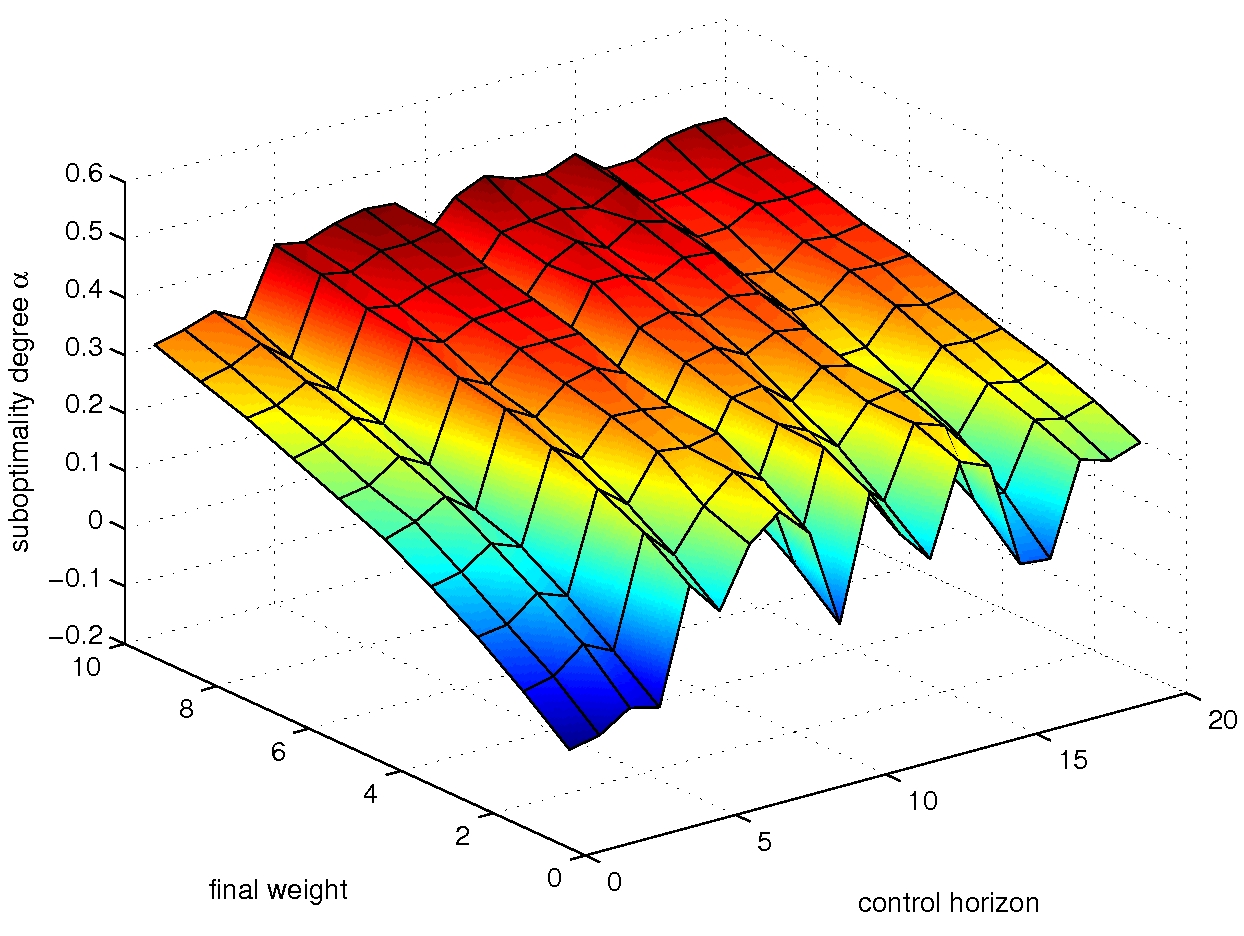}
	\end{center}
	\caption{Approximation of $\alpha_{70,m}^{\omega}$, $m \in \{1, \ldots, 10\}$, for the nonlinear inverted pendulum.}
	\label{fig:nonlinear}
\end{figure}

For this simulation, we were able to increase the range of $m$ from 3 to 20 for which acceptable $\alpha$ values can be computed by internally repeating the optimization at every sampling point. This allows us to keep track of a suitable initial guess of the control while the system is running in open--loop. For $m \geq 21$ we experience numerical problems during the optimization and although most of the resulting trajectories converge to the target point, we cannot see this from our estimates.\\
Note that for $\omega = 1$ the $\alpha$ values are negative for control horizons $m = 1, \ldots, 4$. Still, larger control horizons exhibit a positive $\alpha$ value such that stability is guaranteed. Additionally, an increase in $\alpha$ can be experienced for all control horizons $m$ considered in this example if $\omega$ is increased. This corresponds to the known stabilizing effect of terminal costs.\\
Since computing these estimates can be done with very small additional effort compared to the MPC procedure, such an analysis should be done before enlarging the optimization horizon.
\subsection{Nonlinear Arm--Rotor--Platform Modell}
Last, we consider an arm/\-rotor/\-platform (ARP) model:
\begin{align*}
	\dot{x}_{1}(t) & = x_{2}(t) + x_{6}(t) x_{3}(t) \\
	\dot{x}_{2}(t) & = -\frac{k_{1}}{M} x_{1}(t) - \frac{b_{1}}{M} x_{2}(t) + x_{6}(t) x_{4}(t) - \frac{m r b_{1}}{M^{2}} x_{6}(t) \\
	\dot{x}_{3}(t) & = - x_{6}(t) x_{1}(t) + x_{4}(t) \displaybreak[0] \\
	\dot{x}_{4}(t) & = - x_{6}(t) x_{2}(t) - \frac{k_{1}}{M} x_{3}(t) - \frac{b_{1}}{M} x_{4}(t) + \frac{m r k_{1}}{M^{2}} \\
	\dot{x}_{5}(t) & = x_{6}(t) \\
	\dot{x}_{6}(t) & = -a_{1} x_{5}(t) - a_{2} x_{6}(t) + a_{1} x_{7}(t) + a_{3} x_{8}(t) - p_{1} x_{1}(t) - p_{2} x_{2}(t) \\
	\dot{x}_{7}(t) & = x_{8}(t) \\
	\dot{x}_{8}(t) & = a_{4} x_{5}(t) + a_{5} x_{6}(t) - a_{4} x_{7}(t) - (a_{5} + a_{6}) x_{8}(t) + \frac{1}{J} u(t)
\end{align*}
For details on the specification of the model parameters we refer to \cite[Chapter 7.3.2]{FK96}. \\
For this example, we fix the initial values to $x(t_0) = (0, 0, 0, 0, 10, 0, 0, 0)$, the absolute and relative tolerances for the solver of the differential equation both to $10^{-10}$, the length of the open--loop horizon within the MPC--algorithm to $H = N \cdot T$ with $N = 13$ and sampling period $T = 0.05$, and set the optimality tolerance of the SQP solver to $10^{-8}$. Moreover, the cost functional is given by
\begin{align*}
	J(x, u) = \sum\limits_{j = 0}^{N - 2} \int_{t_j}^{t_{j+1}} \| x(t) \| + u(t)^2 dt + \omega \int_{t_{N-1}}^{t_{N}} \| x(t) \| + u(t)^2 dt
\end{align*}
and the practical region of the equilibrium is estimated using the constant $\varepsilon = 5 \cdot 10^{-7}$. Again, we used a startup sequence of 5 NMPC iterations with $m=1$ to improve the initial guess of the control.

\begin{figure}[!ht]
	\begin{center}
		\includegraphics[width=0.48\textwidth]{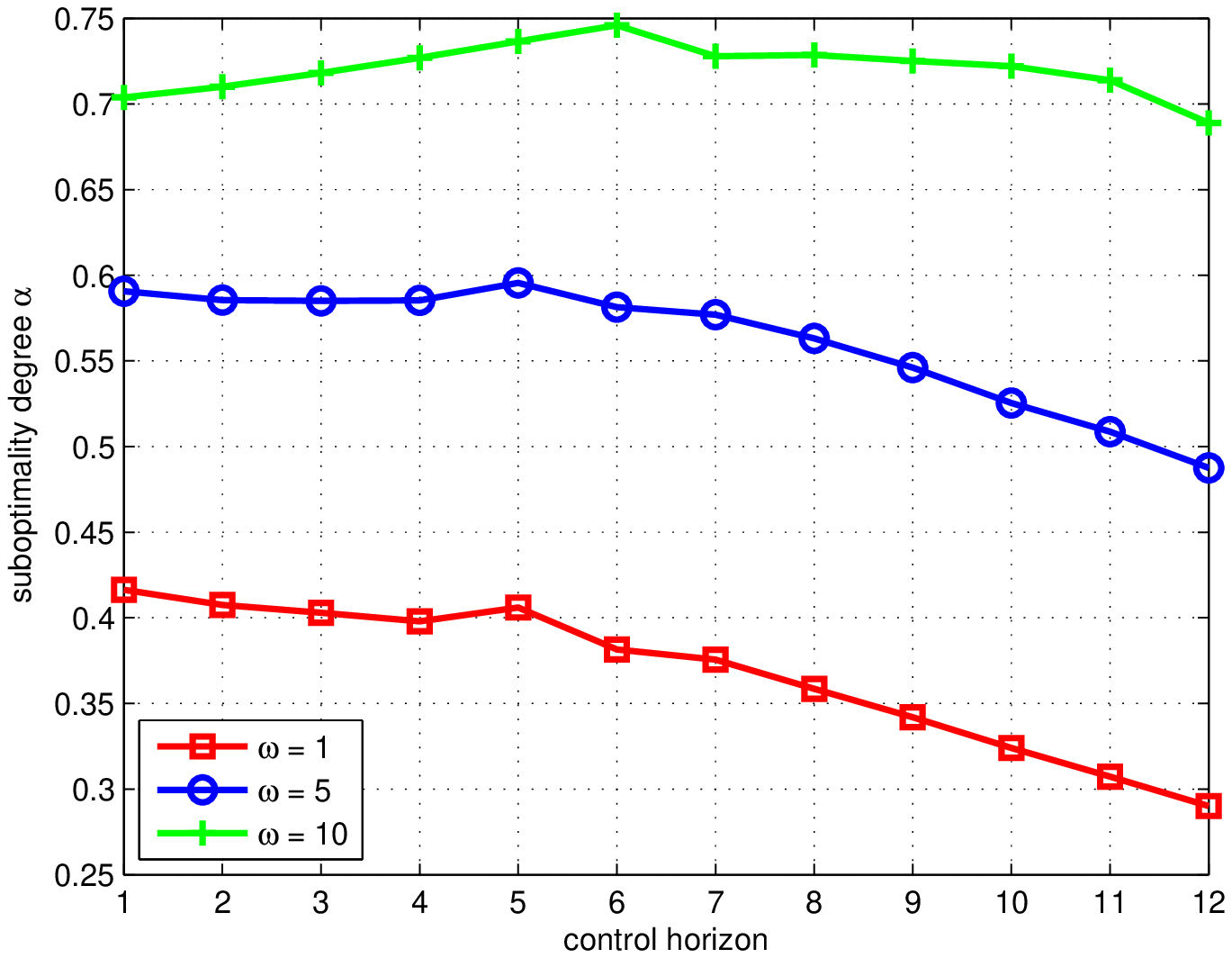}
		\hfill
		\includegraphics[width=0.48\textwidth]{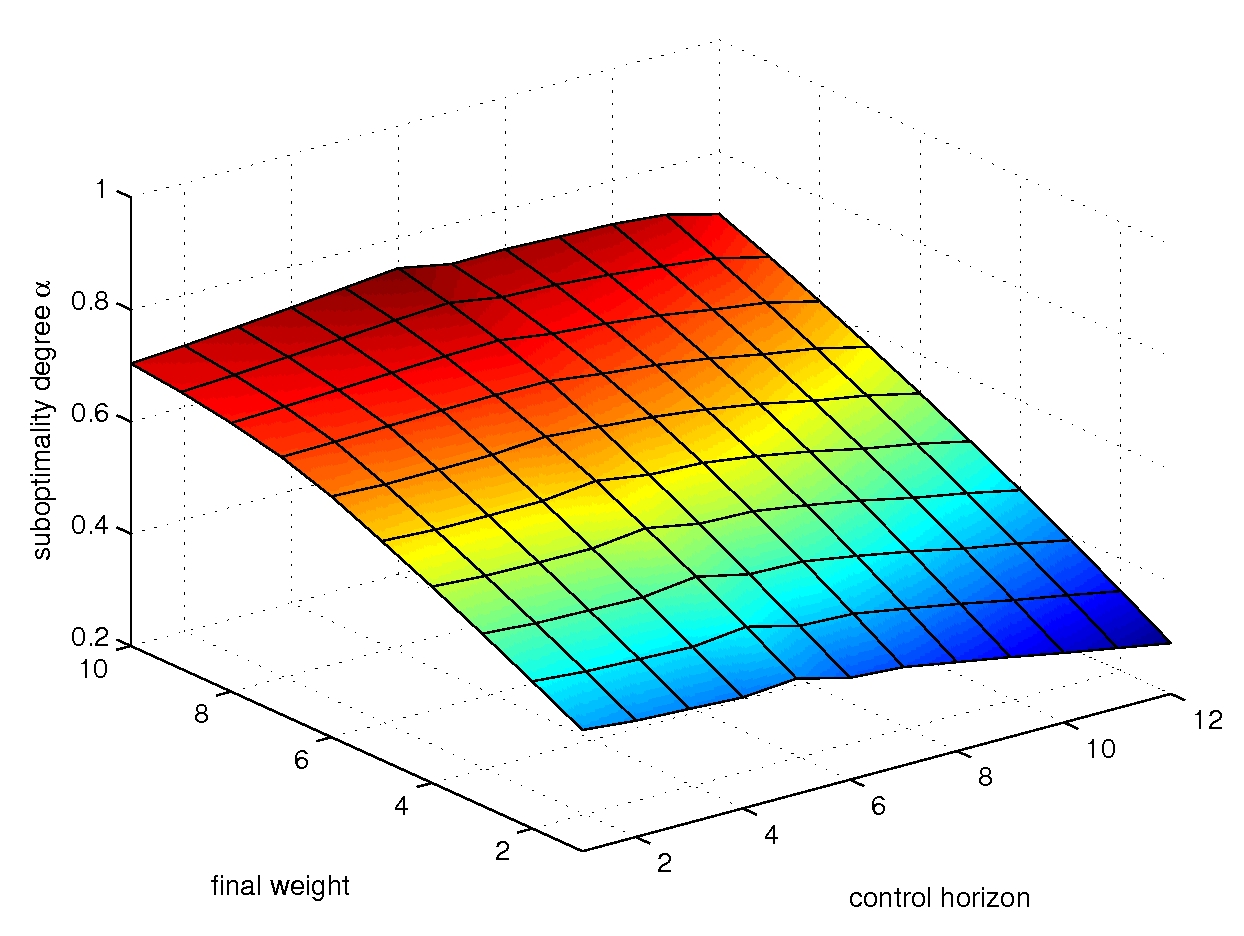}
	\end{center}
	\caption{Approximation of $\alpha_{13,m}^\omega$ for the arm--rotor--platform model.}
	\label{fig:arp}
\end{figure}

Similar to Fig.\ \ref{fig:nonlinear}, Fig.\ \ref{fig:arp} shows the minimal $\alpha$--value along a closed--loop trajectory for a variety of final weights $\omega$ and all possible control horizons $m = 1, \ldots, N$. Again, we obtain an improvement of our suboptimality estimate if the final weight $\omega$ is increased. Considering the control horizon $m$, however, we hardly experience any improvement for small $m$ and a decrease of $\alpha$ for $m$ being large.\\
Note that suboptimality estimates which are computed for every sampling instant inherit a different weighting of these instants if $m$ is changed. Yet, a fair comparison can be obtained if only those control horizon lengths are taken into account which are an integer factor of the largest one and use the estimate \eqref{2:prop:m-step suboptimality estimate:eq1}. In the left of Fig.\ \ref{fig:arp2} we display such a comparison for $m \in \{ 1, 2, 3, 4, 6, 12 \}$. The corresponding cost of the closed--loop costs $V^{\mu_{N,m}}$ considering the first  are shown on the right of Fig.\ \ref{fig:arp2}. It exhibits the expected decrease in $\omega$, but also the rise in $m$. Moreover, increasing the control horizon $m$ implies that the system remains in open--loop for a longer period of time which may be harmful even in terms of stability if modelling errors or external perturbations occur, cf. \cite{Magni}. A detailed quantitative analysis of these effects in our setting is currently under investigation.

\begin{figure}[!ht]
	\begin{center}
		\includegraphics[width=0.48\textwidth]{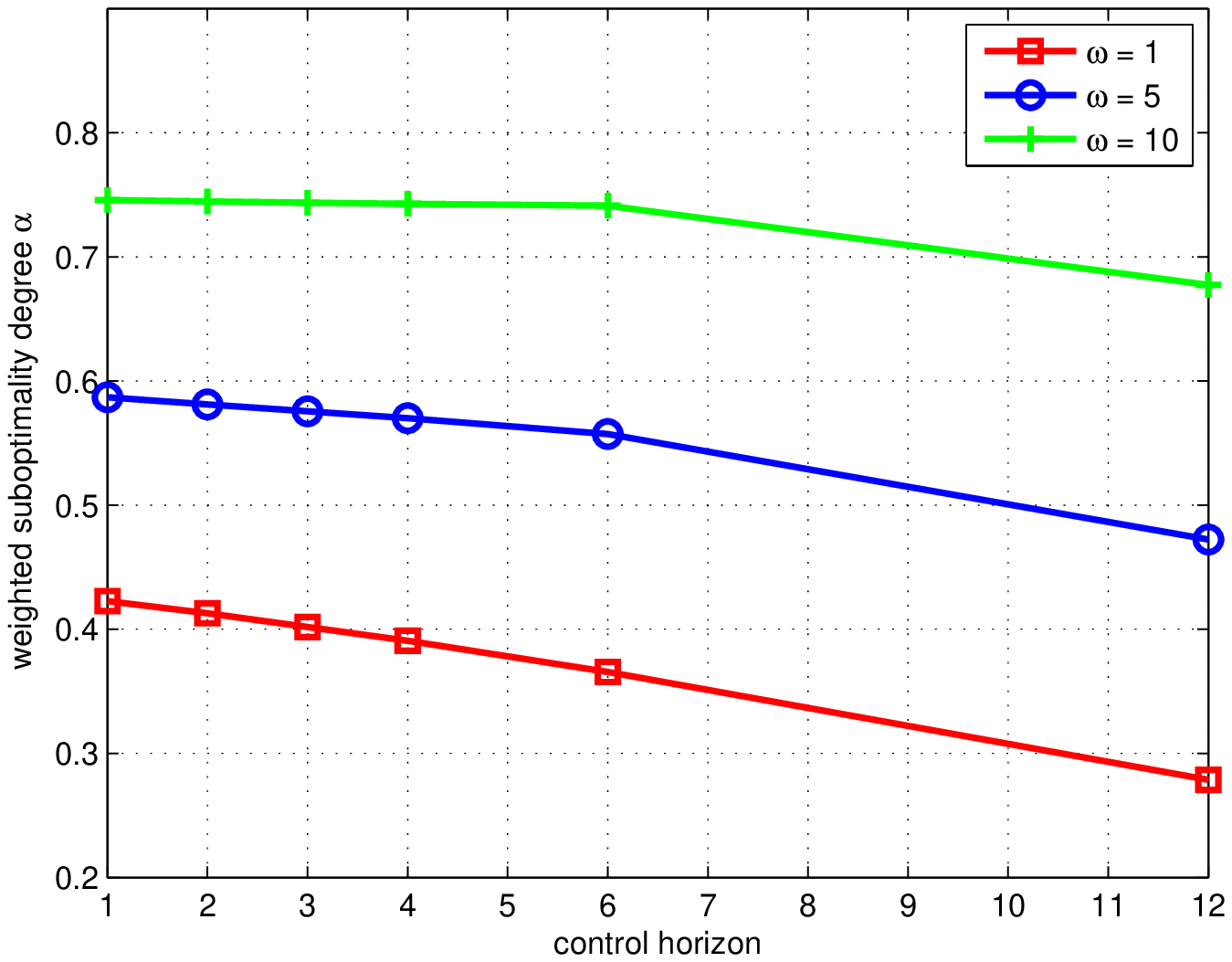}
		\hfill
		\includegraphics[width=0.48\textwidth]{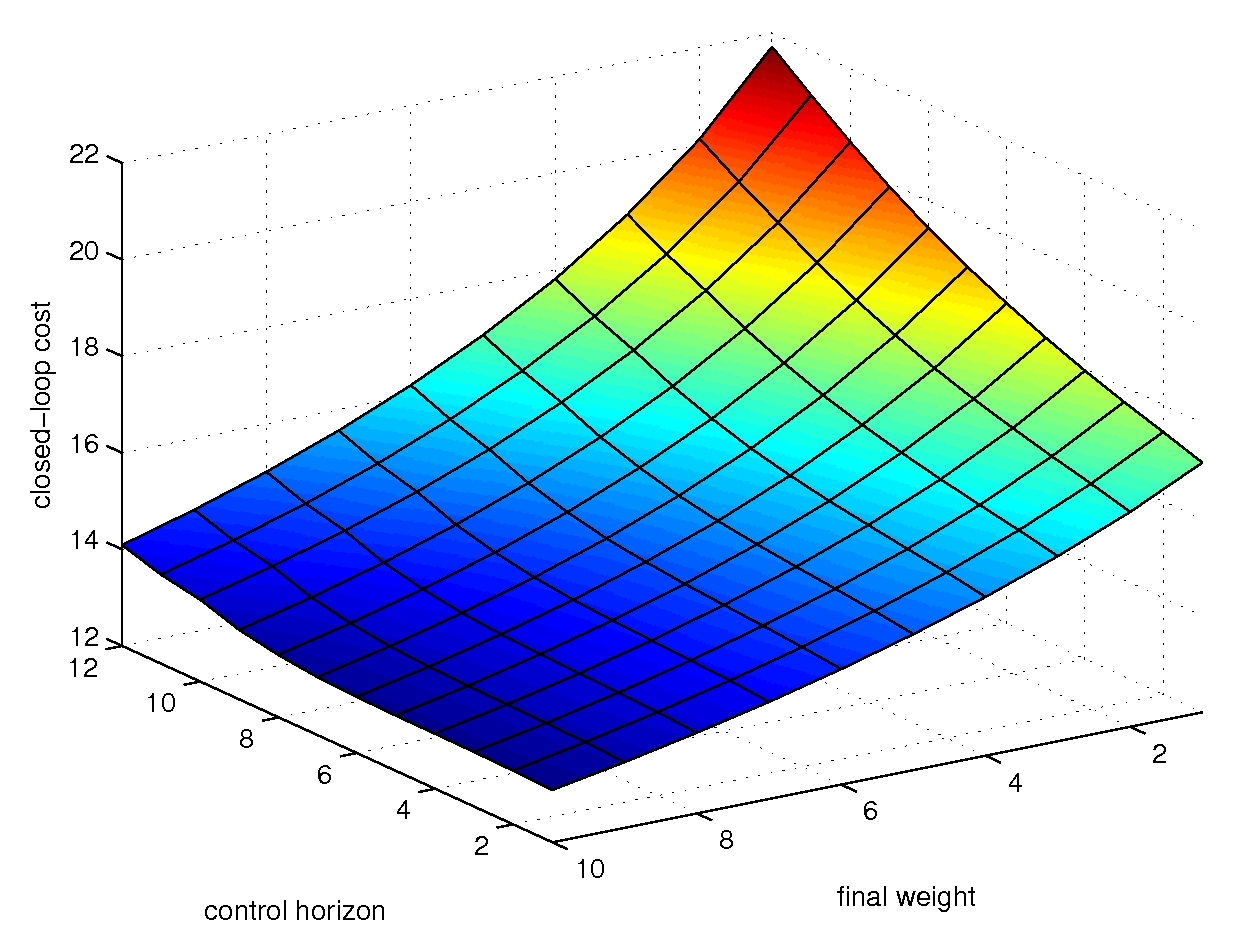}
	\end{center}
	\caption{Weighted approximation of $\alpha_{13,m}^\omega$ (left) and closed--loop cost (right) for the arm--rotor--platform model.}
	\label{fig:arp2}
\end{figure}

\section{Appendix}\label{appendix}

\subsection{Auxiliary results for Section \ref{alphasec}}
In this section we state and prove the technical Lemmata \ref{Appendix_Lemma_Inequalities}, \ref{appendix_technical_lemma:1}, \ref{appendix_lgs} and Corollary \ref{appendix_technical_corollary:1} which are used in order to derive formula \eqref{alpha_formula}.
\begin{lemma}\label{Appendix_Lemma_Inequalities}
	Let $N \in \mathbb{N}_{\geq 2}$, $m \in \{1,\ldots,N-1\}$, $\omega \geq 1$, and $\gamma_i$ be defined as in Proposition \ref{Alpha_Formula_LP_Lemma} satisfying \eqref{3:eq:submultiplicativity}. Then
	\begin{eqnarray*}
		& & (\gamma_{m+1}-\omega) \prod_{i=m+2}^{N-j+m-1} (\gamma_i - 1) \left[ \sum_{n=k}^{N-j+m+k-3} c_n + c_{N-j+m+k-2}\omega \right] \\
		& & \qquad - \left[ \sum_{n=N-j+k-2}^{N-j+m+k-3} c_n + c_{N-j+m+k-2}\omega - c_{N-j+k-2}\omega \right] \prod_{i=m+1}^{N-j+m-1} \gamma_i \geq 0 \qquad\forall\ k \in \mathbb{N}
	\end{eqnarray*}
	holds for $j=N-2,\ldots,m$.
\end{lemma}
\begin{proof}
	We carry out an induction with respect to $j$. The induction start, $j=N-2$, follows for arbitrary $k \in \mathbb{N}$ from
	\begin{eqnarray*}
		& & (\gamma_{m+1}-\omega) \left[ \sum_{n=k}^{m+k-1} c_n + c_{m+k}\omega \right] - \left[ \sum_{n=k}^{m+k-1} c_n + c_{m+k}\omega - c_k\omega \right] \gamma_{m+1} \\
		& = & c_k \gamma_{m+1} \omega - \omega \left[ \sum_{n=k}^{m+k-1} c_n + c_{m+k}\omega \right] = \omega \left[ \sum_{n=0}^{m-1} (c_k c_n - c_{n+k}) + (c_k c_m - c_{k+m}) \omega \right] \stackrel{\eqref{3:eq:submultiplicativity}}{\geq} 0.
	\end{eqnarray*}
	In order to carry out the induction step from $j+1 \rightsquigarrow j$ we rewrite the inequality in consideration for arbitrary but fixed $k \in \mathbb{N}$:
	\begin{eqnarray*}
		& & (\gamma_{m+1}-\omega) \prod_{i=m+2}^{N-j+m-2} (\gamma_i - 1) \left[ c_k \gamma_{N-j+m-1} - \sum_{n=k}^{N-j+m+k-3} c_n - c_{N-j+m+k-2}\omega \right] \\
		& & \qquad + \gamma_{N-j+m-1} \Bigg[ (\gamma_{m+1}-\omega) \prod_{i=m+2}^{N-j+m-2}(\gamma_i-1)\left[\sum_{n=k+1}^{N-j+m+k-3} c_n + c_{N-j+m+k-2}\omega\right] \\
		& & \qquad\quad \phantom{\gamma_{N-j+m-1} \Bigg[} - \left[\sum_{n=N-j+k-2}^{N-j+m+k-3} c_n + c_{N-j+m+k-2}\omega - c_{N-j+k-2}\omega\right] \prod_{i=m+1}^{N-j+m-2} \gamma_i\Bigg] \geq 0.
	\end{eqnarray*}
	The positivity of this expression which consists of two summand follows from \eqref{3:eq:submultiplicativity} and the induction assumption for $j+1$ and $k+1$.
\end{proof}

In order to prove Lemma \ref{appendix_lgs} and -- as a consequence -- Theorem \ref{alpha_formula_thm} we require the following technical assertions which are stated in Lemma \ref{appendix_technical_lemma:1} and Corollary \ref{appendix_technical_corollary:1}. Moreover, note that $d_i < 0, \linebreak[0] a_i, \linebreak[0] b_i > 0$.

\begin{lemma} \label{appendix_technical_lemma:1}
	Assume $\delta_{N-1-i} \in \mathbb{R}$. Then it holds for $i \in \mathbb{N}_{\geq 1}$
	\begin{equation} \label{appendix_technical_lemma:1_eq1}
		\prod_{j=1}^{i-1} (1+\delta_{N-1-j}) = \sum_{k=0}^{i-1} \left( \prod_{j=1}^{k-1} (1+\delta_{N-1-j}) \prod_{j=k+1}^{i-1} \delta_{N-1-j} \right).
	\end{equation}
\end{lemma}
\begin{proof} 
	We carry out an induction over $i$ to prove \eqref{appendix_technical_lemma:1_eq1}. 
	Since the correctness for $i=1,2$ is obvious, we proceed directly with the induction step
	\begin{eqnarray*}
		\prod_{j=1}^{i-1} (1+\delta_{N-1-j}) & = & \prod_{j=1}^{i-2} (1+\delta_{N-1-j}) + \delta_{N-1-(i-1)} \prod_{j=1}^{i-2} (1+\delta_{N-1-j}) \\
		& \stackrel{I.A.}{=} & \prod_{j=1}^{i-2} (1+\delta_{N-1-j}) + \delta_{N-1-(i-1)} \sum_{k=1}^{i-2} \left( \prod_{j=1}^{k-1} (1+\delta_{N-1-j}) \prod_{j=k+1}^{i-2} \delta_{N-1-j} \right) \\
		& = & \sum_{k=0}^{i-1} \left( \prod_{j=1}^{k-1} (1+\delta_{N-1-j}) \prod_{j=k+1}^{i-1} \delta_{N-1-j} \right).
	\end{eqnarray*}
\end{proof}
\begin{corollary}\label{appendix_technical_corollary:1}
	As a consequence of \eqref{appendix_technical_lemma:1_eq1}
	\begin{equation} \label{appendix_technical_corollary:1_eq1}
		\prod_{j=m+2}^N \gamma_j = \prod_{j=m+2}^N (\gamma_j - 1) + \sum_{k=m+2}^N \left( \prod_{j=m+2}^{k-1} \gamma_j \prod_{j=k+1}^N (\gamma_j - 1) \right).
	\end{equation}
	holds for $\gamma_i \in \mathbb{R}$. Moreover, one obtains for $j=N-1,\ldots,2$
	\begin{equation}\label{appendix_technical_corollary:1_eq2}
		(1+\delta)^{N-1-j} = \sum_{i=0}^{N-2-j} (1+\delta)^{N-2-j-i} \delta^i + \delta^{N-1-j}	
	\end{equation}
\end{corollary}
\begin{proof}
	To see the equivalence of \eqref{appendix_technical_lemma:1_eq1}, $i = N-m > 0$, and \eqref{appendix_technical_corollary:1_eq1}, we extract the summand for $k=0$ from the right hand side of \eqref{appendix_technical_lemma:1_eq1} and substitute $\delta_{N-1-j}$ by $\gamma_{m+1+j} - 1$. Then, shifts with respect to the considered control variables yield the assertion. \eqref{appendix_technical_corollary:1_eq2} is a direct consequence of the proof of \eqref{appendix_technical_lemma:1_eq1} which we performed for arbitrary natural numbers with $\delta_i=\delta$, $i=2,\ldots,N-1$.
\end{proof}
\begin{lemma}\label{appendix_lgs}
	Let $\gamma_{m+1}$ be strictly greater than $\omega$. Then the optimal solution $\lambda$ of Problem \ref{Alpha_Formula_Relaxed_Problem} satisfies $A \lambda = b$,\, $\lambda > 0$ componentwise.
\end{lemma}
\begin{proof}
	$\gamma_{m+1} > \omega$ implies a negative coefficient of $\lambda_{N-1}$ in the objective function. Thus, the goal of Problem \ref{Alpha_Formula_Relaxed_Problem}, whose optimum is denoted by $\lambda^* = (\lambda_1^*, \ldots, \lambda_{N-1}*)$, consists of maximizing $\lambda_{N-1}$. Suppose that there exists at least one index $k \in \{1,\ldots,N-1\}$ such that $\sum_{n=1}^{N-1} A_{kn} \lambda_n^* < \bar{b}_{k}$ and deduce a contradiction. Let $k$ be equal to one, define the constants $\varepsilon := \gamma_N-1 - \sum_{i=1}^{N-2} a_i \lambda_i^* - \omega \lambda_{N-1}^* > 0$, $\delta := - \max_{i=1,\ldots,N-2} d_i > 0$, $\beta := \max_{i=1,\ldots,N-2} b_i$, and choose $\tilde{\varepsilon} > 0$ such that
	\begin{equation*}
		\tilde{\varepsilon} \left[ \omega + \beta \sum_{i=1}^{N-2} a_i \frac {(1+\delta)^{N-2-i}}{\delta^{N-1-i}} \right] \leq \varepsilon.
	\end{equation*}
	Then, we increase $\lambda_{N-1}$ by $\tilde{\varepsilon}$ and $\lambda_i,\, i=1,\ldots,N-2$, by $\tilde{\varepsilon}\, \beta (1+\delta)^{N-2-i} / \delta^{N-1-i}$.
		The choice of $\tilde{\varepsilon}$ ensures the validity of inequality one. Since inequality $j \in \{2,\ldots,N-1\}$ holds for $\lambda^*$ the following computation shows that all other constraints are satisfied for the above choice of $\lambda_i,\, i=1,\ldots,N-2$, which leads to a contradiction to the assumed optimality of $\lambda^*$,
		\begin{eqnarray*}
			& & d_{j-1} \tilde{\varepsilon} \beta \frac {(1+\delta)^{N-1-j}}{\delta^{N-j}} + \sum_{i=j}^{N-2} \tilde{\varepsilon} \beta \frac {(1+\delta)^{N-2-i}}{\delta^{N-1-i}} + \tilde{\varepsilon} b_{j-1} \\
			& \leq & \tilde{\varepsilon} \left[ -\delta \beta \frac {(1+\delta)^{N-1-j}}{\delta^{N-j}} + \sum_{i=j}^{N-2} \beta \frac {(1+\delta)^{N-2-i}}{\delta^{N-1-i}} + \beta \right] \\
			& = & \frac {\tilde{\varepsilon} \beta} {\delta^{N-1-j}} \left[ - (1+\delta)^{N-1-j} + \sum_{i=0}^{N-2-j} (1+\delta)^{N-2-j-i} \delta^i + \delta^{N-1-j} \right] \stackrel{\eqref{appendix_technical_corollary:1_eq2}} {=} 0.
		\end{eqnarray*}
	Thus, the first inequality holds with equality and $k > 1$ which implies $\lambda_{k-1} > 0$. This enables us to reduce $\lambda_{k-1}$ without violating the non-negativity condition imposed on this variable. As a consequence, the first inequality is not active and all other inequalities remain valid. Hence, repeating the above argumentation w.r.t. $k=1$ proves $A \lambda^* = \bar{b}$.
	
	Since the above construction is feasible for $\lambda = 0$, $\lambda_{N-1}>0$ holds in the optimum. Hence, inequality $i+1$ implies $\lambda_{i}> 0$ for $i = 1,\ldots,N-2$ which completes the proof.
\end{proof}

\subsection{Proof of Lemma \ref{symmetry_omega_exp}}
\label{AppendixProofSymmetry}

The purpose of this subsection is to prove the symmetry properties stated in Lemma \ref{symmetry_omega_exp} for $\omega > 1$ (the case $\omega = 1$ is covered by Corollary \ref{alpha_symm_cor}). To this end, we require the following technical lemma which is also an essential tool in proving monotonicity properties.
\begin{lemma}\label{lemma_polynomial}
	Let $p:\mathbb{R} \rightarrow \mathbb{R}$ be a polynomial of degree $k > 1$ such that all $k$ roots $z_1,\ldots,z_k$ are real, exactly one of them is negative, and at most one is equal to zero. In addition, let the root of $p^{(k-1)}(z)$ be strictly smaller than $-\frac ck$ with $c \geq 0$ and let $p(\tilde{z}) = \tilde{z}^{k-1} (\tilde{z}+c)$ for some $\tilde{z} > \max\{z_1,\ldots,z_k\}$. Then it follows $p(z) > z^{k-1}(z+c)$ for all $z > \tilde{z}$.
\end{lemma}
\begin{proof} We prove the assertion via induction with respect to $k$. For $k=2$ the polynomial can be written as $p(z) = (z-a)(z+b) = z^2 + z(b-a) - ab$ with $b > 0$ and $a \geq 0$. Due to the stated assumptions, the root $(a-b)/2$ of the first derivative is strictly smaller than $- c/2$. Consequently, it holds $a + c < b$. From $p(\tilde{z}) = \tilde{z}^2 + c \tilde{z}$ we derive $\tilde{z}(b-a-c) - ab = 0$. Thus, for $z > \tilde{z}$ it holds $p(z) - z(z+c) = z (b-a-c) - ab > 0$ implying the assertion for $k=2$.
	We perform the induction step from $k$ to $k+1$. Let the degree of $p$ be $k+1$. Additionally, let $p$ satisfy all assumptions of Lemma \ref{lemma_polynomial} with $k+1$ instead of $k$. Note that this - under consideration of the mean value theorem - guarantees that all derivatives of $p$ have only strictly positive roots except for one negative root. Let $z_0:=\max\{z_1,\ldots,z_{k+1}\}$, i.e., $0<z_0<\tilde{z}$. Then we obtain $p(z_0)=0 < p(\tilde{z}) = \tilde{z}^{k}(\tilde{z}+c)$. Thus, there exists $\bar{z} \in ]z_0,\tilde{z}[: \frac 1 {k+1} p^\prime(\bar{z}) > \bar{z}^{k-1} (\bar{z} + \frac {kc}{k+1})$. Define $\tilde{p} := \frac 1 {k+1} p^\prime$. $\tilde{p}$ has degree $k>1$ and a maximal positive root $0 < z^* \leq z_0$. Thus, there exists $z^{**} \in ]z^*,\bar{z}[: \tilde{p}(z^{**}) = {z^{**}}^k$. Now we apply the induction assumption to $\tilde{p}$ and obtain $\tilde{p}(z) > z^{k-1}(z+\frac {kc}{k+1})$ for all $z>z^{**}$, i.e., $p^\prime(z) > (z^k (z + c))^\prime$ for $z > z^{**}$ which allows us to conclude the assertion.
\end{proof}

Let the $\mathcal{KL}_0$-function be of type \eqref{3:eq:exponential controllability}. Then Lemma \ref{symmetry_omega_exp} states that $\alpha^\omega_{N,N-m} - \alpha^\omega_{N,m} \geq 0$ for $m < N-m$, $m \in \mathbb{N}_{>0}$, $N \in \mathbb{N}_{\geq 3}$.
\begin{proof}
	Lemma \ref{suff_cond_exp} covers the case $\gamma_{m+1} - \omega \leq 0$. Hence, we assume $\gamma_{m+1} - \omega > 0$. Moreover, we suppose $\gamma_{N-m+1} - \omega > 0$ because otherwise $\alpha^\omega_{N,N-m} = 1$ holds and the assertion follows. Thus, we only have to deal with \eqref{alpha_formula} in order to establish the desired inequality. Choosing $N$ as small as possible for given $m \in \mathbb{N}_{\geq 1}$, i.e., $N = 2m + 1$ implies $N - m = m + 1$. Moreover, take the following equality into account
	\begin{equation}\label{symmetry_omega_exp_eq1}
		y_i = C \sum_{n=0}^{i-2} \sigma^n + \omega C \sigma^{i-1} = C \frac {1-\sigma^{i-1}+\omega \sigma^{i-1} - \omega \sigma^i} {1 - \sigma} = C \frac {1 - \eta \sigma^{i-1}} {1-\sigma}
	\end{equation}
	with $\eta := 1 + \sigma \omega - \omega$. Then, \eqref{symmetry_ineq} is equivalent to
	\begin{eqnarray}\label{symmetry_omega_exp_ineq1}
		0 & \leq & \left[ \gamma_{m+1} ( \gamma_{m+2} - \omega ) - ( \gamma_{m+1} - \omega ) ( \gamma_{m+1} - 1 ) \right] \prod_{i = m+2}^N (\gamma_{i}-1) + (\gamma_{m+1}-\gamma_{m+2}) \gamma_{m+1} \prod_{i = m+2}^N \gamma_i \nonumber \\
		\Longleftrightarrow 0 & \leq & \left[ \gamma_{m+1} (\gamma_{m+2}-\gamma_{m+1}) + (\gamma_{m+1}-\omega) \right] \prod_{i = m+2}^N \frac {\gamma_i - 1} {\gamma_i} - (\gamma_{m+2} - \gamma_{m+1}) \gamma_{m+1}.
	\end{eqnarray}
	If $\gamma_{m+1} - \gamma_{m+2} = - \sigma^m \eta \geq 0$, i.e., $\eta \leq 0$ this inequality holds because $\prod_{i=m+2}^N (\gamma_i - 1)/\gamma_i \in (0,1)$. Hence, we only have to deal with the case $\eta > 0$. We aim at using Lemma \ref{lemma_polynomial} to establish the inequality which has to be proven. For this purpose, we require the equalities
	\begin{equation*}
		\frac {\gamma_i - 1}{\gamma_i} \stackrel{\ref{symmetry_omega_exp_eq1}} = C^{-1} \left( C - \frac {1-\sigma} {1-\sigma^{i-1} \eta} \right) 
	\end{equation*}
	and
	\begin{eqnarray*}
		\frac {\gamma_{m+1}(\gamma_{m+2}-\gamma_{m+1})+(\gamma_{m+1}-\omega)} {\gamma_{m+1}(\gamma_{m+2}-\gamma_{m+1})} & = & C^{-2} \left( C^2 + \frac {C(\frac {1-\sigma^m \eta}{1-\sigma}) - \omega} {\sigma^m \eta \frac {1-\sigma^m \eta}{1-\sigma}} \right) \\
		& = & C^{-2} \left( C^2 + \frac C {\sigma^m \eta} - \frac {\omega (1-\sigma)} {\sigma^m \eta (1 - \sigma^m \eta)} \right) \\
		& = & C^{-2} \left( C + \frac 1 {2\sigma^m \eta} \pm \sqrt{ \left(\frac {1}{2\sigma^m \eta}\right)^2 + \frac {\omega (1-\omega)}{\sigma^m \eta (1-\sigma^m \eta)} }\right)
	\end{eqnarray*}
	Overall, we obtain for inequality \eqref{symmetry_omega_exp_ineq1}
	{\small\begin{equation*}
		\underbrace{\left( C + \frac {1}{2\sigma^m \eta} + \sqrt{ \ldots\phantom{\frac 12} } \right) \left( C + \frac 1 {2\sigma^m \eta} - \sqrt{ \ldots\phantom{\frac 12} } \right) \prod_{i=m+2}^N \left( C - \frac {1-\sigma} {1-\sigma^{i-1} \eta} \right)}_{=: p(C)} \geq C^2 C^{N-m-1} = \underbrace{C^{m+2}}_{=: q(C)}
	\end{equation*}}%
	where we have suppressed the argument of the roots. Since $\eta \in (0,1)$ the polynomial $p(C)$ has clearly $m+1$ strictly positive roots and exactly one negative root. We show that the positive root $-1/(2\sigma^m \eta) + \sqrt{ \ldots\phantom{\frac 12} }$ is located in the interval $(0,1)$, i.e.,
	\begin{equation*}
		\frac {\omega(1-\sigma)} {\sigma^m \eta (1-\sigma^m \eta)} < 1 + \frac {1}{\sigma^m \eta} \Longleftrightarrow 0 < \eta (1 - \sigma^m) + \sigma^m \eta (1-\sigma^m \eta).
	\end{equation*}
	Proposition \ref{exp_stab_C1} provides $p(1) = q(1)$, i.e., using the notation of Lemma \ref{lemma_polynomial} we set $\tilde{z} = 1$. Moreover, we calculate 
	the $(m+1)$-first derivative of $p(C)$:
	\begin{equation*}
		p^{(m+1)}(C) = (m+2)! C + (m+1)! \left( \frac {1}{\sigma^m \eta} - \sum_{i=m+2}^N \frac {1-\sigma}{1-\sigma^{i-1}\eta} \right).
	\end{equation*}
	This is a polynomial of degree one. To apply Lemma \ref{lemma_polynomial} it remains to show that its only root is negative. To determine the sign of its root it suffices to investigate the term
	\begin{eqnarray*}
		\frac {1}{\sigma^m \eta} - (1-\sigma) \sum_{i=m+2}^N \frac {1}{1-\sigma^{i-1}\eta} & > & \frac {1}{\sigma^m \eta} - (1-\sigma) \sum_{i=m+2}^N \frac {1}{1-\sigma^{m+1}\eta} \\
		& = & \frac {1 - \sigma^{m+1} \eta - m (1-\sigma) \sigma^m \eta}{\sigma^m \eta (1-\sigma^{m+1} \eta)} \\
		& > & \frac {(1-\sigma)}{\sigma^m (1-\sigma^{m+1} \eta)} \left( \sum_{i=0}^m \sigma^i - m \sigma^m \right) > 0.
	\end{eqnarray*}
	Hence, all assumptions of Lemma \ref{lemma_polynomial} are satisfied with $c=0$ which proves the assertion for $N = 2m+1$. This is the induction start. To complete the proof we have to perform the induction step, i.e., assuming the validity of the assertion for $N \geq 2m+1$ we have to show the correctness for $N+1$. Let \eqref{symmetry_ineq} hold. Using the induction assumption, we aim at proving
	{\begin{eqnarray}
		0 & \leq & \prod_{m+1}^{N+1-m} \gamma_i \prod_{N-m+2}^{N+1} (\gamma_i - 1) (\gamma_{N-m+2} - \omega) + (\gamma_{m+1} - \gamma_{N-m+2}) \prod_{m+1}^{N+1} \gamma_i \nonumber \\
		& & - (\gamma_{N+1} - 1)(\gamma_{m+1} - \gamma_{N-m+1}) \prod_{m+1}^N \gamma_i - (\gamma_{N+1}-1)(\gamma_{N-m+1}-\omega) \prod_{m+1}^{N-m} \prod_{N-m+1}^N (\gamma_i - 1) \nonumber \\
		\Longleftrightarrow 0 & \leq & \underbrace{ \left[ \gamma_{N+1} (\gamma_{m+1} - \gamma_{N-m+2}) - (\gamma_{N+1}-1)(\gamma_{m+1} - \gamma_{N-m+1}) \right]}_{= \gamma_{N+1}(\gamma_{N-m+1}-\gamma_{N-m+2}) + (\gamma_{m+1} - \gamma_{N-m+1}) } \prod_{N-m+1}^N \gamma_i \nonumber \\
		& & + \underbrace{\left[ \gamma_{N-m+1} (\gamma_{N-m+2} - \omega) - (\gamma_{N-m+1} - \omega) (\gamma_{N-m+1} - 1) \right]}_{= \gamma_{N-m+1} (\gamma_{N-m+2}-\gamma_{N-m+1}) + (\gamma_{N-m+1} - \omega) } \prod_{N-m+2}^{N+1} (\gamma_i - 1) \label{symmetry_omega_exp_ineq2} \\
 		\Longleftrightarrow 0 & \leq & - \left[ \gamma_{N-m+1} \gamma_{N+1} C \sigma^{N-m} \eta + \gamma_{N-m+1} C (\sigma^m - \sigma^{N-m}) \eta \right] \prod_{N-m+2}^N \gamma_i \nonumber \\
		& & + \left[ \gamma_{N-m+1} (\gamma_{N+1} - 1) C \sigma^{N-m} \eta + (\gamma_{N+1}-1)(\gamma_{N-m+1} - \omega) \right] \prod_{N-m+2}^N (\gamma_i - 1). \nonumber
	\end{eqnarray}}
	Thus, $\eta \leq 0$ guarantees the validity of the considered inequality. Hence, let $\eta > 0$. Note that $\eta = 1 - \omega(1-\sigma) < 1$. We rewrite \eqref{symmetry_omega_exp_ineq2} as follows
	{\footnotesize\begin{eqnarray*}
		& & - \left[ \frac {\gamma_{N+1}} {C} C \sigma^{N-m} \eta + (\sigma^m - \sigma^{N-m}) \eta \right] C C^m \prod_{i=N-m+1}^N \frac {\gamma_i}{C} + \left[ \gamma_{N-m+1} C \sigma^{N-m} \eta + (\gamma_{N-m+1} - \omega) \right] \prod_{N-m+2}^{N+1} (\gamma_{i} - 1) \geq 0 \\
		\Longleftrightarrow & & \underbrace{\left( \frac {\gamma_{N-m+1}} {\gamma_{N+1}} C^2 + \frac {\gamma_{N-m+1} - \omega}{\sigma^{N-m} \eta \gamma_{N+1}/C} \right) \prod_{i=N-m+1}^N \frac {\gamma_{i+1} - 1}{\gamma_i/C}}_{=: p(C)} \geq \underbrace{C^{m+1} \left(C + \frac {\sigma^m - \sigma^{N-m}}{\sigma^{N-m} \gamma_{N+1}/C} \right)}_{=:q(C)}.
	\end{eqnarray*}}
	Note that both polynomials have the coefficient one, i.e., are normed ($C^{m+2}$). Moreover, $q(C)$ has exactly one negative root located at $-(\sigma^m - \sigma^{N-m})(1-\sigma)/((1-\eta \sigma^{N})\sigma^{N-m})$. Next, we consider the term $(\gamma_{i+1}-1)/(\gamma_{i}/C)$ more closely. This leads to
	\begin{equation*}
		\frac {\gamma_{i+1}-1}{\gamma_i/C} = \frac {C \left( \frac {1-\sigma^i \eta}{1-\sigma} \right) - \frac {1-\sigma}{1-\sigma}} {\frac {1-\sigma^{i-1}\eta}{1-\sigma}} = \frac {1-\sigma^i \eta}{1-\sigma^{i-1}\eta} C - \frac {1-\sigma}{1-\sigma^{i-1}\eta}.
	\end{equation*}
	Thus, the polynomial $(\gamma_{i+1}-1)/(\gamma_{i}/C)$ has its root at $(1-\sigma)/(1-\sigma^i \eta)$, i.e., in the interval $(0,1)$. We aim at determining the roots of the term
	\begin{equation*}
		\frac {\gamma_{N-m+1}} {\gamma_{N+1}} C^2 + \frac {\gamma_{N-m+1}} {\gamma_{N+1}\sigma^{N-m}\eta} C - \frac \omega {\sigma^{N-m}\eta \gamma_{N+1}/C}.
	\end{equation*}
	To this end, we have to solve the equation
	\begin{equation*}
		C^2 + \frac 1 {\sigma^{N-m} \eta} C - \frac {\omega (1-\sigma)}{(1-\sigma^{N-m}\eta) \sigma^{N-m}\eta} \stackrel{!}{=} 0.
	\end{equation*}
	This leads to
	\begin{equation*}
		C = - \frac 1 {2 \sigma^{N-m} \eta} \pm \sqrt{ \left( \frac 1 {2 \sigma^{N-m} \eta} \right)^2 + \frac {\omega (1-\sigma)}{(1-\sigma^{N-m} \eta) \sigma^{N-m} \eta} } = \frac {-1 \pm \sqrt{ 1+\frac {4 \omega \sigma^{N-m} \eta (1-\sigma)}{1-\sigma^{N-m} \eta} }} {2 \sigma^{N-m} \eta}.
	\end{equation*}
	Thus, there is one negative and one positive root and $p(C)$ may be represented by $p(C) = \prod_{i=1}^{m+2} (C - z_i)$ where $z_i$ denote the determined roots. Show that the obtained positive root is also located in $(0,1)$, i.e.,
	\begin{eqnarray*}
		& & 1 + 4 \sigma^{N-m} \eta \omega(1-\sigma)/(1-\sigma^{N-m} \eta) < (1+2\sigma^{N-m} \eta)^2 = 1 + 4 \sigma^{N-m} \eta (1 + \sigma^{N-m} \eta) \\
		\Longleftrightarrow & & \omega (1-\sigma) < (1 + \sigma^{N-m} \eta)(1-\sigma^{N-m} \eta) = 1 - (\sigma^{N-m} \eta)^2 \\
		\Longleftrightarrow & & (\sigma^{2(N-m)} \eta) \eta < 1 - \omega (1-\sigma) = \eta.
	\end{eqnarray*}
	Calculate the $(m+1)^{st}$ derivative of $p(C)$ and $q(C)$
	{\scriptsize\begin{equation*}
		q^{(m+1)}(C) = (m+1)! \left( (m+2) C + \frac {(\sigma^m - \sigma^{N-m})(1-\sigma)} {(1-\sigma^{N} \eta) \sigma^{N-m}} \right), \quad p^{(m+1)}(C) = (m+1)! \left( (m+2) C + \frac 1 {\sigma^{N-m} \eta} - \sum_{i=N-m+1}^N \frac {1-\sigma}{1-\sigma^i \eta} \right).
	\end{equation*}}
	We show that the root of $p^{(m+1)}$ is strictly smaller than its counterpart of $q^{(m+1)}$, i.e., 
	\begin{equation*}
		\frac 1 {\sigma^{N-m} \eta} - \sum_{i=N-m+1}^N \frac {1-\sigma}{1-\sigma^i \eta} - \frac {(\sigma^m - \sigma^{N-m})(1-\sigma)}{(1-\sigma^N \eta) \sigma^{N-m}} > 0.
	\end{equation*}
	Thus, it suffices to show
	{\footnotesize\begin{eqnarray*}
		0 & \leq & \frac 1 {\sigma^{N-m}\eta} - \frac {m (1-\sigma)}{1-\sigma^{N-m+1}\eta} - \frac {(\sigma^m - \sigma^{N-m})(1-\sigma)}{(1-\sigma^N \eta) \sigma^{N-m}} \\
		\Longleftrightarrow 0 & \leq & (1-\sigma^{N} \eta) (1-\sigma^{N-m+1}\eta) - m (1-\sigma) \eta (1-\sigma^{N} \eta) \sigma^{N-m} - \eta (\sigma^m - \sigma^{N-m}) (1-\sigma) (1-\sigma^{N-m+1}\eta).
	\end{eqnarray*}}
	Using $(1-\sigma^N \eta) = (1-\sigma)(1-\sigma^N \eta)/(1-\sigma) = (1-\sigma) (\sum_{i=0}^{N-1} \sigma^i + \sigma^N \omega)$ provides
	{\scriptsize\begin{equation*}
		(1-\sigma^{N-m+1} \eta)(1-\eta)(\sigma^m - \sigma^{N-m}) + (1-\sigma^{N-m+1} \eta) \left( \sum_{i=0}^{m-1} \sigma^i + \sum_{i=m+1}^{N-1} \sigma^i + \sigma^{N-m} + \sigma^N \omega \right) - m \eta (1-\sigma^N \eta) \sigma^{N-m} \geq 0.
	\end{equation*}}
	Observe that the first summand is positive, $\eta < 1$ and $\sum_{i=0}^{m-1} \sigma^i > m \sigma^m$, and $\sum_{i=m+1}^{N-1} \sigma^i \geq (N-m-1) \sigma^{N-1} \geq m \sigma^{N-1}$. Hence, the following calculation establishes the desired inequality
	\begin{equation*}
		0 \leq (1-\sigma^{N-m+1} \eta) (\sigma^m + \sigma^{N-1}) - (1-\sigma^{N} \eta) \sigma^{N-m} \Longleftrightarrow 0 \leq \sigma^m - \sigma^{N-m} + \sigma^{N+1} - \sigma^{N+1} \eta.
	\end{equation*}
	In consideration of Proposition \ref{exp_stab_C1}, all assumptions of Lemma \ref{lemma_polynomial} are satisfied. This shows the assertion due to the definition of $p$ and $q$.
\end{proof}

\subsection{Proof of Lemma \ref{appendix_exp_monotonicity_induciton_start_lemma}}
\label{AppendixProofMonotonicity}

\begin{proof}
	Due to Lemma \ref{suff_cond_exp}, we restrict ourselves to the case $\gamma_{m+1} - \omega > 0$. Taking into accout $N-m = m+2$ we have $a = C \sigma^m \eta \omega (\gamma_{m+2}-1)$. Thus, \eqref{appendix_exp_monotonicity_inequality1} is equivalent to
	{\small\begin{equation*}
		(\gamma_{m+2} C \sigma^m \eta \omega (\gamma_{m+2}-1) + \omega (\gamma_{m+1}-\omega)(\gamma_{m+2}-1)) \prod_{i=m+3}^N (\gamma_i - 1) \geq C \sigma^m \eta \omega (\gamma_{m+2}-1) \gamma_{m+2} \prod_{i=m+3}^N \gamma_i.
	\end{equation*}}
	Using $C \sigma^m \eta \omega (\gamma_{m+2}-1) \gamma_{m+2} = C^2(\sigma^m \eta \omega (1-\sigma^{m+1} \eta)/(1-\sigma))$ yields
	\begin{equation*}
		\underbrace{\left( C^2 + \frac {1-\sigma^m \eta}{\sigma^m \eta (1-\sigma^{m+1} \eta)} C - \frac {\omega (1-\sigma)}{\sigma^m \eta (1-\sigma^{m+1} \eta)} \right) \prod_{i=m+3}^N \left( C - \frac {1-\sigma}{1-\sigma^i-1 \eta} \right)}_{=:p(C)} \geq \underbrace{C^{m+2}}_{=:q(C)}.
	\end{equation*}
	Clearly, the polynomial $\prod_{i=m+3}^N (C - (1-\sigma)/(1-\sigma^i-1 \eta)$ is of degree $m$ and decomposed in linear factors whose roots are located in the open interval $(0,1)$. The other factor of $p(C)$ can be represented as
	\begin{equation*}
		\left( C + \frac {1-\sigma^m \eta \pm \sqrt{ (1-\sigma^m \eta)^2 + 4 \omega (1-\sigma) \sigma^m \eta (1-\sigma^{m+1} \eta) }} {2\sigma^m \eta (1-\sigma^{m+1} \eta)} \right).
	\end{equation*}
	Hence, this polynomial has one strictly positive and one strictly negative root. We show that the positive root is located in $(0,1)$, i.e., is strictly less than one:
	\begin{eqnarray*}
		& & (1 - \sigma^m \eta + 2 \sigma^m \eta (1-\sigma^{m+1} \eta))^2 > (1-\sigma^m \eta)^2 + 4 \omega (1-\sigma) \sigma^m \eta (1-\sigma^{m+1} \eta) \\
		\Longleftrightarrow & & (2 \sigma^m \eta (1-\sigma^{m+1} \eta))^2 + 4 \eta^2 \sigma^m (1-\sigma^{m+1} \eta) (1-\sigma^m) > 0.
	\end{eqnarray*}
	We calculate the $(m+1)^{st}$-derivative of the polynomial $p(C)$ which has degree $(m+2)$
	\begin{equation*}
		p^{(m+1)}(C) = (m+2)! C + (m+1)! \left( \frac {1-\sigma^m \eta}{\eta \sigma^m (1-\sigma^{m+1}\eta)} - \sum_{i=m+3}^N \frac {1-\sigma}{1-\sigma^{i-1} \eta} \right).
	\end{equation*}
	We aim at proving that the root of $p^{(m+1)}(C)$ is strictly negative. To this purpose, it suffices to establish the following inequality
	\begin{equation*}
		\frac {1-\sigma^m \eta} {\eta \sigma^m (1-\sigma^{m+1} \eta)} - \sum_{i=m+3}^N \frac {1-\sigma}{1-\sigma^{i-1} \eta} > 0
	\end{equation*}
	which is in turn equivalent to
	\begin{equation*}
		\frac {1-\sigma^m \eta}{1-\sigma} = \sum_{i=0}^{m-1} \sigma^i + \omega > m \sigma^m > m \eta \sigma^m > \eta \sigma^m (1-\sigma^{m+1} \eta) \sum_{i=m+3}^N \frac 1 {1-\sigma^{i-1} \eta}.
	\end{equation*}
	Under consideration of Proposition \ref{exp_stab_C1} all assumption of Lemma \ref{lemma_polynomial} are satisfied. Thus, we conclude the assertion.
\end{proof}


\end{document}